\documentclass{article}

\usepackage{mathtools}
\usepackage{graphicx}
\usepackage{amsmath}
\usepackage{amssymb}
\usepackage{amsthm} 
\usepackage{tikz}
\usetikzlibrary{arrows.meta, positioning}
\usetikzlibrary{decorations.pathmorphing}

\usepackage{pxfonts}
\usepackage{enumerate}
\usepackage{color}
\usepackage{mathdots}
\usepackage{sectsty}
\usepackage[hidelinks]{hyperref}
\usepackage{tikz-cd}
\usepackage{framed}
\usepackage{caption}
\usepackage{adjustbox}
\usepackage{bbold}
\allowdisplaybreaks
\usepackage[nottoc]{tocbibind}

\expandafter\def\expandafter\normalsize\expandafter{%
  \normalsize  
  \setlength\abovedisplayskip{2ex}
  \setlength\belowdisplayskip{2ex}
  \setlength\abovedisplayshortskip{1ex}
  \setlength\belowdisplayshortskip{1ex}
}
\usepackage{titlesec}
\titlespacing*{\section}
{0pt}{2.5ex plus 1ex minus .2ex}{3.3ex plus .2ex}
\titlespacing*{\subsection}
{0pt}{2.5ex plus 1ex minus .2ex}{3.3ex plus .2ex}

\sectionfont{\scshape\centering\fontsize{11}{14}\selectfont}
\subsectionfont{\scshape\fontsize{11}{14}\selectfont}
\usepackage{fancyhdr}

\usepackage{tikz}
\usetikzlibrary{decorations.markings}

\newenvironment{keywords}{}{}
\newtheorem{thm}{Theorem}[section]
\newtheorem{cor}[thm]{Corollary}
\newtheorem{lem}[thm]{Lemma}
\newtheorem{defn}[thm]{Definition}
\newtheorem{rmk}[thm]{Remark}

\newtheorem{prop}[thm]{Proposition}

\newtheorem{problem}[thm]{Riemann-Hilbert Problem}

\newtheorem*{claim*}{Claim}

\newcommand{\subjclass}[1]{\par\noindent\textbf{Mathematics Subject Classification (2020):} #1}

\newcommand\longvdots[1]{\raisebox{1em}{\rotatebox{-90}{\hbox to #1 {\dotfill}}}}

\newcommand{\defeq}{\stackrel{\mathrm{def}}{=} }

\renewcommand{\d}{{\mathrm d}}

\newcommand{\be}{\begin{equation}}
\newcommand{\ee}{\end{equation}}
\newcommand{\bea}{\begin{eqnarray}}
\newcommand{\eea}{\end{eqnarray}}
\newcommand{\bes}{\begin{equation*}}
\newcommand{\ees}{\end{equation*}}
\newcommand{\beas}{\begin{eqnarray*}}
\newcommand{\eeas}{\end{eqnarray*}}

\usepackage{authblk}

\def\i{{\textnormal{i}}}
\def\e{{\textnormal{e}}}

\newcommand\shorttitle{Joint moments of orthogonal and symplectic type}
\newcommand\authors{T. Assiotis, T. Bothner, J. P. Keating, N. Simm and F. Wei}

\begin{document}

\title{ \sc Painlev\'e representations for joint moments of characteristic
polynomials from the orthogonal and unitary symplectic groups}

\author{\small THEODOROS ASSIOTIS, 
THOMAS BOTHNER, JONATHAN P. KEATING, NICHOLAS SIMM and FEI WEI}

\date{}
\maketitle

\begin{abstract}

We study joint moments of characteristic polynomials and their
second derivatives for Haar-distributed matrices from the unitary
symplectic and orthogonal groups. We establish their large-\(N\)
asymptotics throughout the full admissible range of the moment
parameters and identify the limiting coefficients probabilistically
in terms of the first inverse power sum of the Bessel point process.
This extends the corresponding results of Assiotis, Gunes, Keating,
and Wei [Proc. Lond. Math. Soc., 132 (2026), e70136] to the full
admissible parameter range. Throughout this full range, we also express these limits in terms of
a distinguished solution of the
\(\sigma\)-Painlev\'e~III\(^{\prime}\) equation. Our approach reduces
the problem to singular inverse-power statistics of the Jacobi
unitary ensemble and combines their finite-\(N\)
\(\sigma\)-Painlev\'e~V representation with a refined
Riemann--Hilbert analysis yielding precise small-\(x\) asymptotics
and uniform derivative estimates for the finite-\(N\) correction
term. A key ingredient in the probabilistic identification is a
kernel-based proof that the rescaled first inverse power sum of
the Jacobi ensemble converges in distribution to its Bessel
point process counterpart.

~\\
~\\

    
\end{abstract}

\tableofcontents

\providecommand{\keywords}[1]
{ \textbf{Keywords:} #1 }

\section{Introduction}
The field of random matrices has long-standing connections with Painlev\'e
equations and integrable systems. Jimbo et al. \cite{JMMS1980}
obtained a Painlev\'e~V representation of the sine-kernel Fredholm
determinant on a single interval. Tracy and Widom
\cite{TracyWidomAiry,TracyWidomBessel} related the Airy and Bessel
gap probabilities to Painlev\'e~II and Painlev\'e~III, respectively.
The GUE Tracy--Widom distribution also arises as a scaling limit
in integrable probability, as illustrated by the work of Borodin
and Corwin \cite{BorodinCorwinMacdonald};
see also \cite{CorwinKPZ}.
Forrester and Witte
\cite{ForresterWittePainleve1,forrester2002application}
used Okamoto's \(\tau\)-function theory to express various random matrix
averages in terms of Painlev\'e transcendents. See \cite{ItsAlexander} for further examples.
The Riemann--Hilbert method provides a framework for establishing
and analyzing such Painlev\'e connections. Precisely, the isomonodromic
deformation theory \cite{JM,FN} links Painlev\'e equations to
Riemann--Hilbert problems, while the Riemann--Hilbert
characterization of orthogonal polynomials by Fokas, Its, and
Kitaev \cite{FIK1992} connects this framework to Hankel determinants.
The Deift--Zhou nonlinear steepest descent method then enables
systematic asymptotic analysis \cite{DZ,Deift1999}.

In this paper, we use the
Riemann--Hilbert method to study Painlev\'e representations for joint
moments of characteristic polynomials and their derivatives for
Haar-distributed matrices from the orthogonal and unitary symplectic
groups. Specifically, we define
\begin{equation}\label{defofFN1}
F_N^{\mathrm{USp}}(s,h)
:=
\int_{\mathrm{USp}(2N)}
\lvert\varphi_A(0)\rvert^{s-h}
\lvert\varphi_A''(0)\rvert^h
\,\mathrm d\mu_{\mathrm{Haar}}(A)
\end{equation}
and
\begin{equation}\label{defofFN2}
F_N^{\mathrm{SO}}(s,h)
:=
\int_{\mathrm{SO}(2N)}
\lvert\varphi_A(0)\rvert^{s-h}
\lvert\varphi_A''(0)\rvert^h
\,\mathrm d\mu_{\mathrm{Haar}}(A),
\end{equation}
where \(\mu_{\mathrm{Haar}}\) denotes the normalized Haar measure
on the corresponding group.
Here, for \(A\in\mathrm{USp}(2N)\) (the unitary symplectic group)
or \(A\in\mathrm{SO}(2N)\) (the even orthogonal group),
\(\varphi_A\) denotes the characteristic polynomial evaluated on
the unit circle:
\[
\varphi_A(\theta)
\defeq
\det\!\left(\mathbf I-\e^{-\i\theta}A\right),
\qquad \theta\in\mathbb R.
\]
The primes in \eqref{defofFN1} and \eqref{defofFN2} denote
differentiation with respect to \(\theta\).

The study of such joint moments is motivated in part by the use
of random matrix theory to model the statistical behavior of
\(L\)-functions. This connection can be traced back to Montgomery's pair correlation
conjecture for the zeros of the Riemann zeta function
\cite{Montgomery1973} and Dyson's observation that the same
pair-correlation function arises for normalized eigenvalues of random
unitary matrices \cite{Dyson1962}. The work of Katz and Sarnak
\cite{KatzSarnak} developed a symmetry classification of
families of \(L\)-functions into unitary, unitary symplectic, and
orthogonal types, with each symmetry type corresponding to the
eigenvalue statistics of an appropriate classical compact group
and reflected in the statistical behavior of the low-lying zeros. This random matrix connection with families of \(L\)-functions was
further strengthened by the work of Keating and Snaith
\cite{keating2000randomzeta,keating2000random}, who used moments of
characteristic polynomials over the corresponding classical compact
groups to predict the leading-order asymptotics of central-value
moments in the corresponding \(L\)-function families. More refined conjectures, incorporating lower-order terms in the
moment asymptotics, were subsequently developed by Conrey,
Farmer, Keating, Rubinstein, and Snaith \cite{CFKRS1,CFKRS2}. The moments discussed above involve characteristic polynomials
themselves, without their derivatives. A natural direction for
further investigation is to consider joint moments involving both a
characteristic polynomial and its derivatives, corresponding on the
number-theoretic side to joint statistics of \(L\)-functions and
their derivatives. Such questions involving characteristic polynomials and their
low-order derivatives were introduced in the unitary setting by
Hughes \cite{Hughes2001}, while joint moments of the form \eqref{defofFN1} and
\eqref{defofFN2} for the unitary symplectic and orthogonal groups
were later studied by Altu\u{g},
Bettin, Petrow, Rishikesh, and Whitehead
\cite{altuug2014recursion}, and by Gunes
\cite{gunes2022characteristic}.

A further motivation for studying these joint moments comes from the
rich probabilistic, combinatorial, and integrable structures  that underlie various quantities
associated with characteristic polynomials. In the unitary setting, joint moments involving
derivatives are related to random variables arising from the ergodic
decomposition of Hua--Pickrell measures
\cite{assiotis2022joint, assiotis2021distinguished,AssiotisGunesKeatingWei2026,bothner2025asymptotic,bothner2026joint}. In the orthogonal and unitary symplectic
settings, analogous moment problems are connected with inverse power
sums of the Bessel point process
\cite{gunes2022characteristic,assiotis2025joint}. Related moment problems for
unitary matrices inside the unit disk can also be studied through
joint Gaussian limits of the logarithm of the characteristic
polynomial and its logarithmic derivative \cite{simm2026moments}.
Joint moment problems in the unitary setting also have connections
with symmetric functions, Kostka numbers, and integer matrix
enumeration; see, e.g.,
\cite{dehaye2008joint,barhoumi2020new,grover2026higher, akemann2026derivative,
snape2026combinatorial}. Alongside these probabilistic and combinatorial connections, representations in terms
of Painlev\'e functions reveal the integrable structure underlying
these moments and their asymptotic behavior, see, e.g., \cite{Forrester2006,Basor2019,Bailey2019,assiotis2022joint,assiotis2021distinguished, gunes2022characteristic, keating2024joint,alvarez2025moments,AssiotisGunesKeatingWei2026,forrester2026higher, assiotis2025joint, bothner2025asymptotic}.

In the unitary setting, the counterpart of the joint moments
\eqref{defofFN1} and \eqref{defofFN2} is
\begin{equation}\label{defofFN3}
F_N^{\mathrm{U}}(s,h)
:=
\int_{\mathrm{U}(N)}
\lvert\varphi_A(0)\rvert^{2s-2h}
\lvert\varphi_A'(0)\rvert^{2h}
\,\mathrm d\mu_{\mathrm{Haar}}(A).
\end{equation}
For \(A\in\mathrm{USp}(2N)\) or \(A\in\mathrm{SO}(2N)\), the
self-reciprocal symmetry of the characteristic polynomial implies
that its first derivative at the origin is determined by the
characteristic polynomial itself; hence the second derivative is the
first non-trivial derivative statistic in these settings. Painlev\'e representations for the joint moment (\ref{defofFN3}) have been
extensively studied. Forrester and Witte
\cite{Forrester2006} used the Hankel determinant expression for the
leading coefficient obtained by Conrey, Rubinstein, and Snaith
\cite{conrey2006moments} to establish its connection with a
\(\sigma\)-Painlev\'e~III\(^{\prime}\) function in the case
\(h=s\in\mathbb N\). This was subsequently extended by Basor et al. \cite{Basor2019} to general admissible integers $h$ and $s$ via an alternative Riemann–Hilbert analysis.  Their approach also gave an exact relationship (i.e. not just one for the leading order coefficient) between the joint moments and a \(\sigma\)-Painlev\'e~V function.  For the leading order coefficient, their results were obtained independently by Bailey et al. \cite{Bailey2019}. Later, Assiotis, Keating, and Warren \cite{assiotis2022joint} established
the existence of the limiting moments for real parameters in the full
admissible range and obtained a probabilistic representation, while
Assiotis, Bedert, Gunes, and Soor \cite{assiotis2021distinguished} provided a
\(\sigma\)-Painlev\'e~III\(^{\prime}\) representation for admissible integer
\(h\in \mathbb{N}\) and admissible real \(s\). Most recently, Bothner and Wei \cite{bothner2025asymptotic} obtained a
Painlev\'e representation for general admissible real exponents
\(h\) and real \(s\).

However, Painlev\'e representations for the corresponding joint
moments in the unitary symplectic and orthogonal settings,
\eqref{defofFN1} and \eqref{defofFN2}, have been much less studied.
Gunes \cite{gunes2022characteristic} obtained a Painlev\'e
representation for \eqref{defofFN1} in the unitary symplectic case,
while Assiotis, Gunes, Keating, and Wei \cite{assiotis2025joint} obtained
corresponding representations for both \eqref{defofFN1} and
\eqref{defofFN2} for integer \(h\) and real \(s\) in a restricted
admissible range.  To describe these restrictions more precisely and to facilitate
comparison with our main theorem, we recall the probabilistic
representation obtained in \cite{assiotis2025joint}.
The asymptotic behavior of
\eqref{defofFN1} and \eqref{defofFN2} is related to the first inverse
power sum of the Bessel point process on \(\mathbb R_{+}\). We
therefore recall the definition of this point process.
\begin{defn}\label{def:Bessel}Let $a>-1$. We define the Bessel point process on $\mathbb{R}_+$ with parameter $a$  to be the determinantal point process with correlation kernel $\mathsf{K}_a^{\mathrm{Bes}}$,
\begin{equation*}
    \mathsf{K}_a^{\mathrm{Bes}}(x,y) =\frac{\sqrt{x}J_{a+1}(\sqrt{x})J_{a}(\sqrt{y})-\sqrt{y}J_{a+1}(\sqrt{y})J_{a}(\sqrt{x})}{2(x-y)},
\end{equation*}
where $J_a$ denotes the Bessel function of the first kind with parameter $a$. It almost surely consists of a strictly increasing sequence\footnote{We note that the superscript $l$ in $\mathfrak{B}_a^l$ does not correspond to raising to the $l$-th power but rather denotes the $l$-th ordered point in the point process.} $0<\mathfrak{B}_a^1<\mathfrak{B}_a^2<\mathfrak{B}_a^3<\cdots$.
\end{defn}
\noindent Define the random variable $\mathfrak{e}(a)$ by
\begin{equation}\label{defofmathfrake2026102}
    \mathfrak{e}(a)\defeq\sum_{i\geq 1}\frac{1}{\mathfrak{B}_{a}^{i}},\ \ \ \ a>-1.
\end{equation}
This is the first inverse power sum of the Bessel point process and
is almost surely finite. Combining \cite[Theorem~1.3 and Proposition~2.5]
{assiotis2025joint}, one obtains the following probabilistic
representation in a restricted range of the parameters.

\begin{prop}[\cite{assiotis2025joint}]\label{limitcase}
Let \(\mathfrak e(\cdot)\) be the random variable defined in
\eqref{defofmathfrake2026102}. Let \( h \in \mathbb{R}_{>0} \). For \( s > -\tfrac{1}{2} \) and \( h < \lfloor s + \tfrac{3}{2} \rfloor' \), we have
\begin{align*}
\lim_{N\rightarrow \infty}\frac{F_{N}^{\mathrm{USp}}(s,h)}{N^{\frac{s(s+1)}{2}+2h}}=\frac{2^{\frac{s^2}{2}}G(1+s)\sqrt{\Gamma(1+s)}}{\sqrt{G(1+2s)\Gamma(1+2s)}}\mathbb{E}\left[\left(1+\frac{\mathfrak{e}(s+\frac{1}{2})}{2}\right)^{h}\right].
\end{align*}
In the case of the orthogonal group, for \( s > \tfrac{1}{2} \)  and \( h < \lfloor s + \tfrac{1}{2} \rfloor' \), we have
\begin{align*}
\lim_{N\rightarrow \infty}\frac{F_{N}^{\mathrm{SO}}(s,h)}{N^{\frac{s(s-1)}{2}+2h}}=\frac{2^{\frac{s^2}{2}}G(1+s)\sqrt{\Gamma(1+2s)}}{\sqrt{G(1+2s)\Gamma(1+s)}}\mathbb{E}\left[\left(1+\frac{\mathfrak{e}(s-\frac{1}{2})}{2}\right)^{h}\right].
\end{align*}
Here, \(G(z)\) denotes the Barnes \(G\)-function and for any positive real number $x$,  we define
\bea\label{definex}
\lfloor x \rfloor' =
\begin{cases}
\lfloor x \rfloor, & x \notin \mathbb{Z}, \\
x - 1, & x \in \mathbb{Z}.
\end{cases}
\eea
\end{prop}
Moreover, when \(h\in\mathbb N\) and \(s,h\) lie in the range
specified in Proposition~\ref{limitcase},
\cite[Theorem~1.6]{assiotis2025joint} gives a representation of
the corresponding limiting quantities in terms of a
Painlev\'e transcendent.  

\subsection{Main results}

The main purpose of the present paper is
therefore to settle the convergence problem for the joint moments
\eqref{defofFN1} and \eqref{defofFN2} throughout the full admissible
range of the parameters, and to establish a complete Painlev\'e
description of their limiting values over this entire range.

We now state our main results. More precisely,
Theorem~\ref{limitcase2} below extends the convergence and the
probabilistic limit representations in
Proposition~\ref{limitcase} to the full admissible range of the
moment parameters, while Theorem~\ref{maintheorem2} below expresses the
corresponding limiting quantities, equivalently the leading-order
coefficients of the joint moments, in terms of a distinguished
solution of the \(\sigma\)-form of the
Painlev\'e~III\(^{\prime}\) equation.

\begin{thm}\label{limitcase2}
Let \(\mathfrak e(\cdot)\) be the random variable defined in
\eqref{defofmathfrake2026102}. In the case of the unitary symplectic group,
for \( s > -\tfrac{3}{2} \) and  $h\in \mathbb{C}$ with \( \textnormal{Re}(h) <  s + \tfrac{3}{2} \), we have
\begin{align}\label{usp}
\lim_{N\rightarrow \infty}\frac{F_{N}^{\mathrm{USp}}(s,h)}{N^{\frac{s(s+1)}{2}+2h}}=\frac{2^{\frac{s^2}{2}}G(1+s)\sqrt{\Gamma(1+s)}}{\sqrt{G(1+2s)\Gamma(1+2s)}}\mathbb{E}\left[\Bigg(1+\frac{\mathfrak{e}(s+\frac{1}{2})}{2}\Bigg)^{h}\right].
\end{align}
In the case of the even orthogonal group, for \( s > -\tfrac{1}{2} \) and $h\in \mathbb{C}$ with \( \textnormal{Re}(h) < s + \tfrac{1}{2}  \), we have
\begin{align}\label{orthogonal}
\lim_{N\rightarrow \infty}\frac{F_{N}^{\mathrm{SO}}(s,h)}{N^{\frac{s(s-1)}{2}+2h}}=\frac{2^{\frac{s^2}{2}}G(1+s)\sqrt{\Gamma(1+2s)}}{\sqrt{G(1+2s)\Gamma(1+s)}}\mathbb{E}\left[\Bigg(1+\frac{\mathfrak{e}(s-\frac{1}{2})}{2}\Bigg)^{h}\right].
\end{align}
Here, \(F_{N}^{\mathrm{USp}}(s,h)\) and
\(F_{N}^{\mathrm{SO}}(s,h)\) are defined in
\eqref{defofFN1} and \eqref{defofFN2}, respectively, and
\(G(z)\) denotes the Barnes \(G\)-function.
\end{thm}

The parameter ranges in Theorem~\ref{limitcase2} are precisely the
full admissible ranges. Indeed, Lemma~\ref{explicitformulaforfiniteN} below
reduces the two cases to Jacobi parameters
\[
(a,b)=\left(s+\frac12,\frac12\right),
\qquad
(a,b)=\left(s-\frac12,-\frac12\right),
\]
and the natural conditions \(a,b>-1\) and
\(\operatorname{Re}(h)<a+1\) give exactly the ranges stated in
Theorem~\ref{limitcase2}.

\begin{rmk}
Using the functional relation
$
G(s+1)=G(s)\Gamma(s),
$
we may rewrite the square of the prefactor in (\ref{usp}) as
\[\frac{G(1+s)^2\Gamma(1+s)}
     {G(1+2s)\Gamma(1+2s)}
=\frac{G(1+s)G(2+s)}
     {G(2+2s)}.
\]
Hence, for \(s>-3/2\), the prefactor in (\ref{usp}) is naturally understood as the
positive square root
\[
2^{s^2/2}
\sqrt{\frac{G(1+s)G(2+s)}
     {G(2+2s)}},
\]
with the apparent singularities at \(s=-\frac12\) and \(s=-1\)
in the original expression removed.
\end{rmk}

\begin{thm}\label{maintheorem2}
Let \(\mathfrak e(\cdot)\) be the random variable defined in
\eqref{defofmathfrake2026102}. 
Let \( h \in \mathbb{C} \). Define
\begin{equation}\label{defofn}
n=
\begin{cases}
0, & \operatorname{Re}(h)<0,\\[4pt]
\min\{m\in\mathbb N:m>\operatorname{Re}(h)\},
& \operatorname{Re}(h)\geq0.
\end{cases}
\end{equation}
In the case of the unitary symplectic group, for \( \mathrm{Re}(h) < s + \frac{3}{2} \) with $s>-\frac{3}{2}$, we have
\begin{align*}
\mathbb{E}\left[\Bigg(1+\frac{\mathfrak{e}(s+\frac{1}{2})}{2}\Bigg)^{h}\right]
=\frac{(-1)^{n}}{\Gamma(n-h)}\displaystyle\int_{0}^{\infty}z^{n-h-1}\frac{\mathrm{d}^{n}}{\mathrm{d}z^{n}}\left(\exp\left[\int_0^{z}\sigma\left(t;s+\frac{1}{2}\right)\frac{\mathrm{d} t}{t}\right]\right)\mathrm{d}z<+\infty.
\end{align*}
In the case of the even orthogonal group, for $\textnormal{Re}(h)<s+ \tfrac{1}{2}$ with $s>-\frac{1}{2}$, we have
\begin{align*}
\mathbb{E}\left[\Bigg(1+\frac{\mathfrak{e}(s-\frac{1}{2})}{2}\Bigg)^{h}\right]=
\frac{(-1)^{n}}{\Gamma(n-h)}\displaystyle\int_{0}^{\infty}z^{n-h-1}\frac{\mathrm{d}^{n}}{\mathrm{d}z^{n}}\left(\exp\left[\int_0^{z}\sigma\left(t;s-\frac{1}{2}\right)\frac{\mathrm{d} t}{t}\right]\right)\mathrm{d}z<+\infty.
\end{align*}
Here, for $a>-1$, the function $\sigma(t)=\sigma(t; a)$ satisfies the $\sigma$-Painlev\'e \textnormal{III}$'$ equation,
\begin{equation}\label{e42}
\bigl(t\sigma''(t)\bigr)^2
+4\bigl(\sigma'(t)+1\bigr)^2
\bigl(t\sigma'(t)-\sigma(t)\bigr)
-\left(
a\sigma'(t)+a+\frac12
\right)^2
=0,\ \ \ \ \ \ t>0,\ \ a>-1,
\end{equation}
and it has the explicit form
\begin{equation}\label{definitionofsigma}
\sigma(t; a)
=-\frac{\i}{2}\,\left[C_1(t)\right]_{12}
+ \frac{1-4a^2}{16}-t,
\end{equation}
where $\left[C_1(t)\right]_{12}$ is the \((1,2)\)-entry of \(C_{1}(t)=C_{1}(t;a)\) appearing in the unique solution of RHP \ref{modelproblem}, see \eqref{2.2}.
\end{thm}
In Proposition~\ref{prop:sigma-small-x}, we provide the boundary
asymptotics of \(\sigma(t;a)\) as \(t\downarrow0\). We remark that 
for $h \in \mathbb{N}$, it follows from \cite{gunes2022characteristic} and \cite{assiotis2025joint} that
\begin{align*}
\mathbb{E}\left[\Bigg(1+\frac{\mathfrak{e}(s+\frac{1}{2})}{2}\Bigg)^{h}\right]=(-1)^h \displaystyle\frac{\mathrm{d}^{h}}{\mathrm{d}z^{h}}
\exp\left[
\int_{0}^{z} \sigma\left(t;s+\frac{1}{2}\right)\, \frac{\mathrm{d}t}{t}
\right]\bigg|_{z=0}
\end{align*}
provided that $h<s+\frac{3}{2}$ and $s>-\frac{3}{2}$, 
and 
\begin{align*}
\mathbb{E}\left[\Bigg(1+\frac{\mathfrak{e}(s-\frac{1}{2})}{2}\Bigg)^{h}\right]=
(-1)^h \displaystyle\frac{\mathrm{d}^{h}}{\mathrm{d}z^{h}}
\exp\left[
\int_{0}^{z} \sigma\left(t;s-\frac{1}{2}\right)\, \frac{\mathrm{d}t}{t}
\right]\bigg|_{z=0}
\end{align*}
provided that $h<s+\frac{1}{2}$ and $s>-\frac{1}{2}$. In fact, these identities follow directly from
Theorem~\ref{maintheorem2} upon taking \(n=h+1\), together with
\eqref{defofsigma2} and the asymptotic results in
Lemma~\ref{asympoticsforv(x)forsmallx}.

We remark that, for the odd orthogonal case \(O^{-}(2N)\), results
analogous to Theorems~\ref{limitcase2} and~\ref{maintheorem2} can
be obtained by almost identical methods. To avoid excessive
length, we omit a detailed discussion of this case.

\subsection{Proof strategies}\label{strategy}
We first introduce the Jacobi ensemble notation that will be used in
the proofs of our main results. For \(a,b>-1\), let
\[
\bigl(x_1^{(N)},\ldots,x_N^{(N)}\bigr)
\]
denote the ordered eigenvalues of the Jacobi unitary ensemble with
joint probability density proportional to
\[
\prod_{j=1}^N x_j^a(1-x_j)^b
\prod_{1\leq i<j\leq N}(x_j-x_i)^2,
\qquad
0<x_1<\cdots<x_N<1.
\]
We denote expectation with respect to this ensemble by
\(\mathbb E_N^{(a,b)}[\cdot]\). The starting point for the proofs of our main results is
\cite[Proposition~3.1]{assiotis2025joint}, which reduces the
joint moments over the unitary symplectic and even orthogonal groups
to moments of the singular Jacobi statistic 
\begin{equation}\label{defofXN}
\mathsf{X}_{N}^{a,b}:=1+\frac{1}{2N^2}\sum_{j=1}^N\frac{1}{x_j^{(N)}}.
\end{equation}
More precisely, we have the following finite-\(N\) identities.
\begin{lem}\label{explicitformulaforfiniteN}
Let \(F_{N}^{\mathrm{USp}}(\cdot,\cdot)\) and
\(F_{N}^{\mathrm{SO}}(\cdot,\cdot)\) be defined in
\eqref{defofFN1} and \eqref{defofFN2}, respectively. For \(s>-\tfrac32\) and \(h\in\mathbb C\) satisfying
$
\textnormal{Re}(h)<s+\frac32,
$
we have
\begin{align*}
F_{N}^{\mathrm{USp}}(s,h)
&=
\left(
\int_{\mathrm{USp}(2N)}
\lvert\varphi_A(0)\rvert^{s}
\,\mathrm d\mu_{\mathrm{Haar}}(A)
\right)N^{2h}
\mathbb{E}_{N}^{(s+\frac12,\frac12)}
\left[
\Big(
\mathsf{X}_{N}^{(s+\frac{1}{2},\frac{1}{2})}
\Big)^{h}
\right]<+\infty.
\end{align*}
For \(s>-\tfrac12\) and \(h\in\mathbb C\) satisfying
$
\textnormal{Re}(h)<s+\frac12,
$
we have
\begin{align*}
F_{N}^{\mathrm{SO}}(s,h)
&=
\left(
\int_{\mathrm{SO}(2N)}
\lvert\varphi_A(0)\rvert^{s}
\,\mathrm d\mu_{\mathrm{Haar}}(A)
\right)
N^{2h}\mathbb{E}_{N}^{(s-\frac12,-\frac12)}
\left[
\Big(
\mathsf{X}_{N}^{(s-\frac{1}{2},-\frac{1}{2})}
\Big)^{h}
\right]<+\infty.
\end{align*}
Here $\mathsf{X}_{N}^{(a,b)}$ is defined in (\ref{defofXN}).
\end{lem}

By Lemma~\ref{explicitformulaforfiniteN}, the joint moments
\eqref{defofFN1} and \eqref{defofFN2} can be written as the product
of the corresponding characteristic-polynomial moment and
\begin{equation}\label{thb1}
\mathbb E_N^{(a,b)}
\left[
(\mathsf X_N^{a,b})^h
\right],
\end{equation}
with the appropriate choices of \(a\) and \(b\).
Since
\begin{equation*}
	 x_j^{(N)}\in(0,1),
\end{equation*}
we have
$
\mathsf X_N^{a,b}\geq1$
almost surely. In particular, for \(h\in\mathbb C\), the power
\((\mathsf X_N^{a,b})^h\) is unambiguously defined using the real
logarithm. Equipped with Lemma~\ref{explicitformulaforfiniteN}, we then analyze \eqref{thb1} through the following representation in
terms of derivatives of the Laplace transform:
\begin{equation}\label{1005noon}
\mathbb{E}_{N}^{(a,b)}
\left[
(\mathsf{X}_N^{a,b})^h
\right]
=
\frac{(-1)^n}{\Gamma(n-h)}
\int_0^\infty
z^{n-h-1}
\frac{\mathrm d^n}{\mathrm d z^n}
\mathbb{E}_{N}^{(a,b)}
\left[
\e^{-z\mathsf{X}_N^{a,b}}
\right]
\,\mathrm dz,
\end{equation}
where \(n\) is chosen as in (\ref{defofn}). See Lemma~\ref{expressionformoments} for details on
\eqref{1005noon}. The above expression holds for any $h\in \mathbb{C}$ with $\textnormal{Re}(h)<a+1$.
Thus, the key object in our analysis is the Laplace transform
\begin{equation}\label{1005laplace}
\mathbb E_N^{(a,b)}
\left[
\e^{-z\mathsf X_N^{a,b}}
\right]
\end{equation}
and its derivatives. The main idea of our approach is to express this Laplace transform
in terms of a solution of the \(\sigma\)-Painlev\'e~V equation.
More precisely, the Laplace transform can be identified with the
normalized \(\tau\)-function associated with the corresponding
\(\sigma\)-Painlev\'e~V solution and hence admits an exponential
representation in terms of this solution. The finite-\(N\)
\(\sigma\)-Painlev\'e~V equation degenerates, as \(N\to\infty\),
to a \(\sigma\)-Painlev\'e~III\(^{\prime}\) equation.

To justify passing to the limit in the integral representation
\eqref{1005noon}, the main difficulty is to obtain sufficiently
precise and uniform estimates near \(z=0\) of the underlying Painlev\'e transcendent. For this purpose, we
use the Riemann--Hilbert method to derive a full small-\(z\)
asymptotic expansion of the associated Painlev\'e function, together
with uniform derivative estimates for the finite-\(N\) correction
term. The behavior for large \(z\) is considerably simpler: since
\(\mathsf X_N^{a,b}\geq1\) almost surely, the Laplace transform \eqref{1005laplace} and its derivatives
decay exponentially as \(z\to\infty\). Combining the estimates near
zero and infinity yields an integrable dominating function, allowing
us to apply the dominated convergence theorem. This proves the convergence of the joint moments throughout the full
admissible range of the parameters, as stated in
Theorem~\ref{limitcase2}, and identifies their limiting values,
equivalently their leading-order coefficients, in terms of a
particular solution of the
\(\sigma\)-Painlev\'e~III\(^{\prime}\) equation, as stated in
Theorem~\ref{maintheorem2}.

\subsection{Hankel determinant representation and Painlev\'e equations}

By \eqref{defofXN},
\[
\mathbb E_N^{(a,b)}
\left[
\e^{-z\mathsf X_N^{a,b}}
\right]
=
\e^{-z}
\mathbb E_N^{(a,b)}
\Bigg[
\e^{-\frac{z}{2N^2}
\sum_{j=1}^N\frac{1}{x_j^{(N)}}}
\Bigg].
\]
Thus, it suffices to study the Laplace transform of the first inverse
power sum of the Jacobi eigenvalues. For fixed \(N\), the same quantity admits the
following Hankel determinant representation,
\bea\label{Pollaczek}
\mathbb{E}_{N}^{(a,b)}
\Bigg[\mathrm{e}^{-z\sum_{j=1}^{N}\frac{1}{x_j^{(N)}}}\Bigg]
=
\frac{1}{\mathcal{Z}_{N}^{a,b}}
\det_{0\leq i,j\leq N-1}
\left(
\int_0^1 x^{i+j}\mathrm{e}^{-z/x}x^a(1-x)^b\,\mathrm{d}x
\right).
\eea
Here \(\mathcal Z_N^{a,b}\) denotes the normalizing
constant of the Jacobi unitary ensemble:
\bea\label{defofZ}
\mathcal{Z}_{N}^{a,b}=\frac{1}{N!}\prod_{j=1}^{N}\frac{\Gamma(a+j)\Gamma(b+j)\Gamma(1+j)}{\Gamma(a+b+N+j)}.
\eea
The Hankel determinant in~\eqref{Pollaczek} was studied
in~\cite{chen2010painleve,chen2019riemann} for $z\geq0$ and
$a,b>0$, where the weight \(w^{(a,b)}(x;z)\), defined by
\begin{equation}\label{defofw}
w^{(a,b)}(x;z)
:=
\e^{-z/x}x^{a}(1-x)^{b},
\end{equation}
is referred to as a Pollaczek--Jacobi-type weight. At $z=0$, it reduces to the classical shifted Jacobi
weight. For $z>0$, the factor $\mathrm{e}^{-z/x}$ causes the
weight to vanish faster than any algebraic power as
$x\downarrow0$, violating the Szeg\H{o} condition.
The terminology \textit{Pollaczek--Jacobi type} reflects this
strong endpoint decay, reminiscent of that of classical
Pollaczek weights.

Evaluating the moments in the right-hand side of \eqref{Pollaczek} in terms of confluent hypergeometric
functions and reversing the order of both rows and columns
gives the following representation for the rescaled Laplace
transform; see Chen and Dai~\cite[Eq.~(2.6) and the subsequent
moment formula]{chen2010painleve}:
\bea\label{1005night}
\mathbb{E}_{N}^{(a,b)}
\Bigg[\mathrm{e}^{-z
\sum_{j=1}^{N}\frac{1}{x_j^{(N)}}}\Bigg]
=
\mathrm{e}^{-Nz}
\frac{(\Gamma(b+1))^N}{\mathcal{Z}_{N}^{a,b}}
\det_{0\leq i,j\leq N-1}
\Bigg(
U\Big(b+1,-a-2N+2+i+j;z\Big)
\Bigg),
\eea
where \(U(a,b;z)\) denotes the confluent hypergeometric function of
the second kind.
Identity (\ref{1005night}) is also a special case of
Assiotis, Gunes, Keating and Wei~\cite[Proposition~4.1]{assiotis2025joint},
obtained by setting $n_2=\cdots=n_k=0$ in that proposition.

For fixed $N$, \eqref{1005night}, together with
the $\tau$-function theory of Forrester and
Witte~\cite[Section 3]{forrester2002application}, provides a
characterization of the Laplace transform in terms of the
Jimbo--Miwa--Okamoto $\sigma$-form of Painlev\'e~V.
That is, 
\bea\label{pathforrealz}
\mathbb{E}_{N}^{(a,b)}\left[ \mathrm{e}^{-z\sum_{j=1}^N \frac{1}{\mathsf{x}_j^{(N)}  } } \right] 
=\exp\left(\int_{0}^{z}\sigma_N(t)\frac{\mathrm{d}t}{t}\right),\ \ \ \ z\geq 0,
\eea
where
$\sigma_{N}(z)$ satisfies the following $\sigma$-Painlev\'e-V equation 
\begin{align}
\left( z\frac{\mathrm{d}^2 \sigma_N}{\mathrm{d}z^2} \right)^2
=& -4 \left( \frac{\mathrm{d}\sigma_N}{\mathrm{d}z} \right)^3 z + \big(a^2+2az+4bz+z^2+4\sigma_N\big) \left( \frac{\mathrm{d} \sigma_N}{\mathrm{d}z} \right)^2\nonumber \\
& +\big(-2(a+2b+z) \sigma_N +2N(a+z)(a+b+N) \big) \frac{\mathrm{d} \sigma_N}{\mathrm{d}z} \nonumber\\
& + \big(\sigma_N - N (a+b+N)\big)^2.\label{equationforfiniteN}
\end{align}
Chen and Dai~\cite{chen2010painleve} established the
Painlev\'e~V connection (\ref{pathforrealz}) using ladder operator techniques. Subsequently, Chen, Chen and Fan~\cite{chen2019riemann}
applied the Deift--Zhou steepest descent method to obtain
large-$N$ asymptotics for the determinant in \eqref{Pollaczek} and its associated
orthogonal polynomials, identifying a Painlev\'e~III'
description in a suitable double scaling limit.
We recall the following asymptotic results obtained from \cite[Theorem~7]{chen2019riemann}.

\begin{prop}\label{largezbehaviour}
Let \(a, b>0\), and define
\begin{equation}\label{36}
v(x)=v(x;a)
:=
\i\,[C_1(x)]_{12},\ \ \ \ x>0
\end{equation}
where \([C_1(x)]_{12}\) denotes the \((1,2)\)-entry of
\(C_1(x)=C_1(x;a)\) appearing in the unique solution of
RHP~\ref{modelproblem}; see \eqref{2.2}.
Then there exists \(z_0>0\) such that, uniformly for
\(z\in(0,z_0]\), as \(N\to\infty\),
\begin{equation}\label{361005}
\mathbb{E}_{N}^{(a,b)}
\left[
\e^{-\frac{z}{2N^2}
\sum_{j=1}^{N}\frac{1}{x_j^{(N)}}}
\right]
=
\exp\!\left(
\int_0^z
\frac{1-4a^2-8v(x)}{16x}\,\d x
\right)
\left[
1+\mathcal{O}_{a,b}(N^{-1})
\right].
\end{equation}
\end{prop}

\subsection{Finite-$N$ representation and uniform derivative estimates}\label{subsection4}

Proposition~\ref{largezbehaviour} identifies the leading
Painlev\'e contribution to the Laplace transform up to a relative
error of order \(N^{-1}\). For the subsequent analysis, however, we
need a more precise finite-\(N\) representation, together with
uniform control of derivatives of the subleading term in \eqref{361005} with respect
to the scaling variable. The following Proposition \ref{representation} provides such a
representation and gives uniform bounds for arbitrary fixed
derivatives of the finite-\(N\) subleading correction near the origin.

We note that, although the derivation of \eqref{361005} in
\cite{chen2019riemann} is stated under the stronger assumption $a,b>0$, the result remains valid throughout the natural Jacobi range
\(a>-1\) and \(b>-1\) for \(z>0\). Indeed, the factor
\(\e^{-z/x}\) provides exponential decay as \(x\downarrow0\), so
that no additional restriction \(a>0\) is required for the
integrability of the perturbed weight at the origin. The condition
\(a>-1\) is inherited from the underlying Jacobi ensemble, while
integrability at \(x=1\) requires \(b>-1\).

\begin{prop}\label{representation}
Let $a,b>-1$. Let $v(x)$ be given as in \eqref{36}. Then we have for $z\in (0,\infty)$,
\begin{align*}
\mathbb{E}_{N}^{(a,b)}\left[\mathrm{e}^{-\frac{z}{2N^2}\sum_{j=1}^{N}\frac{1}{x_{j}^{(N)}}}\right]
=\exp\left(\int_{0}^{z}\frac{1-4a^2-8v(x)}{16x}\d x\right)\exp\left(\int_{0}^{z}\frac{\int_{0}^{x}f_{N,a,b}(t)\d t}{x}\d x\right),
\end{align*}
where there is an \(N_0=N_0(a)\in\mathbb N\) such that, for every
\(N\geq N_0\), the function
$
x\longmapsto f_{N,a,b}(x)
$
is smooth on \((0,\infty)\). Moreover, for every fixed
\(j\in\mathbb Z_{\geq0}\), there exists a constant
\(x_0=x_0(a,j)\in(0,1)\) such that
\begin{equation}\label{eq:Rtilde-at-zero-derivative4}
\frac{\d^j}{\d x^j}f_{N,a,b}(x)
=
\mathcal O_{a,b,j}
\left(
\frac{\rho_{a,j+1}(x)}{N}
\left(1+\frac{1}{N}\right)
\right),
\end{equation}
uniformly for $x\in(0,x_0]$ and for all $N\geq N_0$.
Here, \(\rho_{a,j}(x)\) is defined in \eqref{definieofpho} for
\(j\in\mathbb Z_{\geq0}\).
\end{prop}

The proof of Proposition~\ref{representation} is given in
Section~\ref{proofs}. A key technical feature in the proof
is the uniform control of parameter derivatives as $x\downarrow0$. Indeed, after normalizing
the model solution in RHP \ref{modelproblem} at infinity as in \eqref{eq:def-widehat-Psi} and in \eqref{eq:def-R2} , we write
\[
\widehat{\Psi}(\xi,x)
=
\mathbf I+\frac{C_1(x)}{\xi}+\mathcal R_2(\xi,x).
\]
For each fixed $x>0$, the large-$\xi$ asymptotics of RHP \ref{modelproblem} only gives
\[
\mathcal R_2(\xi,x)=\mathcal O_x(|\xi|^{-2}),
\qquad |\xi|\to\infty,
\]
and such a fixed-$x$ estimate does not by itself justify
differentiation of the remainder with respect to $x$, nor does it
provide estimates that are uniform as $x\downarrow0$. To overcome this difficulty, we use the $x$-part of the Lax pair, cf. \cite[$(1.14)$]{xu2014critical},
\[
\frac{\partial}{\partial x}\widehat{\Psi}(\xi,x)
=
\frac{1}{\xi}C_1'(x)\widehat{\Psi}(\xi,x),
\]
which yields the exact differential equation
\[
\frac{\partial}{\partial x}\mathcal R_2(\xi,x)
=
\frac{1}{\xi}C_1'(x)\mathcal R_2(\xi,x)
+
\frac{1}{\xi^2}C_1'(x)C_1(x).
\]
This identity, together 
with uniform control up to the jump rays, allows us to prove inductively the estimates for $\mathcal R_2(\xi,x)$ and its derivatives which remain valid in the
small-$x$ regime. These estimates are then propagated through the
local matching conditions and the singular-integral formulation of
the small-norm Riemann--Hilbert problem, ultimately yielding the
uniform derivative bounds for the finite-$N$ correction
$f_{N,a,b}$ in \eqref{eq:Rtilde-at-zero-derivative4}; see
Section~\ref{esonerrormatrix} for details.

\subsection{Small-$x$ asymptotics and a refined local Riemann--Hilbert analysis}\label{subsection5}

In order to analyze the more delicate behavior of
\[
\mathbb{E}_{N}^{(a,b)}
\left[
\e^{-\frac{z}{2N^2}
\sum_{j=1}^{N}\frac{1}{x_{j}^{(N)}}}
\right]
\]
near \(z=0\), Proposition~\ref{representation} shows that we need
a precise description of the associated Painlev\'e function
\(v(z)\) near the origin. Since, by \eqref{36} and
\eqref{definitionofsigma},
\begin{equation}\label{defofsigma2}
\sigma(z;a)
=
\frac{1-4a^2-8v(z)}{16}-z,
\end{equation}
This question is equivalent to analyzing the behavior of
\(\sigma(z;a)\) as \(z\downarrow0\). For this purpose, the leading-order asymptotics are not sufficient,
and we require a more precise small-\(z\) description. In particular,
we derive an asymptotic expansion to arbitrary order, treating the
algebraic and logarithmic contributions separately according as
\(a\notin\mathbb Z_{\geq0}\) or \(a\in\mathbb Z_{\geq0}\).
The result is summarized in the next proposition.

\begin{prop}\label{smallasyanlysis}
Let \(a>-1\). Then there exists \(x_{0}=x_{0}(a)\in(0,1)\) such that for \(0<x\leq x_{0}<1\) and any integer \(\mathrm{M}\geq 0\), the following asymptotic expansions hold. For \(a\notin \mathbb{Z}_{\geq 0}\),
\begin{align*}
\sigma(x;a)
=
\sum_{k=1}^{\mathrm{M}+1}\mathsf{b}_{k}(a)x^k
+\sum_{\ell=1}^{\mathrm{M}+1}\sum_{k=\ell}^{\mathrm{M}+1}
\mathsf{b}_{k,\ell}(a)x^{k+\ell a}
+\mathcal{O}_{\mathrm{M},a}\left(
\max\Big\{
x^{(a+1)(\mathrm{M}+2)},
x^{\mathrm{M}+2}
\Big\}
\right),
\end{align*}
and for \(a\in\mathbb{Z}_{\geq 0}\),
\begin{align*}
\sigma(x;a)
=&\,\sum_{k=1}^{\mathrm{M}+1}\mathsf{b}_{k}(a)x^k
+\sum_{\ell=1}^{\mathrm{M}+1}\sum_{k=\ell}^{\mathrm{M}+1}
\mathsf{b}_{k,\ell}(a)(\ln x)^{\ell}x^{k+\ell a}\\
&\hspace{1.5cm}
+\mathcal{O}_{\mathrm{M},a}\left(
\max\Big\{
|\ln x|^{M+2}x^{a+\mathrm{M}+2},
x^{\mathrm{M}+2}
\Big\}
\right),
\end{align*}
where the coefficients \(\mathsf{b}_{j}(a)\) and \(\mathsf{b}_{j,\ell}(a)\) depend only on \(a\).
\end{prop}

The proof of Proposition~\ref{smallasyanlysis} is given in
Section~\ref{section6}.
A key methodological refinement in the proof is a more precise resolution of the
local structure in RHP \ref{modelproblem} at the origin. In
\cite[(2.3)]{xu2014critical}, the same local behavior is described only as
\begin{equation}\label{thb2}
Q(x)\bigl(\mathbf I+\mathcal O(\xi)\bigr),
\qquad \xi\to0.
\end{equation}
For our purposes, we refine \eqref{thb2} explicitly into an
analytic matrix factor and an upper triangular matrix factor whose
off-diagonal entry is given by the scalar Cauchy transform
$f(\xi;x,a)$ defined in \eqref{Cauchy transform}. See Remark~\ref{refinementexplaination} for a detailed explanation. This refinement is crucial for our higher-order analysis. It allows
the singular and non-analytic behavior at the origin to be absorbed
into the local parametrix \(\Theta_0\), defined in (\ref{mp3}), so that the associated ratio function in the small-norm
Riemann--Hilbert problem~\ref{SN3} is analytic at the origin, while
the remaining mismatch is confined to the matching circle, with jump
\[
J_R(\xi)
=
\Theta_0(\xi;x,a)\Theta(\xi;a)^{-1}
\]
defined in (\ref{defofJR}).
By expanding the jump matrix to arbitrary order and applying the
standard small-norm Riemann--Hilbert analysis recursively, we obtain
the corresponding asymptotic expansion of the ratio function, which
ultimately yields the full small-$x$ asymptotic expansion of
\(\sigma(x;a)\) stated in Proposition~\ref{smallasyanlysis}.  We remark that all the coefficients in the asymptotic expansions of
Proposition~\ref{smallasyanlysis} can, in principle, be determined
iteratively. Moreover, by computing the first few coefficients, we
obtain the following explicit boundary behavior of
\(\sigma(x;a)\) as \(x\downarrow0\). This is a direct consequence of
Lemma~\ref{asyzero}.
\begin{prop}\label{prop:sigma-small-x}
Let \(a>-1\), and let \(\sigma(x;a)\) be defined by
\eqref{definitionofsigma}. Then, as \(x\downarrow0\),
\begin{equation}\label{eq:sigma-small-x}
\sigma(x;a)
=
\begin{cases}
\displaystyle
-\frac{1+2a}{2a}\,x
+
\frac{\pi}{\sin(\pi a)}
\frac{x^{a+1}2^{-a-1}}
{\Gamma^2(1+a)\Gamma(2+a)}
+
\mathcal O\!\left(
\max\{x^2,x^{2a+2}\}
\right),
&
a\notin\mathbb Z_{\geq0},
\\[1.2em]
\displaystyle
-\frac{1+2a}{2a}\,x
+
\frac{(-1)^a x^{a+1}2^{-a-1}\ln x}
{\Gamma^2(1+a)\Gamma(2+a)}
+
\mathcal O(x^2),
&
a\in\mathbb Z_{\geq1},
\\[1.2em]
\displaystyle
\frac{x}{2}
\left[
\ln\!\left(\frac{x}{2}\right)
+3\gamma_{\mathrm E}-3
\right]
+
\mathcal O\!\left(x^2|\ln x|^2\right),
&
a=0,
\end{cases}
\end{equation}
where $\gamma_{\mathrm E}$ is the Euler-Mascheroni constant.
\end{prop}

These boundary asymptotics are used to derive the small-\(x\)
estimates for the Painlev\'e function and its derivatives in
Proposition~\ref{prop:q-t-derivative-estimates}, which in turn are
used in the proof of Proposition~\ref{representation}.

\subsection{Full-range moment convergence and probabilistic identification}
Using the representation of the Laplace transform \eqref{pathforrealz} in terms of the
corresponding Painlev\'e~V solution obtained in
Subsection~\ref{subsection4}, together with the small-\(x\)
asymptotic analysis of \(v(x)\) and its derivatives in
Subsection~\ref{subsection5} and the uniform estimates for the
derivatives of the finite-\(N\) correction term near the origin in
Proposition~\ref{representation}, we obtain the required small-\(z\)
asymptotic estimates for the Laplace transform (\ref{1005laplace}) and its derivatives;
see Remark~\ref{summarize} for details. Following the strategy
outlined in Subsection~\ref{strategy}, these estimates allow us to
establish the convergence of moments throughout the full
admissible range, as stated in Theorem~\ref{limitcase2}, and to
identify their limiting values in terms of a particular solution of
the \(\sigma\)-Painlev\'e~III\(^{\prime}\) equation, as stated in
Theorem~\ref{maintheorem2}.

To complete the probabilistic representation in
Theorem~\ref{limitcase2}, we further identify these limiting values
with moments of the corresponding Bessel point-process statistic. An important ingredient for this probabilistic identification is
the convergence in distribution of the rescaled inverse power sum
\begin{equation}\label{defofYN}
\mathsf Y_N^{a,b}
:=
\frac{1}{N^2}
\sum_{j=1}^{N}
\frac{1}{x_j^{(N)}}.
\end{equation}
For \(a>0\), the convergence
\[
\mathsf Y_N^{a,b}
\overset{\mathrm d}{\longrightarrow}
M(a,b),\ \ \ \ \textnormal{as}\ \ N\rightarrow\infty,
\]
to a non-negative random variable \(M(a,b)\) was established in
\cite{gunes2022characteristic}, where it was conjectured that
\(M(a,b)\) has the same distribution as
\[
\mathfrak e(a)
=
\sum_{i\geq1}\frac{1}{\mathfrak B_a^i}
\]
defined in (\ref{defofmathfrake2026102}).
This identification was subsequently established in
\cite[Proposition~2.5]{AssiotisGunesKeatingWei2026}.
In the present paper, we extend this convergence in distribution of $\mathsf{Y}_N^{a,b}$ to
the full range $a,b>-1$. A new difficulty arises when
\(-1<a\leq0\). Indeed, the hard-edge convergence of any fixed number of rescaled
Jacobi eigenvalues remains valid throughout this range; see, for example,
\cite{holcomb2012edge,li2022operator}. More precisely, for every fixed
\(\ell\in\mathbb N\),
\[
\left(
N^2x_1^{(N)},\ldots,N^2x_\ell^{(N)}
\right)
\overset{\mathrm d}{\longrightarrow}
\left(
\mathfrak B_a^1,\ldots,\mathfrak B_a^\ell
\right),
\qquad N\to\infty.
\]
However, this convergence alone does not control the full inverse
power sum. Moreover, the uniform
first-moment bound used in the case \(a>0\) is no longer available
when \(-1<a\leq0\). We overcome this difficulty by a truncation
argument, separating the contribution from eigenvalues in an
\(\mathcal{O}(N^{-2})\)-neighborhood of the hard edge from the remaining
tail and establishing a bound for the latter that is uniform in
\(N\). As a consequence, we obtain the following distributional convergence
result; the proof is given in Section~\ref{distconv}.

\begin{prop}\label{lem:YN-convergence-all-a}
Let \(a,b>-1\). Let \(\mathsf{Y}_{N}^{a,b}\) and \(\mathfrak e(a)\) be as in
\eqref{defofYN} and \eqref{defofmathfrake2026102}, respectively. Then
\[
\mathsf Y_N^{a,b}
\overset{\mathrm d}{\longrightarrow}
\mathfrak e(a),
\qquad N\to\infty.
\]
\end{prop}
Proposition~\ref{lem:YN-convergence-all-a} gives, as $N\rightarrow \infty$,
\[
\mathsf X_N^{a,b}
\overset{\mathrm d}{\longrightarrow}
1+\frac{\mathfrak e(a)}{2}\ \ \ \ \ \forall\,a,b>-1.
\]
Together with the uniform integrability furnished by the
small-\(z\) estimates for the Laplace transform and its derivatives,
the latter distributional convergence yields the probabilistic
identification
\begin{equation}\label{goal0928}
\mathbb{E}_{N}^{(a,b)}
\left[
\left(\mathsf{X}_N^{a,b}\right)^h
\right]
\longrightarrow
\mathbb{E}
\left[
\left(
1+\frac{\mathfrak e(a)}{2}
\right)^h
\right],
\qquad N\to\infty,
\end{equation}
for $a,b>-1$, and $h \in \mathbb{C}$ with
\(\operatorname{Re}(h)<a+1\).

\section*{Acknowledgments}

JPK acknowledges support from EPSRC grant EP/Z535990/1.

\medskip
\noindent\textbf{Open Access.}
For the purpose of Open Access, the authors have applied a
CC BY public copyright licence to any Author Accepted Manuscript
(AAM) version arising from this submission.

\medskip
\noindent\textbf{Use of AI.}
ChatGPT was used to assist with correcting grammatical errors,
identifying typographical errors, and improving the clarity and
presentation of the manuscript.

\section{A Refined Model Riemann--Hilbert Problem for the Small-Norm Analysis}\label{modelRHP}

In this section, we recall the model Riemann--Hilbert problem from
\cite[Section~2]{xu2014critical}, with a refinement of the local
condition at the origin. Unlike the formulation in
\cite[Equation~(2.3)]{xu2014critical}, we explicitly retain the
non-analytic term that ensures compatibility with the prescribed
jump condition on \(\Sigma_2\), as explained in Remark
\ref{refinementexplaination}. In particular, in condition~(d) below,
we make this term explicit through the function \(f(\xi;x,a)\),
whereas in the previous formulation \eqref{mp2} it was absorbed into
an \(\mathcal{O}(\xi)\)-term. This refined formulation will be useful in the
construction of the local approximation \(\Theta_0\), which in turn
enters the small-norm Riemann--Hilbert problem RHP~\ref{SN3}.

\begin{problem}\label{modelproblem}
Let $x\in(0,\infty)$ and $a>-1$. Find $\Psi(\xi)=\Psi(\xi;x,a)\in\mathbb{C}^{2\times 2}$ with the following properties:
\begin{enumerate}
\item[$(a)$] $\Psi(\xi)$ is analytic in $\mathbb{C} \setminus \bigcup_{j=1}^3 \Sigma_j$, where 
$\Sigma_1=\left\{\zeta\in\mathbb{C}:\arg\zeta=\frac{2\pi}{3}\right\},
\Sigma_2=\left\{\zeta\in\mathbb{C}:\arg\zeta=\pi\right\}$,
$
\Sigma_3=\left\{\zeta\in\mathbb{C}:\arg\zeta=-\frac{2\pi}{3}\right\}.
$
\item[$(b)$] $\Psi(\xi)$ satisfies the jump condition
\[
\Psi_+(\xi) = \Psi_-(\xi)
\begin{cases}
\begin{pmatrix}
1 & 0 \\
\mathrm{e}^{\pi \i a} & 1
\end{pmatrix}, & \xi \in \Sigma_1, \\[0.8em]

\begin{pmatrix}
0 & 1 \\
-1 & 0
\end{pmatrix}, & \xi \in \Sigma_2, \\[0.8em]

\begin{pmatrix}
1 & 0 \\
\mathrm{e}^{-\pi \i a} & 1
\end{pmatrix}, & \xi \in \Sigma_3.
\end{cases}
\]

\item[$(c)$] The asymptotic behavior of $\Psi(\xi)$ at infinity is
\begin{equation}\label{2.2}
\Psi(\xi, x;a)
=
\left[
\textbf{I} + \frac{C_1(x)}{\xi} + \mathcal{O}\big(\xi^{-2}\big)
\right]
\xi^{-\frac{1}{4}\sigma_3}
\frac{\textbf{I} + \i\sigma_1}{\sqrt{2}}
\mathrm{e}^{\xi^{\frac{1}{2}}\,\sigma_3},
\quad
\arg \xi \in (-\pi,\pi),\ \xi \to \infty,
\end{equation}
where $C_1(x)$ is a matrix independent of $\xi$;

\begin{figure}[ht]
\centering

\begin{tikzpicture}[scale=2]

\usetikzlibrary{decorations.markings}

\tikzset{
  contour/.style={
    thick,
    postaction={decorate},
    decoration={
      markings,
      mark=at position 0.6 with {\arrow{>}}
    }
  }
}

\draw[contour] ({cos(120)},{sin(120)}) -- (0,0)
node[pos=0.2,right] {$\Sigma_1$};
\draw[contour] (-1,0) -- (0,0)
node[pos=0.3,above] {$\Sigma_2$};
\draw[contour] ({cos(240)},{sin(240)}) -- (0,0)
node[pos=0.2,right] {$\Sigma_3$};

\draw[densely dotted] (0,0) -- (1.2,0);

\node at (0,0) [below right] {$0$};

\node at (0.60,0.35) {$\Omega_1$};
\node at (-0.45,0.35) {$\Omega_2$};
\node at (-0.45,-0.35) {$\Omega_3$};
\node at (0.60,-0.35) {$\Omega_4$};
\end{tikzpicture}
\caption{Contour configuration and sector decomposition for the model RH problem for $\Psi$ in the $\xi$-plane. Both $\Omega_2$ and $\Omega_3$ have opening angle $\pi/3$.}\label{fig1}
\end{figure}
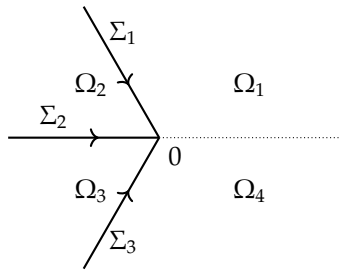
\item[$(d)$] The behavior of $\Psi(\xi)$ at $\xi = 0$ is of the form
\[
\Psi(\xi; x,a)
=\widetilde{\Psi}(\xi;x,a)\begin{pmatrix}1 & f(\xi;x,a)\\ 0 & 1\end{pmatrix}
\mathrm{e}^{\frac{x}{\xi}\sigma_3}
\xi^{\frac{a}{2}\sigma_3}
\begin{cases}
\textbf{I}, & \xi \in \Omega_1 \cup \Omega_4, \\[0.6em]
\begin{pmatrix}
1 & 0 \\
- \mathrm{e}^{\pi \i a} & 1
\end{pmatrix}, & \xi \in \Omega_2, \\[0.8em]
\begin{pmatrix}
1 & 0 \\
\mathrm{e}^{-\pi \i a} & 1
\end{pmatrix}, & \xi \in \Omega_3,
\end{cases}
\]
for $\arg \xi \in (-\pi,\pi)$ as $\xi \to 0$, where $\Omega_1$-$\Omega_4$ are illustrated in Figure~\ref{fig1}. Here, $\xi\mapsto\widetilde{\Psi}(\xi;x,a)$ is analytic at $\xi=0$ with $\det\widetilde{\Psi}(\xi;x,a)=1$, $f(\xi;x,a)$ denotes the function
\begin{equation}\label{Cauchy transform}
	f(\xi;x,a):=\frac{\e^{-\xi}}{2\pi\i}\int_{-\infty}^0(-u)^a\e^{u+\frac{2x}{u}}\frac{\d u}{u-\xi},\ \ \ \xi\in\mathbb{C}\setminus(-\infty,0],
\end{equation}
and $\sigma_j$ are the Pauli matrices given by
\[
\sigma_1 =
\begin{pmatrix}
0 & 1 \\
1 & 0
\end{pmatrix}, \quad
\sigma_2 =
\begin{pmatrix}
0 & -\mathrm{i} \\
\mathrm{i} & 0
\end{pmatrix}, \quad
\sigma_3 =
\begin{pmatrix}
1 & 0 \\
0 & -1
\end{pmatrix}.
\]
\end{enumerate}
\end{problem}

\begin{rmk}\label{refinementexplaination}
RHP \ref{modelproblem} is slightly different from the corresponding RHP in \cite{xu2014critical} in that its constraints (d) and (b) are consistent with one another - which is not the case for what is written in \cite[page $1266$]{xu2014critical}. Indeed, we have
\begin{equation*}
	f_+(\xi;x,a)-f_-(\xi;x,a)=(-\xi)^a\e^{\frac{2x}{\xi}},\ \ \ \xi<0,
\end{equation*}
and so the right-hand side in constraint (d) of RHP \ref{modelproblem} for $\Psi(\xi;x,a)$ produces exactly the jumps near $\xi=0$ as required by condition (b) in RHP \ref{modelproblem}. What's more, as $\xi\rightarrow 0$ and $\xi\notin \Sigma_2$, for fixed $x>0$ and $a>-1$,
\begin{align}
	f(\xi;x,a)=&\,\frac{\i}{2\pi}\e^{-\xi}\int_0^{\infty}v^a\e^{-v-\frac{2x}{v}}\frac{\d v}{v+\xi}\nonumber\\
	=&\,\frac{\i}{2\pi}\e^{-2\xi}\int_0^1v^a\e^{-\frac{2x}{v}}\frac{\d v}{v+\xi}+(\textnormal{analytic at}\ \xi=0)\nonumber\\
	=&\,\frac{\i}{2\pi}\e^{-2\xi}(2x)^a\int_{2x}^{\infty}t^{-a}\e^{-t}\,\bigg(\frac{1}{t}-\frac{1}{t+\frac{2x}{\xi}}\bigg)\,\d t+(\textnormal{analytic at}\ \xi=0)\nonumber\\
	=&\,-\frac{\i}{2\pi}\e^{-2\xi}(2x)^{a-1}\xi\Big(\Gamma(1-a,2x)+\mathcal{O}(\xi)\Big)+(\textnormal{analytic at}\ \xi=0),\label{mp1}
\end{align}
in terms of the incomplete Gamma function $\Gamma(a,z)$. It is important to notice that the first summand in \eqref{mp1} is not analytic at $\xi=0$. Still, using \eqref{mp1} in constraint (d), we have at once, as $\xi\rightarrow 0$ and $\xi\notin\Sigma_1\cup\Sigma_2\cup\Sigma_3$, for fixed $x>0$ and $a>-1$,
\begin{equation}\label{mp2}
	\Psi(\xi;x,a)=\check{\Psi}(\xi;x,a)\big(\textbf{I}+\mathcal{O}(\xi)\big)\mathrm{e}^{\frac{x}{\xi}\sigma_3}
\xi^{\frac{a}{2}\sigma_3}
\begin{cases}
\textbf{I}, & \xi \in \Omega_1 \cup \Omega_4, \\[0.6em]
\begin{pmatrix}
1 & 0 \\
- \mathrm{e}^{\pi \i a} & 1
\end{pmatrix}, & \xi \in \Omega_2, \\[0.8em]
\begin{pmatrix}
1 & 0 \\
\mathrm{e}^{-\pi \i a} & 1
\end{pmatrix}, & \xi \in \Omega_3,
\end{cases},
\end{equation}
with $\xi\mapsto\check{\Psi}(\xi;x,a)$ analytic at $\xi=0$, but $\mathcal{O}(\xi)$ not analytic at $\xi=0$. Constraint \eqref{mp2} is constraint \cite[$(2.3)$]{xu2014critical} with $Q(x)\equiv\check{\Psi}(0;x,a)$. The reader will now realize why constraint (d) in RHP \ref{modelproblem} is more precise than \eqref{mp2}. Still, constraint (d) in RHP \ref{modelproblem} implies \eqref{mp2}, so the workings in \cite[Sections $2.1-2.3$]{xu2014critical} apply verbatim to our RHP \ref{modelproblem}.
\end{rmk}
Constraints (a)-(c) in RHP \ref{modelproblem} are in fact satisfied by the following model function, which constitutes a certain $x\downarrow 0$ limiting version of the solution of RHP \ref{modelproblem}, compare \cite[Section $5.1$]{xu2014critical}. Precisely, if $I_a(\xi)$ and $K_a(\xi)$ denote the modified Bessel functions of index $a>-1$, defined with their principal branch, and $I'_a(\xi),K'_a(\xi)$ are their derivatives, then
\begin{equation}\label{mp2a}
	\Theta(\xi;a):=\begin{pmatrix}1 & 0\\ \frac{i}{8}(4a^2+3)&1\end{pmatrix}\pi^{\frac{\sigma_3}{2}}\begin{pmatrix}I_a(\xi^{\frac{1}{2}}) & \frac{\i}{\pi}K_a(\xi^{\frac{1}{2}})\\ \i\pi\xi^{\frac{1}{2}}I_a'(\xi^{\frac{1}{2}}) & -\xi^{\frac{1}{2}}K_a'(\xi^{\frac{1}{2}})\end{pmatrix}\begin{cases}
\textbf{I}, & \textnormal{arg}\,\xi\in(-\frac{2\pi}{3},\frac{2\pi}{3}), \\[0.6em]
\begin{pmatrix}
1 & 0 \\
- \mathrm{e}^{\pi \i a} & 1
\end{pmatrix}, & \textnormal{arg}\,\xi \in(\frac{2\pi}{3},\pi), \\[0.8em]
\begin{pmatrix}
1 & 0 \\
\mathrm{e}^{-\pi \i a} & 1
\end{pmatrix}, & \textnormal{arg}\,\xi \in(-\pi,-\frac{2\pi}{3}),
\end{cases}
\end{equation}
satisfies conditions (a) and (b) in RHP \ref{modelproblem}, it obeys the normalization
\begin{equation}\label{mp2b}
	\Theta(\xi;a)=\left\{\textbf{I}+\frac{4a^2-1}{128\xi}\begin{pmatrix}4a^2-9 & 16\i\\ \frac{\i}{12}(4a^2-9)(4a^2-13) & 9-4a^2\end{pmatrix}+\mathcal{O}\big(\xi^{-2}\big)\right\}\xi^{-\frac{1}{4}\sigma_3}\frac{\textbf{I}+\i\sigma_1}{\sqrt{2}}\e^{\xi^{\frac{1}{2}}\sigma_3},
\end{equation}
near $\xi=\infty,\xi\notin\Sigma_1\cup\Sigma_2\cup\Sigma_3$ for $\textnormal{arg}\,\xi\in(-\pi,\pi)$. Define
\[
g(\xi;a)
:=
\begin{cases}
-\dfrac{\i}{2\sin(\pi a)},
& a\notin\mathbb Z,
\\[1em]
\dfrac{\i}{2\pi}(-1)^{a+1}\ln \xi,
& a\in\mathbb Z_{\geq0}.
\end{cases}
\]
Then, in a vicinity of \(\xi=0\), with \(\arg\xi\in(-\pi,\pi)\),
\begin{equation*}
\Theta(\xi;a)
=
\mathring{\Theta}(\xi;a)
\xi^{\frac{a}{2}\sigma_3}
\begin{pmatrix}
1 & g(\xi;a)\\
0 & 1
\end{pmatrix}
\begin{cases}
\mathbf I,
&
\arg\xi\in
\left(-\dfrac{2\pi}{3},\dfrac{2\pi}{3}\right),
\\[0.8em]
\begin{pmatrix}
1 & 0\\
-\e^{\pi\i a} & 1
\end{pmatrix},
&
\arg\xi\in
\left(\dfrac{2\pi}{3},\pi\right),
\\[1em]
\begin{pmatrix}
1 & 0\\
\e^{-\pi\i a} & 1
\end{pmatrix},
&
\arg\xi\in
\left(-\pi,-\dfrac{2\pi}{3}\right).
\end{cases}
\end{equation*}
Here, \(\xi\mapsto\mathring{\Theta}(\xi;a)\) is analytic at \(\xi=0\)
and is given by
\begin{align*}
a\notin\mathbb Z:\qquad
\mathring{\Theta}(\xi;a)
&=
\begin{pmatrix}
1 & 0\\
\dfrac{\i}{8}(4a^2+3) & 1
\end{pmatrix}
\pi^{\sigma_3/2}
\begin{pmatrix}
I_{11}(\xi;a)
&
\dfrac{\i}{2\sin(\pi a)}I_{11}(\xi;-a)
\\[1ex]
I_{21}(\xi;a)
&
\dfrac{\i}{2\sin(\pi a)}I_{21}(\xi;-a)
\end{pmatrix},
\\[2ex]
a\in\mathbb Z_{\geq0}:\qquad
\mathring{\Theta}(\xi;a)
&=
\begin{pmatrix}
1 & 0\\
\dfrac{\i}{8}(4a^2+3) & 1
\end{pmatrix}
\pi^{\sigma_3/2}
\begin{pmatrix}
I_{11}(\xi;a)
&
I_{12}(\xi;a)
\\[1ex]
I_{21}(\xi;a)
&
I_{22}(\xi;a)+(-\xi)^a I_{11}(\xi;a)
\end{pmatrix}
\\
&\qquad\times
\begin{pmatrix}
1 & \dfrac{\i}{\pi}(-\xi)^a\log 2\\
0 & 1
\end{pmatrix}.
\end{align*}
in terms of the entire functions
\begin{equation*}	
	\begin{cases}\displaystyle I_{11}(\xi;a):=\frac{1}{2^a}\sum_{k=0}^{\infty}\frac{(\xi/4)^k}{k!\Gamma(a+k+1)},&\smallskip\\ \displaystyle I_{21}(\xi;a):=\frac{\i\pi}{2^a}\sum_{k=0}^{\infty}\frac{(2k+a)(\xi/4)^k}{k!\Gamma(a+k+1)},&\end{cases} 
\end{equation*}
and, for \(a\in\mathbb Z_{\geq0}\),
\begin{equation*}
\begin{aligned}
I_{12}(\xi;a)
:=
\frac{\i}{2\pi}
\Bigg[
&2^a
\sum_{k=0}^{a-1}
\frac{\Gamma(a-k)}{k!}
\left(\frac{-\xi}{4}\right)^k
+
\left(\frac{-\xi}{2}\right)^a
\sum_{k=0}^{\infty}
\bigl(\psi(k+1)+\psi(a+k+1)\bigr)
\frac{(\xi/4)^k}{k!(a+k)!}
\Bigg],
\end{aligned}
\end{equation*}
and
\begin{equation*}
\begin{aligned}
I_{22}(\xi;a)
:=
-\frac12
\Bigg[
&2^a
\sum_{k=0}^{a-1}
\frac{\Gamma(a-k)}{k!}
(2k-a)
\left(\frac{-\xi}{4}\right)^k
\\
&\quad+
\left(\frac{-\xi}{2}\right)^a
\sum_{k=0}^{\infty}
\bigl(\psi(k+1)+\psi(a+k+1)\bigr)
\frac{(2k+a)(\xi/4)^k}{k!(a+k)!}
\Bigg].
\end{aligned}
\end{equation*}
 Here $\psi(z)$ is the digamma function. As seen from the above, $\Theta(\xi;a)$ does not match up with condition (d) in RHP \ref{modelproblem}, still the model function $\Theta(\xi;a)$ is useful in the $x\downarrow 0$ asymptotic analysis of RHP \ref{modelproblem} as shown in \cite[Section $5$]{xu2014critical}. In order to correct the mismatch between $\Theta(\xi;a)$ and $\Psi(\xi;x,a)$ near $\xi=0$, one considers the following additional model function: let $\xi\in\mathbb{D}_{\epsilon}(0)\setminus\bigcup_{j=1}^3\Sigma_j$ and keep $a>-1,x>0$ fixed. Define
\begin{equation}\label{mp3}
	\Theta_0(\xi;x,a):=\mathring{\Theta}(\xi;a)\begin{pmatrix}1 & \mathring{f}(\xi;x,a)\\ 0 & 1\end{pmatrix}\xi^{\frac{a}{2}\sigma_3}\e^{\frac{x}{\xi}\sigma_3}\begin{cases}
\textbf{I}, & \textnormal{arg}\,\xi\in(-\frac{2\pi}{3},\frac{2\pi}{3}), \\[0.6em]
\begin{pmatrix}
1 & 0 \\
- \mathrm{e}^{\pi \i a} & 1
\end{pmatrix}, & \textnormal{arg}\,\xi \in(\frac{2\pi}{3},\pi), \\[0.8em]
\begin{pmatrix}
1 & 0 \\
\mathrm{e}^{-\pi \i a} & 1
\end{pmatrix}, & \textnormal{arg}\,\xi \in(-\pi,-\frac{2\pi}{3}),
\end{cases}
\end{equation}
with the entire function $\xi\mapsto\mathring{\Theta}(\xi;a)$ from before and the function
\begin{equation*}
	\mathring{f}(\xi;x,a):=f(\xi;x,a)-\frac{\i}{2}\begin{cases}\displaystyle\frac{\gamma^{\ast}(-a,\xi)}{\sin(\pi a)},&a\notin\mathbb{Z}_{\geq 0}\smallskip\\ \frac{1}{\pi}\big[(-\xi)^a\psi(a+1)-\sum_{\substack{k=0\\ k\neq a}}^{\infty}\frac{a!(-\xi)^k}{k!(k-a)}\big],&a\in\mathbb{Z}_{\geq 0}\end{cases};\ \ \ \gamma^{\ast}(a,z):=\frac{z^{-a}}{\Gamma(a)}\gamma(a,z),
\end{equation*}
with $f(\xi;x,a)$ as in condition (d) of RHP \ref{modelproblem}. Evidently, \eqref{mp3} satisfies conditions (a),(b) and (d) in RHP \ref{modelproblem}\footnote{One must recall that $z\mapsto\gamma^{\ast}(a,z)$ is an entire function.}. 
Crucially, for any fixed $a>-1$ and $x>0$, the map
\begin{equation*}
    \xi\mapsto \Psi(\xi;x,a)\Theta_0(\xi;x,a)^{-1}
\end{equation*}
is analytic at $\xi=0$.
It thus remains to compare $\Theta_0(\xi;x,a)$ and $\Theta(\xi;x,a)$ asymptotically as $x\downarrow 0$, for $|\xi|=\epsilon>0$ fixed, with $\epsilon>0$ sufficiently small. The key underlying expansion of $f(\xi;x,a)$ is worked out below.
\begin{prop}\label{appendix lemma} Let $a>-1$ be fixed, and suppose $\xi\in\mathbb{C}$ has modulus $|\xi|=\delta\in(0,1)$. Then for any $x>0$, using principal branches for all fractional exponents,
\begin{equation*}
	f(\xi;x,a)\stackrel{a\notin\mathbb{Z}_{\geq 0}}{=}-\frac{\i}{2\sin(\pi a)}\left[\sum_{n=0}^{\infty}\frac{(2x)^n}{n!}\xi^{a-n}\frac{\Gamma(n-a,\xi)}{\Gamma(n-a)}-\sum_{m=0}^{\infty}\frac{(2x)^{1+a+m}}{m!}\xi^{-1-m}\frac{\Gamma(1+m,\xi)}{\Gamma(2+a+m)}\right],
\end{equation*}
provided $\xi\neq-\delta$. For $\xi=|\delta|\e^{\pm\i\pi}$ the principal branch limiting values in the first series must be taken. In addition, using principal branches for $\xi\mapsto\ln\xi$ and for the exponential integral $\xi\mapsto E_1(\xi)$,
\begin{align*}
	f(\xi;x,a)\stackrel{a\in\mathbb{Z}_{\geq 0}}{=}\frac{\i}{2\pi}&\,\Bigg[\sum_{n=0}^a\frac{(-2x)^n}{n!}(-\xi)^{a-n}\big(\psi(a-n+1)-\ln\xi\big)-\sum_{n=0}^a\frac{(-2x)^n}{n!}\sum_{\substack{k=0\\ k\neq a-n}}^{\infty}\frac{(a-n)!(-\xi)^k}{k!(k+n-a)}\\
	&+(-1)^{a+1}\sum_{m=1}^{\infty}\frac{(2x)^{a+m}}{(a+m)!}\frac{\xi^{-m}}{(m-1)!}\Gamma(m,\xi)\Bigg(2\psi(a+m+1)-\sum_{k=m}^{a+m}\frac{1}{k}-\ln(2x)\\
	&-(m-1)!\Bigg[\frac{E_1(\xi)}{\Gamma(m,\xi)}+\sum_{k=1}^{m-1}\frac{\Gamma(k,\xi)}{\Gamma(m,\xi)k!}\Bigg]\Bigg),
\end{align*}
provided $\xi\neq-\delta$. For $\xi=|\delta|\e^{\pm\i\pi}$ the principal branch limiting values in $\xi\mapsto\ln\xi$ and in $\xi\mapsto E_1(\xi)$ must be taken.
Here $\psi(z)$ is the digamma function and
\begin{equation}\label{defofE1}
E_1(z):=\int_z^{\infty}\e^{-t}\frac{\d t}{t},\ z\in\mathbb{C}\setminus(-\infty,0].
\end{equation}

\end{prop}
\begin{proof} Suppose $|\xi|=\delta\in(0,1)$ lies off the negative real axis. Using, valid for any $z\in\mathbb{C}$ with $\textnormal{Re}(z)>0$ and $\textnormal{arg}\,z\in(-\frac{\pi}{2},\frac{\pi}{2})$,
\begin{equation*}
	\e^{-z}=\frac{1}{2\pi\i}\int_{c-\i\infty}^{c+\i\infty}\Gamma(s)z^{-s}\d s,\ \ \ c>0,
\end{equation*}
we have
\begin{equation*}
	f(\xi;x,a)=-\frac{\e^{-\xi}}{2\pi\i}\int_0^{\infty}v^a\e^{-v-\frac{2x}{v}}\frac{\d v}{v+\xi}=\frac{\e^{-\xi}}{(2\pi)^2}\int_{c-\i\infty}^{c+\i\infty}\Gamma(s)(2x)^{-s}\left[\int_0^{\infty}v^{a+s}\e^{-v}\frac{\d v}{v+\xi}\right]\d s,
\end{equation*}
where the innermost integral converges because $\textnormal{Re}(a+s)=a+c>-1$ by assumption. We can evaluate the same in terms of the incomplete Gamma function $\Gamma(a,z)$,
\begin{equation*}
	\int_0^{\infty}v^{a+s}\e^{-v}\frac{\d v}{v+\xi}=\Gamma(1+a+s)\int_0^{\infty}\e^{-t\xi}(1+t)^{-a-s-1}\d t=\e^{\xi}\xi^{a+s}\Gamma(1+a+s)\Gamma(-a-s,\xi).
\end{equation*}
The last evaluation holds initially for $\textnormal{Re}\,\xi>0$, but by continuation eventually for $|\xi|=\delta\in(0,1)$ off the negative real axis. So far,
\begin{equation}\label{mp4}
	f(\xi;x,a)=\frac{1}{(2\pi)^2}\int_{c-\i\infty}^{c+\i\infty}\Gamma(s)\Gamma(1+a+s)(2x)^{-s}\xi^{a+s}\Gamma(-a-s,\xi)\d s,
\end{equation}
where $\mathbb{C}\ni s\mapsto (2x)^{-s}\xi^{a+s}\Gamma(-a-s,\xi)$ is entire, for all $\xi\in\mathbb{C}\setminus(-\infty,0]$ and $x>0,a>-1$. On the other hand, 
\begin{equation}\label{mp5}
	s\mapsto\Gamma(s)\Gamma(1+a+s)
\end{equation}
has poles at $s=-n,n\in\mathbb{Z}_{\geq 0}$ and at $s=-1-a-m,m\in\mathbb{Z}_{\geq 0}$. And those poles do not coincide precisely for $a\notin\mathbb{Z}$, i.e. for $a\notin\mathbb{Z}$, the integrand in \eqref{mp4} has only simple poles. In that case, residue theorem yields the explicit formula (using Euler's reflection formula en route)
\begin{equation*}
	f(\xi;x,a)=-\frac{\i}{2\sin(\pi a)}\left[\sum_{n=0}^{\infty}\frac{(2x)^n}{n!}\xi^{a-n}\frac{\Gamma(n-a,\xi)}{\Gamma(n-a)}-\sum_{m=0}^{\infty}\frac{(2x)^{1+a+m}}{m!}\xi^{-1-m}\frac{\Gamma(1+m,\xi)}{\Gamma(2+a+m)}\right],
\end{equation*}
as claimed above. If $\xi=-\delta<0$, then by convergence properties of the underlying series and the known connection formulas, for $\Gamma(a,z\e^{\pm\i\pi})$,
\begin{align*}
	f_{\pm}(-\delta;x,a)=&\,\lim_{\epsilon\downarrow 0}f(-\delta\pm\i\epsilon;x,a)=-\frac{\i}{2\sin(\pi a)}\Bigg[\delta^a\e^{\pm\i\pi a}\e^{-\frac{2x}{\delta}}-\sum_{n=0}^{\infty}\frac{(2x)^n}{n!}\gamma^{\ast}(n-a,-\delta)\\
	&-\sum_{m=0}^{\infty}\frac{(2x)^{1+a+m}}{m!}\xi^{-1-m}\frac{\Gamma(1+m,-\delta)}{\Gamma(2+a+m)}\Bigg].
\end{align*}
In the case $a\in\mathbb{Z}_{\geq 0}$, for $|\xi|=\delta\in(0,1)$ off the negative real axis, the first $a+1$ poles at $s=0,-1,\ldots,-a$ of \eqref{mp5} are again simple and we obtain from residue theorem
\begin{align*}
	f(\xi;x,a)=\frac{\i}{2\pi}&\,\sum_{n=0}^a\frac{(-2x)^n}{n!}\xi^{a-n}\Gamma(1+a-n)\Gamma(n-a,\xi)\\
	&+\frac{1}{(2\pi)^2}\int_{c'-\i\infty}^{c'+\i\infty}\Gamma(s)\Gamma(1+s+a)(2x)^{-s}\xi^{a+s}\Gamma(-a-s,\xi)\d s,\ \ c'\in(-1-a,-a)
\end{align*}
The remaining poles of \eqref{mp5} at $s=s_m:=-a-m,m\in\mathbb{N}$ are double poles and so we require the leading order Laurent expansion
\begin{equation*}
	\Gamma(s)\Gamma(1+a+s)=\frac{(-1)^{a+1}}{(m+a)!(m-1)!}\left\{\frac{1}{(s-s_m)^2}+\frac{1}{s-s_m}\left(2\psi(a+m+1)-\sum_{k=m}^{a+m}\frac{1}{k}\right)+\mathcal{O}(1)\right\},
\end{equation*}
valid for $0<|s-s_m|<\epsilon$, and the leading order Taylor expansion
\begin{equation*}
	(2x)^{-s}\xi^{a+s}\Gamma(-a-s,\xi)=(2x)^{a+m}\xi^{-m}\Gamma(m,\xi)\left\{1+(s-s_m)\bigg(\ln\bigg(\frac{\xi}{2x}\bigg)-\frac{\Gamma'(m,\xi)}{\Gamma(m,\xi)}\bigg)+\mathcal{O}\big((s-s_m)^2\big)\right\},
\end{equation*}
valid for $|s-s_m|<\epsilon$ and $|\xi|=\delta\in(0,1)$ off the negative real axis, with fixed $x>0$. Here, $\Gamma'(a,z)=\partial_a\Gamma(a,z)$ and therefore
\begin{align*}
	f(\xi;x,&\,a)=\frac{\i}{2\pi}\Bigg[\sum_{n=0}^a\frac{(-2x)^n}{n!}\xi^{a-n}(a-n)!\,\Gamma(n-a,\xi)\\
	&+(-1)^{a+1}\sum_{m=1}^{\infty}\frac{(2x)^{a+m}}{(a+m)!}\frac{\xi^{-m}}{(m-1)!}\Gamma(m,\xi)\Bigg(2\psi(a+m+1)-\sum_{k=m}^{a+m}\frac{1}{k}+\ln\bigg(\frac{\xi}{2x}\bigg)-\frac{\Gamma'(m,\xi)}{\Gamma(m,\xi)}\Bigg)\Bigg],
\end{align*}
as claimed above, since
\begin{equation*}
	\Gamma(n-a,\xi)=\frac{(-1)^{a-n}}{(a-n)!}\big(\psi(a-n+1)-\ln\xi\big)-\xi^{n-a}\sum_{\substack{k=0\\ k\neq a-n}}^{\infty}\frac{(-\xi)^k}{k!(k+n-a)},
\end{equation*}
and since
\begin{equation*}
	\Gamma'(m,\xi)=\Gamma(m,\xi)\ln\xi+(m-1)!\left[E_1(\xi)+\sum_{k=1}^{m-1}\frac{\Gamma(k,\xi)}{k!}\right]
\end{equation*}
with $E_{1}(z)$ defined as in (\ref{defofE1}). If $\xi=-\delta<0$, then taking the appropriate limiting values is straightforward, for all multivalued functions are outside the relevant series, and $\xi\mapsto\Gamma(n,\xi)$ is entire as soon as $n\in\mathbb{N}$. The proof is now complete.
 \end{proof}
 As a simple consequence of Proposition \ref{appendix lemma} we record the following leading order estimates
 \begin{cor}\label{impmatch} Let $a>-1$ be fixed and choose $|\xi|=\delta\in(0,1)$. Then, as $x\downarrow 0$,
 \begin{equation*}
 	f(\xi;x,a)=-\frac{\i}{2\sin(\pi a)}\Big[\xi^a-\gamma^{\ast}(-a,\xi)+\mathcal{O}(x)+\mathcal{O}\big(x^{1+a}\big)\Big],\ \ a\notin\mathbb{Z}_{\geq 0}
\end{equation*}
and
\begin{equation*}
	f(\xi;x,a)=\frac{\i}{2\pi}\bigg[(-\xi)^a\big(\psi(a+1)-\ln\xi\big)-\sum_{\substack{k=0\\ k\neq a}}^{\infty}\frac{a!(-\xi)^k}{k!(k-a)}+\mathcal{O}(x)+\mathcal{O}\big(x^{1+a}|\ln x|\big)\bigg],\ \ a\in\mathbb{Z}_{\geq 0}.
\end{equation*}
Throughout the appropriate principal branch values must be chosen when $\xi=\delta\e^{\pm\i\pi}$.
 \end{cor}

\section{Small-\(x\) Analysis for \(\sigma(x;a)\)}
\label{section6}
To deduce the conclusion of Theorem \ref{maintheorem2}, we require a more delicate analysis of \(\sigma(x;a)\) as \(x\downarrow 0\). To this end, we now study the Riemann-Hilbert Problem \ref{modelproblem} in the limit $x\downarrow 0$. We begin by defining, for any $x>0$ and $a>-1$,
\begin{equation}\label{R0def}
R(\xi;x,a):=\Psi(\xi;x,a)
\begin{cases}
	\Theta_0(\xi;x,a)^{-1}, & |\xi|<\epsilon,\\[6pt]
	\Theta(\xi;a)^{-1}, & |\xi|>\epsilon,
\end{cases}
\end{equation}
where $\Psi(\xi,x;a)$ is the solution of RHP \ref{modelproblem}, and $\Theta_0(\xi;x,a),\Theta(\xi;a)$ are defined in Section \ref{modelRHP}, see \eqref{mp2a} and \eqref{mp3}. Throughout, $\epsilon\in(0,1)$ is kept fixed.
The matrix-valued function \(R(\xi)\) defined in \eqref{R0def} solves the following Riemann--Hilbert problem.
\begin{problem}[Small norm problem]\label{SN3} Let $a>-1$ and $x>0$. The function $R(\xi)=R(\xi;x,a)\in\mathbb{C}^{2\times 2}$ defined in \eqref{R0def} has the following properties:
\begin{enumerate}
	\item[(1)] $\xi\mapsto R(\xi)$ is analytic for $\xi\in\mathbb{C}\setminus\Sigma_R$ with $\Sigma_R=\partial\mathbb{D}_{\epsilon}(0)$ shown in Figure \ref{fig9x}. On $\Sigma_R$, $R(\xi)$ admits continuous boundary values $R_{\pm}(\xi)$ as one approaches $\Sigma_R$ from either side of $\mathbb{C}\setminus\Sigma_R$ non-tangentially.
	\begin{figure}[tbh]
    \centering
	\begin{tikzpicture}[xscale=0.6,yscale=0.6]
	        \draw [->] (-5,0) -- (2,0) node [right] {$\footnotesize{\textnormal{Re}\,\xi}$};
		\draw [->] (-2.5,-1.8) -- (-2.5,1.8) node [above] {$\footnotesize{\textnormal{Im}\,\xi}$};
	\draw [thick, color=red, fill=white, decoration={markings, mark=at position 0.07 with {\arrow{<}}}, decoration={markings, mark=at position 0.4 with {\arrow{<}}}, decoration={markings, mark=at position 0.7 with {\arrow{<}}},postaction={decorate}] (-2.5,0) circle [radius=0.6];
	\draw[fill, color=blue!60!black](-2.5,0) circle [radius=0.04];
\end{tikzpicture}
\caption{The oriented jump contour $\Sigma_R$ for $R(\xi)$, in the complex $\xi$-plane drawn in \textcolor{red}{red}. The point $\xi=0$ is colored in \textcolor{blue!60!black}{blue}.}
\label{fig9x}
\end{figure}
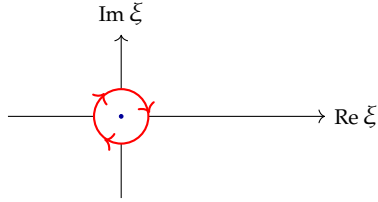
	\item[(2)] On $\Sigma_R$ we record the jump condition $R_+(\xi)=R_-(\xi)J_R(\xi)$ with $J_R(\xi)=J_R(\xi;x,a)$ given as
	\begin{equation}\label{defofJR}
		J_{R}(\xi)=\Theta_0(\xi;x,a)\Theta(\xi;a)^{-1},\ \ \ \xi\in\Sigma_R.
	\end{equation}
	By construction, $R(\xi)$ is analytic in $\xi\notin \Sigma_{R}$.
	\item[(3)] As $\xi\rightarrow\infty$, $R(\xi)\rightarrow I$.
\end{enumerate}
\end{problem}
Indeed, the absence of jumps in the ratio \eqref{R0def} on $\Sigma_1\cup\Sigma_2\cup\Sigma_3$ follows at once from our discussion below \eqref{mp2a} and \eqref{mp3}. The jump on $\Sigma_R$ is also immediate and near $\xi=0$, we have at once that
\begin{equation*}
	R(\xi;x,a)=\widetilde{\Psi}(\xi;x,a)\begin{pmatrix}1 & f(\xi;x,a)-\mathring{f}(\xi;x,a)\\ 0 & 1\end{pmatrix}\mathring{\Theta}(\xi;a)^{-1},\ \ \ \xi\rightarrow 0.
\end{equation*}
Consequently, $\xi\mapsto R(\xi)$ is bounded at $\xi=0$, and so analytic at $\xi=0$, i.e. $R(\xi)$ has no jumps in the interior of $\Sigma_R$. Lastly, the normalization at $\xi=\infty$ is also immediate for, as $\xi\rightarrow\infty$, by construction \eqref{mp2a} and \eqref{2.2}
\begin{equation*}
	R(\xi)=\left[\textbf{I}+\frac{C_1(x)}{\xi}+\mathcal{O}\big(\xi^{-2}\big)\right]\left[\textbf{I}-\frac{4a^2-1}{128\xi}\begin{pmatrix}4a^2-9 & 16\i\\ \frac{\i}{12}(4a^2-9)(4a^2-13) & 9-4a^2\end{pmatrix}+\mathcal{O}\big(\xi^{-2}\big)\right]\sim\textbf{I}.
\end{equation*}
The fact that RHP \ref{SN3} constitutes a small-norm problem as $x\downarrow 0$ is made precise below.
\begin{prop}\label{proponDelta}
Let $\epsilon\in(0,1)$ be fixed. For any $a>-1$, there exist constants $c_{a}>0$ and
$x_{0}=x_{0}(a)\in(0,1)$ such that, for $0<x\le x_{0}$,

\begin{equation}\label{less1}
\|J_R(\cdot;x,a)-\mathbf I\|_{L^2(\Sigma_R)\cap L^\infty(\Sigma_R)}
\leq
c_a
\begin{cases}
\max\{x,x^{a+1}\},
& a\notin\mathbb Z_{\geq0},
\\[1ex]
\max\{x,x^{a+1}|\ln x|\},
& a\in\mathbb Z_{\geq0}.
\end{cases}
\end{equation}
Moreover, for any integer $k\geq 1$, there is a constant $\widetilde{c}_{a,k}>0$ such that

\begin{equation}\label{less1b}
\left\|
\frac{\d^k}{\d x^k}J_R(\cdot;x,a)
\right\|_{L^2(\Sigma_R)\cap L^\infty(\Sigma_R)}
\leq
\widetilde c_{a,k}
\begin{cases}
\displaystyle
\max\{1,x^{a+1-k}\},
& a\notin\mathbb Z_{\geq0},
\\[1ex]
\displaystyle
\max\{1,x^{a+1-k}|\ln x|\},
& a\in\mathbb Z_{\geq0}.
\end{cases}
\end{equation}
holds for any $x$ with $0<x\leq x_{0}$.
\end{prop}
\begin{proof} By \eqref{mp3} and the conditions listed below \eqref{mp2a}, for $a\notin\mathbb{Z}_{\geq 0}$ and $|\xi|=\epsilon$ with $\epsilon\in(0,1)$ sufficiently small but fixed,
\begin{align*}
	J_R(\xi;x,a)=&\,\,\mathring{\Theta}(\xi;a)\begin{pmatrix}1 & \mathring{f}(\xi;x,a)\\ 0 & 1\end{pmatrix}\e^{\frac{x}{\xi}\sigma_3}\begin{pmatrix}1 & \frac{\i\xi^a}{2\sin(\pi a)}\\ 0 & 1\end{pmatrix}\mathring{\Theta}(\xi;a)^{-1}\\
	=&\,\,\mathring{\Theta}(\xi;a)\begin{pmatrix}1 & \mathring{f}(\xi;x,a)\\ 0 & 1\end{pmatrix}\left\{\begin{pmatrix}1 & \frac{\i\xi^a}{2\sin(\pi a)}\\ 0 & 1\end{pmatrix}+\mathcal{O}(x)\right\}\mathring{\Theta}(\xi;a)^{-1}\\
	=&\,\,\mathring{\Theta}(\xi;a)\left\{\begin{pmatrix}1&\mathcal{O}(x)+\mathcal{O}(x^{1+a})\\ 0 & 1\end{pmatrix}+\mathcal{O}(x)\right\}\mathring{\Theta}(\xi;a)^{-1},
\end{align*}
where the third equality follows from Corollary \ref{impmatch}. Quite similarly, for $a\in\mathbb{Z}_{\geq 0}$,
\begin{align*}
	J_R(\xi;x,a)=&\,\,\mathring{\Theta}(\xi;a)\begin{pmatrix}1 & \mathring{f}(\xi;x,a)\\ 0 & 1\end{pmatrix}\e^{\frac{x}{\xi}\sigma_3}\begin{pmatrix}1 & \frac{\i}{2\pi}(-\xi)^a\ln\xi\\ 0 & 1\end{pmatrix}\mathring{\Theta}(\xi;a)^{-1}\\
	=&\,\,\mathring{\Theta}(\xi;a)\begin{pmatrix}1 & \mathring{f}(\xi;x,a)\\ 0 & 1\end{pmatrix}\left\{\begin{pmatrix}1 & \frac{\i}{2\pi}(-\xi)^a\ln\xi\\ 0 & 1\end{pmatrix}+\mathcal{O}(x)\right\}\mathring{\Theta}(\xi;a)^{-1}\\
	=&\,\,\mathring{\Theta}(\xi;a)\left\{\begin{pmatrix}1&\mathcal{O}(x)+\mathcal{O}(x^{1+a}|\ln x|)\\ 0 & 1\end{pmatrix}+\mathcal{O}(x)\right\}\mathring{\Theta}(\xi;a)^{-1},
\end{align*}
using once more Corollary \ref{impmatch}. The above establishes \eqref{less1} and the derivative estimate \eqref{less1b} follows similarly.
\end{proof}
Seeing that RHP \ref{SN3} is indeed a small norm problem as $x\downarrow 0$ by \eqref{less1} we can thus solve the same iteratively, for small $x$, by general theory \cite{DZ,deift1999strong}.

\begin{lem}\label{integralrepresentation}
Let $\epsilon\in(0,1)$ be fixed. For any $a>-1$, there exist constants
$c=c(a)>0$ and $x_0=x_0(a)\in(0,1)$ such that RHP \ref{SN3} is uniquely solvable for all $0<x\le x_0$. Moreover, its solution admits the integral representation
\begin{equation}\label{z1}
R(\xi)
=
\mathbf{I}
+\frac{1}{2\pi\i}
\int_{\Sigma_R}
R_-(\mu)\bigl(J_R(\mu)-\mathbf{I}\bigr)
\frac{\d\mu}{\mu-\xi},
\qquad
\xi\in\mathbb{C}\setminus\Sigma_R,
\end{equation}
where
\begin{equation}\label{estimateonR-}
\|R_-(\cdot;x,a)-\mathbf{I}\|_{L^2(\Sigma_R)}
\le
c
\begin{cases}
\max\{x,x^{a+1}\},
& a\notin\mathbb Z_{\ge0},
\\[0.3em]
\max\{x,x^{a+1}|\ln x|\},
& a\in\mathbb Z_{\ge0},
\end{cases}
\qquad
\forall\,0<x\le x_0.
\end{equation}
\end{lem}
\begin{proof} 
Estimate \eqref{estimateonR-} follows from standard norm estimates of Cauchy singular integral operators. For completeness and for later reference, we briefly provide its derivation.
Let
\[
\mathcal C_{\pm}:
L^2\!\left(\Sigma_R;\mathbb C^{2\times2}\right)
\longrightarrow
L^2\!\left(\Sigma_R;\mathbb C^{2\times2}\right)
\]
denote the Cauchy operators defined by
\bea\label{eq:Cauchy-projections}
\mathcal{C}_{\pm}[f](\xi)
=
\lim_{\substack{\xi'\to\xi\\ \xi'\in \mathrm{side}\,\pm}}
\frac{1}{2\pi\i}
\int_{\Sigma_R}
f(\mu)\frac{\d\mu}{\mu-\xi'},
\qquad \xi\in\Sigma_R,
\eea
where ``side $+$'' denotes the exterior of $\Sigma_R$ and ``side $-$'' its interior. By the Sokhotski--Plemelj formula,
\beas
\mathcal{C}_{+}[f]-\mathcal{C}_{-}[f]
=
f,
\qquad \text{a.e. on}\ \ \Sigma_R,
\eeas
and it is well known that $\mathcal{C}_{\pm}$ are bounded linear operators on $L^2\!\left(\Sigma_R;\mathbb C^{2\times2}\right)$. By \eqref{z1},
\beas
R_-(\xi)
=
\mathbf I+\mathcal{C}_{-}\!\left[R_-(J_R-\mathbf I)\right](\xi).
\eeas
For convenience, introduce the operator
\bea\label{defofcg}
\mathcal{C}_{J}[f]
:=
\mathcal{C}_{-}\!\left[f(J_R-\mathbf I)\right],\ \ \ f\in L^2\big(\Sigma_R;\mathbb{C}^{2\times 2}\big).
\eea
Then $\mathcal{C}_{J}$ is a bounded operator on $L^2\!\left(\Sigma_R;\mathbb C^{2\times2}\right)$. Indeed, in operator norm on $L^2\!\left(\Sigma_R;\mathbb C^{2\times2}\right)$,
\bea\label{normofcg}
\|\mathcal{C}_{J}\|
\le
\|\mathcal{C}_{-}\|
\,\|J_R(\cdot;x,a)-\mathbf I\|_{L^\infty(\Sigma_R)}
\stackrel{\eqref{less1}}{\le}
\frac12,
\eea
provided that $x$ is sufficiently small. Consequently,
\begin{align}\label{neumann}
R_-(\xi)
=
(I-\mathcal{C}_{J})^{-1}[\mathbf{I}](\xi)
=
\sum_{n=0}^{\infty}
	\mathcal{C}_{J}^{\,n}[\mathbf I](\xi),
\end{align}
where $\mathcal{C}_{J}^{\,n}$ denotes the $n$-fold composition of $\mathcal{C}_{J}$, and the Neumann series converges in operator norm on $L^2\!\left(\Sigma_R;\mathbb C^{2\times2}\right)$. It therefore follows from \eqref{normofcg} that there exists a constant $c=c(\Sigma_R)>0$ such that
\bea\label{0827}
\|R_-(\cdot;x,a)-\mathbf I\|_{L^2(\Sigma_R)}
=
\left\|
\sum_{n=1}^{\infty}
\mathcal{C}_{J}^{\,n}[\mathbf I]
\right\|_{L^2(\Sigma_R)}
\le
c\,
\|J_R(\cdot;x,a)-\mathbf I\|_{L^\infty(\Sigma_R)}
\eea
for all sufficiently small $x$. Estimate \eqref{estimateonR-} now follows immediately from \eqref{less1}.
\end{proof}
\begin{prop}
Let $a>-1$. Let $v(x)=v(x;a)$ be as in (\ref{36}). Then
\begin{equation}\label{defforv2}
v(x)=\frac{1}{8}(1-4a^2)-
\left[\frac{1}{2\pi}
\int_{\Sigma_{R}}
R_{-}(u)
\bigl(J_{R}(u)-\textbf{I}\bigr)\d u \right]_{12}.
\end{equation}
\end{prop}
\begin{proof}
Recall that in (\ref{36}),
\begin{equation}
v(x)=\i\,\left[C_1(x)\right]_{12},
\end{equation}
where $C_1(x)$ is defined in \eqref{2.2}. We record the following chain of equalities,
\begin{align*}
C_{1}(x)
&\stackrel{\eqref{2.2}}{=}
\lim_{\xi\to\infty}
\xi\Bigg(
\Psi(\xi,x;a)
\e^{-\xi^{\frac{1}{2}}\sigma_3}
\left(
\frac{\textbf{I}+\i\sigma_1}{\sqrt{2}}
\right)^{-1}
\xi^{\frac{\sigma_3}{4}}
-\textbf{I}
\Bigg)
\\
&\stackrel{\eqref{R0def}}{=}
\lim_{\xi\to\infty}
\xi\Bigg(
R(\xi)\Theta(\xi)
\e^{-\xi^{\frac{1}{2}}\sigma_3}
\left(
\frac{\textbf{I}+\i\sigma_1}{\sqrt{2}}
\right)^{-1}
\xi^{\frac{\sigma_3}{4}}
-\textbf{I}
\Bigg)
\\
&\stackrel{\eqref{mp2b}}{=}
\lim_{\xi\to\infty}
\xi\Bigg(
R(\xi)
\left(
\textbf{I}+
\frac{4a^2-1}{128\xi}
\begin{pmatrix}
4a^2-9 & 16\i\\
\frac{\i}{12}(4a^2-9)(4a^2-13) & 9-4a^2
\end{pmatrix}
+\mathcal{O}\!\left(\xi^{-2}\right)
\right)
-\textbf{I}
\Bigg)\\
&\stackrel{\eqref{z1}}{=}\frac{4a^2-1}{128}
\begin{pmatrix}
4a^2-9 & 16\i\\[0.3em]
\dfrac{\i}{12}(4a^2-9)(4a^2-13)
&
9-4a^2
\end{pmatrix}
+
\frac{\i}{2\pi}
\int_{\Sigma_{R}}
R_{-}(u)
\bigl(J_{R}(u)-\textbf{I}\bigr)\d u,
\end{align*}
and taking the $(12)$-entry therein results, after multiplication with $\i$, in \eqref{defforv2}.
\end{proof}

\begin{lem}\label{asyzero}
Let $a>-1$. Let $v(x)=v(x;a)$ be as in \eqref{36}. Then, as $x\downarrow 0$, 
\begin{align*}
v(x)
=
\begin{cases}
\displaystyle
\frac{1}{8}\big(1-4a^2\big)
+\frac{x}{a}
-\frac{\pi}{\sin(\pi a)}\frac{x^{a+1}}{2^a\Gamma^2(1+a)\Gamma(2+a)}
+\mathcal{O}\left(\max\big\{x^2,x^{2a+2}\big\}\right),
& a\notin\mathbb{Z}_{\geq 0},
\\[1.2em]
\displaystyle
\frac{1}{8}(1-4a^2)
+\frac{x}{a}
-
\frac{(-1)^a2^{-a}x^{a+1}\ln x}
{\Gamma^2(1+a)\Gamma(2+a)}
+\mathcal{O}\left(x^2\right),
& a\in\mathbb{Z}_{\geq1},
\\[1.2em]
\displaystyle
\frac{1}{8}
-
x\left[
\log\!\left(\frac{x}{2}\right)
+3\gamma_{\mathrm E}-1
\right]
+\mathcal{O}\left(x^{2}|\ln x|^2\right),
& a=0,
\end{cases}
\end{align*}
where $\gamma_{\mathrm E}$ is the Euler-Mascheroni constant.
\end{lem}
\begin{proof}
Firstly,
by \eqref{0827}, Proposition \ref{proponDelta} and Cauchy-Schwarz inequality,
\begin{align}\label{0827even}
\frac{1}{2\pi}
\int_{\Sigma_{R}}
R_{-}(u)
\big(J_{R}(u)-\textbf{I}\big)\d u 
=&\frac{1}{2\pi}
\oint_{\Sigma_{R}}
\big(J_{R}(u)-\textbf{I}\big)\d u
+
\begin{cases}
\displaystyle
\mathcal{O}\left(\max\big\{x^2,x^{2a+2}\big\}\right) & a\notin \mathbb{Z}_{\geq 0}\smallskip\\
\displaystyle
\mathcal{O}\left(\max\big\{x^2,x^{2a+2}|\ln x|^2\big\}\right)
& a \in \mathbb{Z}_{\geq 0}.
\end{cases}
\end{align}
Secondly, refining the workings that lead to \eqref{less1}, we have for $a>-1$,
\begin{equation*}
	J_R(\xi)=\mathring{\Theta}(\xi)\left\{\mathbf{I}+\frac{x}{\xi}\begin{pmatrix}1 & m(\xi;x,a)\\ 0 & -1\end{pmatrix}+\mathcal{O}\big(x^2\big)\right\}\mathring{\Theta}(\xi)^{-1},\ \ \ x\downarrow 0,
\end{equation*}
uniformly in $|\xi|=\epsilon\in(0,1)$. Here,
\begin{equation*}
	m(\xi;x,a)\stackrel{a\notin\mathbb{Z}_{\geq 0}}{=}\frac{\i}{\sin(\pi a)}\left[\gamma^{\ast}(-a,\xi)-\frac{\e^{-\xi}}{\Gamma(1-a)}+(2x)^a\frac{\e^{-\xi}}{\Gamma(2+a)}+\mathcal{O}(x)+\mathcal{O}\big(x^{1+a}\big)\right]
\end{equation*}
followed by ($z\mapsto\textnormal{Ein}(z)$ is the complementary exponential integral, an entire function),
\begin{align*}
	m(\xi;x,a)\stackrel{a=0}{=}&\,\,\frac{\e^{-\xi}}{\i\pi}\big(2\psi(2)-1-\ln(2x)\big)+\frac{\i}{\pi}\big(\textnormal{Ein}(\xi)-\gamma_E\big)+\mathcal{O}\big(x|\ln x|\big)+\mathcal{O}(x),\\
	m(\xi;x,a)\stackrel{a\in\mathbb{N}}{=}&\,\,\frac{\i}{\pi}\Bigg[(-\xi)^a\psi(a)+\xi\sum_{\substack{k=0\\ k\neq a-1}}^{\infty}\frac{(a-1)!(-\xi)^k}{k!(k+1-a)}+\ln(2x)\frac{(-2x)^a}{(a+1)!}\e^{-\xi}+\mathcal{O}(x)\Bigg].
\end{align*}
After that, by residue theorem,
\begin{equation*}
	\frac{1}{2\pi}\oint_{\Sigma_R}\big(J_R(u)-\mathbf{I}\big)\d u=-\i x\mathring{\Theta}(0)\left\{\begin{pmatrix}1 & m(0;x,a)\\ 0 & -1\end{pmatrix}+\mathcal{O}\big(x^2\big)\right\}\mathring{\Theta}(0)^{-1},
\end{equation*}
with the special value
\begin{align*}
	m(0;x,a)&\,\stackrel{a\notin\mathbb{Z}_{\geq 0}}{=}\frac{\i}{\sin(\pi a)}\Bigg[\frac{(2x)^a}{\Gamma(2+a)}+\mathcal{O}(x)+\mathcal{O}\big(x^{1+a}\big)\Bigg],\\
	m(0;x,0)=&\,\,\frac{\i}{\pi}\Bigg[\ln(2x)-2\psi(2)+1-\gamma_E+\mathcal{O}\big(x|\ln x|\big)\Bigg],\\
	 m(0;x,a)&\stackrel{a\in\mathbb{N}}{=}\frac{\i}{\pi}\Bigg[\ln(2x)\frac{(-2x)^a}{(a+1)!}+\mathcal{O}(x)\Bigg].
\end{align*}
Since $\mathring{\Theta}(0)$ is explicit, see right below \eqref{mp2b}, we obtain consequently
\begin{align*}
&-\mathrm{i}x
\left[
\mathring{\Theta}(0)
\left\{
\begin{pmatrix}
1 & m(0;x,a)\\
0 & -1
\end{pmatrix}
+\mathcal O(x^2)
\right\}
\mathring{\Theta}(0)^{-1}
\right]_{12}
\\
=&
\begin{cases}
\displaystyle
-\frac{x}{a}
+
\frac{\pi}{\sin(\pi a)}
\frac{x^{a+1}}
{2^a\Gamma^2(1+a)\Gamma(2+a)}
+
\mathcal O(x^2)
+
\mathcal O(x^{a+2}),
& a\notin\mathbb Z_{\geq0},
\\[1.2ex]
\displaystyle
x\left[
\ln\!\left(\frac{x}{2}\right)
+3\gamma_E-1
\right]
+
\mathcal O\!\left(x^2|\ln x|\right),
& a=0,
\\[1.2ex]
\displaystyle
-\frac{x}{a}
+
\frac{(-1)^a x^{a+1}\ln x}
{2^a\Gamma^2(1+a)\Gamma(2+a)}
+
\mathcal O(x^2),
& a\in\mathbb N.
\end{cases}
\end{align*}
and the conclusion of the present lemma thus follows from \eqref{0827even}.
\end{proof}
We now prove the full small-$x$ asymptotic expansion of
\(\sigma(x;a)\) stated in Proposition~\ref{smallasyanlysis}.
\begin{proof}[Proof of Proposition~\ref{smallasyanlysis}]
Firstly, we abbreviate
\[
H(x;a)=\frac{\i}{2\pi}\int_{\Sigma_R}R_-(\mu)\big(J_R(\mu)-\textbf{I}\big)\d\mu.
\]
By \eqref{defforv2} and \eqref{definitionofsigma}, we have
\begin{equation}\label{expressionofv}
\sigma(x;a)=-\frac{\i}{2}[H(x;a)]_{12}-x.
\end{equation} In the following, we only prove the conclusion for the case \(a>-1\) and \(a\notin \mathbb{Z}_{\ge 0}\); the non-generic case $a\in\mathbb{Z}_{\geq 0}$ follows from almost identical arguments. Choose \(x_{0}=x_{0}(a)\in (0,1)\) such that \(x_{0}<\epsilon/2\) and \eqref{less1} holds. First, we claim that
\begin{equation}\label{truncation}
H(x;a)=\sum_{n=0}^{\mathrm{M}}\frac{\i}{2\pi}\int_{\Sigma_R}\mathcal{C}_{J}^{n}[\textbf{I}](\mu)\big(J_{R}(\mu)-\textbf{I}\big)\d\mu+\mathcal{O}_{a}\left(\max\Big\{x^{\mathrm{M}+2},x^{(a+1)(\mathrm{M}+2)}\Big\}\right)
\end{equation}
holds for any \(x\in [0,x_{0})\) and \(\mathrm{M}\in\mathbb{Z}_{\geq 0}\). Indeed, by \eqref{normofcg}, \eqref{neumann}, and the Cauchy--Schwarz inequality,
\begin{align*}
&\Bigg|H(x;a)-\sum_{n=0}^{\mathrm{M}}\frac{\i}{2\pi}\int_{\Sigma_R}\mathcal{C}_{J}^{n}[\textbf{I}](\mu)\big(J_R(\mu)-\textbf{I}\big)\d\mu\Bigg| \\
\qquad\leq& \frac{1}{2\pi}\left\|\sum_{n=\mathrm{M}+1}^{\infty}\mathcal{C}_{J}^{n}[\textbf{I}](\cdot;x,a)\right\|_{L^2(\Sigma_R)}
\|J_R(\cdot;x,a)-\textbf{I}\|_{L^2(\Sigma_R)} \\
\qquad\leq &\,c\,\|J_{R}(\cdot;x,a)-\textbf{I}\|_{L^{\infty}(\Sigma_R)}^{\mathrm{M}+1}
\|J_R(\cdot;x,a)-\textbf{I}\|_{L^2(\Sigma_R)},
\qquad c=c(\Sigma_R)>0.
\end{align*}
Then \eqref{truncation} follows from \eqref{less1}. Next, we claim that for each \(n\in\mathbb{Z}_{\geq 0}\),
\begin{align}\label{seriesforM}
\frac{\i}{2\pi}\int_{\Sigma_R}\mathcal{C}_{J}^{n}[\textbf{I}](\mu)\big(J_{R}(\mu)-\textbf{I}\big)\d\mu
=
\sum_{k=n+1}^\infty \mathsf{b}_{k}^{(n)}(a)
x^k+\sum_{\ell=1}^{n+1}\sum_{k=n+1}^\infty \mathsf{b}_{k,\ell}^{(n)}(a)x^{a \ell+k},
\end{align}
where
$
\mathsf{b}_{k}^{(n)}(a), \mathsf{b}_{k,\ell}^{(n)}(a)
=\mathcal{O}_{n,a}\left(k^{2n}\epsilon^{-k}\right),
$
and, by the choice of \(x_{0}\), the series is absolutely and uniformly convergent for \(x \in [0, x_{0}]\). Indeed, by the workings in Proposition \ref{proponDelta} and by Proposition \ref{appendix lemma}, each entry of \(J_{R}(\xi;x,a)-\textbf{I}\) has the form
\begin{align}\label{entryform}
\sum_{k=1}^\infty \mathsf{c}_{k}(a)\frac{f(\xi)}{\xi^{k}}x^k
+x^{a}\sum_{k=1}^\infty \mathsf{d}_{k}(a)\frac{h(\xi)}{\xi^{k}}x^k,
\end{align}
where
$
\mathsf{c}_{k}(a),\mathsf{d}_{k}(a)=\mathcal{O}_{a}(1),
$
and \(f(\xi)\), \(h(\xi)\) are analytic in \(\mathbb{D}_{\epsilon}(0)\) and are given in terms of the entries of $\mathring{\Theta}(\xi;a)$ and $\mathring{\Theta}(\xi;a)^{-1}$. Moreover, the series in \eqref{entryform} converges uniformly for \(\xi \in \Sigma_{R}\) and \(x\in (0,x_{0}]\). Next, by the definition of \(\mathcal{C}_{J}\) in \eqref{defofcg}, for \(\zeta\in \Sigma_{R}\),
\begin{align}
&\lim_{\substack{w\to\xi\\ w\in\text{side}\,-}}
    \frac{1}{2\pi \i}\int_{\Sigma_R} \frac{\eqref{entryform}}{\mu -w}\d\mu\nonumber\\
=&\lim_{\substack{w\to\xi\\ w\in\text{side}\,-}}
\frac{1}{2\pi \i}\sum_{k=1}^{\infty}
\left(
\mathsf{c}_{k}(a)\int_{\Sigma_R} \frac{f(\mu)}{\mu^{k}(\mu-w)}\d\mu
+x^{a}\,\mathsf{d}_{k}(a)\int_{\Sigma_R} \frac{h(\mu)}{\mu^{k}(\mu -w)}\d\mu
\right)x^{k}.
\label{Cauchyintegral}
\end{align}
For fixed \(w\) inside \(\Sigma_{R}\), noting that \(\Sigma_{R}=\partial\mathbb{D}_{\epsilon}(0)\) is oriented clockwise, Cauchy's integral formula yields
\begin{align*}
\frac{1}{2\pi \i}\int_{\Sigma_R}\frac{f(\mu)}{\mu^{k}(\mu-w)}\d\mu
&=
-\frac{f(w)}{w^k}
-\frac{1}{(k-1)!}
\frac{\d^{k-1}}{\d\mu^{k-1}}
\left(\frac{f(\mu)}{\mu-w}\right)\Bigg|_{\mu=0} \\
&=
\sum_{j=0}^{k-1}
\frac{f^{(k-1-j)}(0)}{(k-1-j)!}\frac{1}{w^{j+1}}
-\frac{f(w)}{w^k},
\end{align*}
and consequently
\begin{align*}
\eqref{Cauchyintegral}
=
\lim_{\substack{w\to\xi\\ w\in\text{side}\,-}}
\sum_{k=1}^{\infty}
\Bigg(
&\mathsf{c}_k(a)
\Bigg[
\sum_{j=1}^{k}
\frac{f^{(k-j)}(0)}{(k-j)!}\frac{1}{w^{j}}
-\frac{f(w)}{w^k}
\Bigg] \\
&+
x^{a}\,\mathsf{d}_k(a)
\Bigg[
\sum_{j=1}^{k}
\frac{h^{(k-j)}(0)}{(k-j)!}\frac{1}{w^{j}}
-\frac{h(w)}{w^k}
\Bigg]
\Bigg)\,x^{k}.
\end{align*}
Since \(f(\mu)\) and \(h(\mu)\) are analytic in \(\mathbb{D}_{\epsilon}(0)\), we have
$
\left|f^{(j)}(0)/j!\right|,
\left|h^{(j)}(0)/j!\right|
=\mathcal{O}(\epsilon^{-j}), j\ge 0,
$
and
$
|f(w)|, |h(w)|=\mathcal{O}(1).
$
By our choice of \(x_{0}\), we have \(x/\epsilon < 1/2\) for all \(x \in (0, x_0]\). Hence, the limit with respect to \(w\) can be interchanged with the infinite series in \(k\). Therefore, each entry of \(\mathcal{C}_{J}[\textbf{I}](\xi;x,a)\) is of the form
\begin{align}\label{entryform1}
\sum_{k=1}^\infty
\left(
\sum_{j=1}^{k}\mathsf{c}_{k,j}^{(1)}(a)\frac{1}{\xi^{j}}
-\mathsf{c}_k(a)\frac{f(\xi)}{\xi^k}
\right)x^{k}
+x^{a}\sum_{k=1}^\infty
\left(
\sum_{j=1}^{k}\mathsf{d}_{k,1,j}^{(1)}(a)\frac{1}{\xi^{j}}
-\mathsf{d}_k(a)\frac{h(\xi)}{\xi^k}
\right)x^k,
\end{align}
where
$
\mathsf{c}_{k,j}^{(1)}(a),
\mathsf{d}_{k,1,j}^{(1)}(a)
=\mathcal{O}_{a}\left(\epsilon^{-(k-j)}\right).
$ Note that for each \(n\geq 2\),
\begin{align*}
\mathcal{C}_{J}^{n}[\textbf{I}](\xi;x,a)
=
\lim_{\substack{w\to\xi\\ w\in\text{side}\,-}}
\frac{1}{2\pi \i}
\int_{\Sigma_R}
\mathcal{C}_{J}^{n-1}[\textbf{I}](\mu;x,a)
\big(J_{R}(\mu;x,a)-\textbf{I}\big)
\frac{\d\mu}{\mu-w}.
\end{align*} Using an argument similar to that leading to \eqref{entryform1}, together with induction, we find that each entry of \(\mathcal{C}_{J}^{n}[\textbf{I}](\xi;x,a)\) can be written as a sum of \(\mathcal{O}_n(1)\) terms of the form
\begin{align}\label{entryformforgeneraln}
\sum_{k=n}^\infty \sum_{j=1}^{k}
\left(
\mathsf{c}_{k,j}^{(n)}(a)
-\tilde{\mathsf{c}}_{k,j}^{(n)}(a)f_n(\xi)
\right)\frac{x^k}{\xi^{j}}
+\sum_{\ell=1}^{n}\sum_{k=n}^\infty \sum_{j=1}^{k}
\left(
\mathsf{d}_{k,\ell,j}^{(n)}(a)
-\tilde{\mathsf{d}}_{k,\ell,j}^{(n)}(a)h_n(\xi)
\right)\frac{x^{k+\ell a}}{\xi^{j}},
\end{align}
where
$
\mathsf{c}_{k,j}^{(n)}(a),
\mathsf{d}_{k,\ell,j}^{(n)}(a),
\tilde{\mathsf{c}}_{k,j}^{(n)}(a),
\tilde{\mathsf{d}}_{k,\ell,j}^{(n)}(a)
=
\mathcal{O}_{n,a}\left(k^{2(n-1)}\epsilon^{-(k-j)}\right),
$
and \(f_n(\xi)\), \(h_n(\xi)\) are certain functions analytic in \(\mathbb{D}_{\epsilon}(0)\). By our choice of \(x_{0}\) satisfying \(x_{0}/\epsilon < 1/2\), the series expansion of \(\mathcal{C}_{J}^{n}[\textbf{I}](\xi;x,a)\) in powers of \(x\) converges uniformly with respect to \(\xi \in \Sigma_{R}\) for all \(x \in (0, x_{0}]\). After multiplying by \(J_{R}(\xi; x, a)-\textbf{I}\), whose entries are of the form \eqref{entryform}, each entry of
\[
\mathcal{C}_{J}^{n}[\textbf{I}](\xi; x, a)
\big(J_{R}(\xi; x, a) - \textbf{I}\big)
\]
can again be rewritten as a series of the same type as \eqref{entryformforgeneraln}, with \(n\) replaced by \(n+1\). Applying Cauchy's integral formula term-by-term then yields \eqref{seriesforM}. Thirdly, combining \eqref{truncation} and \eqref{seriesforM}, we obtain that for \(\mathrm{M}\in\mathbb{Z}_{\geq 0}\) and \(x\in (0, x_{0}]\),
\begin{align}\label{seriesexpansionforM}
H(x;a)
=&
\sum_{k=1}^{\mathrm{M}+1}\tilde{\mathsf{b}}_{k}(a)x^k
+\sum_{\ell=1}^{\mathrm{M}+1}\sum_{k=\ell}^{\mathrm{M}+1}
\tilde{\mathsf{b}}_{k,\ell}(a)x^{k+a \ell}+
\mathcal{O}_{\mathrm{M},a}
\left(
\max\Big\{
x^{(a+1)(\mathrm{M}+2)},
x^{\mathrm{M}+2}
\Big\}
\right),
\end{align}
where
$
\tilde{\mathsf{b}}_{k}(a)
=
\sum_{n=0}^{k-1}\mathsf{b}_{k}^{(n)}(a),
\tilde{\mathsf{b}}_{k,\ell}(a)
=
\sum_{n=\ell-1}^{k-1}\mathsf{b}_{k,\ell}^{(n)}(a),
$
with \(\mathsf{b}_{k}^{(n)}(a)\) and \(\mathsf{b}_{k,\ell}^{(n)}(a)\) defined in \eqref{seriesforM}. Finally, we return to the asymptotic expansion of \(\sigma(x;a)\) as \(x \downarrow 0\). By \eqref{expressionofv} and \eqref{seriesexpansionforM}, for \(a\notin \mathbb{Z}_{\geq 0}\), we obtain the claimed asymptotic formula for \(\sigma(x;a)\). This completes the proof.
 \end{proof}

In the rest of this section, we derive the small-\(x\) estimates
for the derivatives of \(v(x)\) stated in
Lemma~\ref{asympoticsforv(x)forsmallx}. We first state a preliminary
lemma used in the derivation.

\begin{lem}\label{estimateforCJn}
Let $a>-1$. Let $x_0(a)$ be given as in
Proposition \ref{proponDelta}. Let $\mathcal C_J$ be the bounded linear
operator on $L^2(\Sigma_R;\mathbb{C}^{2\times 2})$ defined in \eqref{defofcg}, where
$\Sigma_R$ denotes the fixed contour of the Riemann--Hilbert problem
\ref{SN3}. For any $\ell\in\mathbb Z_{\ge0}$ and $n\ge1$, there exists a constant
$c_{a,\ell}>0$ such that, for any $x\in (0,x_{0})$,

\begin{equation}\label{estimateforCJnderivative}
\begin{aligned}
\left\|
\frac{\d^{\ell}}{\d x^{\ell}}
\mathcal C_J^n[\mathbf I]
\right\|_{L^2(\Sigma_R)}
&\leq
n^{\ell} c_{a,\ell}^{\,n}
\begin{cases}
\displaystyle
\max\{x,x^{a+1}\}^{\,n-\ell}
\left(1+x^{a+1-\ell}\right),
& a\notin\mathbb Z_{\geq0},
\\[1.2ex]
\displaystyle
\max\{x,x^{a+1}|\ln x|\}^{\,n-\ell}
\left(1+x^{a+1-\ell}|\ln x|\right),
& a\in\mathbb Z_{\geq0}.
\end{cases}
\end{aligned}
\end{equation}
\end{lem}
\begin{proof}
Let
\[
\Delta(\xi;x,a)=J_R(\xi;x,a)-\mathbf I.
\]
Recall that, see \eqref{defofcg},
\[
\mathcal{C}_{J}[f]
=
\mathcal{C}_{-}[f\Delta].
\]
Hence, $\mathcal{C}_{J}^{n}[\mathbf I]$ is an
$n$-linear expression in $\Delta$ involving $n$ Cauchy operators.
For $n\geq1$, define the $n$-linear operator $\mathcal{T}_n$ recursively by
\[
\mathcal{T}_1[F_1]=\mathcal{C}_{-}[F_1],
\qquad
\mathcal{T}_{n+1}[F_1,\ldots,F_{n+1}]
=
\mathcal{C}_{-}
\bigl[
\mathcal{T}_n[F_1,\ldots,F_n]F_{n+1}
\bigr],
\]
so that
\[
\mathcal{C}_{J}^{n}[\mathbf{I}]
=
\mathcal{T}_n[\Delta,\ldots,\Delta],
\]
and, by the boundedness of $\mathcal C_-$ on $L^2(\Sigma_R;\mathbb{C}^{2\times 2})$,
\bea\label{estimateforTn}
\|\mathcal T_n[f_1,\ldots,f_n]\|_{L^2(\Sigma_R)}
\leq
c^n
\prod_{j=2}^{n}\|f_j\|_{L^\infty(\Sigma_R)}
\|f_1\|_{L^2(\Sigma_R)} .
\eea
Here and below, the constant $c$ is independent of $x$. Taking the $l$-th derivative with respect to $x$ and applying the
Leibniz rule, we obtain
\[
\frac{\d^{\ell}}{\d x^{\ell}}
\mathcal C_J^n[\mathbf I]
=
\sum_{\substack{k_1+\cdots+k_n=\ell\\ k_j\geq0}}
\frac{\ell!}{k_1!\cdots k_n!}
\mathcal T_n
\left[
\frac{\d^{k_1}\Delta}{\d x^{k_1}},
\ldots,
\frac{\d^{k_n}\Delta}{\d x^{k_n}}
\right].
\]
For a fixed multi-index $(k_1,\ldots,k_n)$, let
$r$ be the number of positive integers among
$k_1,\ldots,k_n$. Since
$
k_1+\cdots+k_n=\ell
$,
we have
$
r\leq\ell
$.
Consequently, at least $n-r$ factors in the above multilinear
expression are not differentiated. Note that $0<x<x_0<1$ and $a+1>0$. Since $\mathbb{R}\ni\alpha\mapsto x^\alpha$ is a decreasing function, we have
$
\max\{x,x^{a+1}\}<1.$ Then, by combining the estimates for $J_R-\mathbf I$ and its derivatives in Proposition \ref{proponDelta} with the bound \eqref{estimateforTn}, we obtain that every term
in the above sum satisfies
\begin{equation}\label{0825eq1}
\begin{aligned}
\left\|
\mathcal T_n
\left[
\frac{\d^{k_1}\Delta}{\d x^{k_1}},
\ldots,
\frac{\d^{k_n}\Delta}{\d x^{k_n}}
\right]
\right\|_{L^2(\Sigma_R)}
&\leq
c_{a,\ell}^{\,n}
\begin{cases}
\displaystyle
\max\{x,x^{a+1}\}^{\,n-r}
\left(1+x^{a+1-\ell}\right),
& a\notin\mathbb Z_{\geq0},
\\[1.2ex]
\displaystyle
\max\{x,x^{a+1}|\ln x|\}^{\,n-r}
\left(1+x^{a+1-\ell}|\ln x|\right),
& a\in\mathbb Z_{\geq0}.
\end{cases}
\end{aligned}
\end{equation}
Since $r\leq \ell$, the above $r$ in (\ref{0825eq1}) can be replaced by $\ell$. Finally, using the multinomial identity
$
\sum_{k_1+\cdots+k_n=\ell}
\frac{\ell!}{k_1!\cdots k_n!}
=n^{\ell},
$
we obtain the desired estimate in this lemma.
\end{proof}

\begin{lem}\label{asympoticsforv(x)forsmallx}
Let \(a>-1\), and let \(v(x)\) be defined as in \eqref{36}.
Then for any $k\in\mathbb{Z}_{\geq 0}$, there exists \(x_{0}=x_{0}(a,k)\in(0,1)\) such that, for
\(0<x\leq x_{0}\),
\[
\frac{\mathrm{d}^{k}v(x)}{\mathrm{d}x^{k}}
=
\begin{cases}
\mathcal{O}_{a,k}(1)
+\mathcal{O}_{a,k}\!\left(x^{a+1-k}\right),
& a\notin\mathbb{Z}_{\geq 0},
\\[1ex]
\mathcal{O}_{a,k}(1)
+\mathcal{O}_{a,k}\!\left(
x^{a+1-k}|\ln x|
\right),
& a\in\mathbb{Z}_{\geq 0}.
\end{cases}
\]
\end{lem}
\begin{proof}
By \eqref{defforv2}, it suffices to study the small-$x$ behavior of
\[
\frac{\d^k}{\d x^k}
\left(
\int_{\Sigma_R}
R_-(u)
\bigl(J_R(u)-\mathbf I\bigr)\,\d u
\right).
\]
We first observe that, recall \eqref{neumann}, for any $\mathrm{M}\in\mathbb{Z}_{\geq 0}$,
\begin{align}
&\Bigg|
\frac{1}{2\pi}
\frac{\d^k}{\d x^k}
\int_{\Sigma_R}
R_-(u)\bigl(J_R(u)-\mathbf I\bigr)\,\d u
-
\sum_{n=0}^{M}
\frac{1}{2\pi}
\frac{\d^k}{\d x^k}
\left(
\int_{\Sigma_R}
\mathcal C_J^n[\mathbf I](\mu)
\bigl(J_R(\mu)-\mathbf I\bigr)\,\d\mu
\right)
\Bigg|
\nonumber\\
=&
\Bigg|
\frac{1}{2\pi}
\int_{\Sigma_R}
\frac{\d^k}{\d x^k}
\left[
\left(
R_-(\mu)
-
\sum_{n=0}^{M}
\mathcal C_J^n[\mathbf I](\mu)
\right)
\bigl(J_R(\mu)-\mathbf I\bigr)
\right]
\,\d\mu
\Bigg|
\nonumber\\
\leq&
\frac{1}{2\pi}
\sum_{\ell=0}^{k}
\binom{k}{\ell}
\left\|
\frac{\d^\ell}{\d x^\ell}
\left(
\sum_{n=M+1}^{\infty}
\mathcal C_J^n[\mathbf I](\cdot;x,a)
\right)
\right\|_{L^2(\Sigma_R)}
\left\|
\frac{\d^{k-\ell}}{\d x^{k-\ell}}
\bigl(J_R(\cdot;x,a)-\mathbf I\bigr)
\right\|_{L^2(\Sigma_R)}
\nonumber\\
=&
\frac{1}{2\pi}
\sum_{\ell=0}^{k}
\binom{k}{\ell}
\left\|
\sum_{n=M+1}^{\infty}
\frac{\d^\ell}{\d x^\ell}
\mathcal C_J^n[\mathbf I](\cdot;x,a)
\right\|_{L^2(\Sigma_R)}\left\|
\frac{\d^{k-\ell}}{\d x^{k-\ell}}
\bigl(J_R(\cdot;x,a)-\mathbf I\bigr)
\right\|_{L^2(\Sigma_R)}.
\label{0825estimate3}
\end{align}
In the above, the first inequality follows from Hölder's inequality, while the last equality follows from the standard theorem on term-by-term
differentiation of Banach-valued series, since the series
\begin{equation*}
\sum_{n=0}^{\infty}
\frac{\d^{\ell}}{\d x^{\ell}}\mathcal C_J^n[\mathbf I]
\end{equation*}
converges uniformly in $L^2(\Sigma_R;\mathbb{C}^{2\times 2})$ for $x\in(0,x_0]$ by Lemma
\ref{estimateforCJn},  where $x_0>0$ is sufficiently small. We next estimate \eqref{0825estimate3}. By
Lemma~\ref{estimateforCJn},
\begin{align}
&
\left\|
\sum_{n=M+1}^{\infty}
\frac{\mathrm d^\ell}{\mathrm d x^\ell}
\mathcal C_J^n[\mathbf I](\cdot;x,a)
\right\|_{L^2(\Sigma_R)}\nonumber
\\
\qquad\leq&
\sum_{n=M+1}^{\infty}
n^\ell c_{a,\ell}^{\,n}
\begin{cases}
\displaystyle
\max\{x,x^{a+1}\}^{\,n-\ell}
\left(1+x^{a+1-\ell}\right),
& a\notin\mathbb Z_{\geq0},
\\[1.2ex]
\displaystyle
\max\{x,x^{a+1}|\log x|\}^{\,n-\ell}
\left(1+x^{a+1-\ell}|\ln x|\right),
& a\in\mathbb Z_{\geq0}.
\end{cases}\label{0925}
\end{align}
Choosing $x_0>0$ sufficiently small so that
$
c_{a,\ell}
\max\{x,x^{a+1}|\ln x|\}
<\frac12$, $0<x\leq x_0$,
we obtain
\begin{align}
&
(\ref{0925})
\leq
C_{a,\ell,M}
\begin{cases}
\displaystyle
\max\{x,x^{a+1}\}^{\,M+1-\ell}
\left(1+x^{a+1-\ell}\right),
& a\notin\mathbb Z_{\geq0},
\\[1.2ex]
\displaystyle
\max\{x,x^{a+1}|\log x|\}^{\,M+1-\ell}
\left(1+x^{a+1-\ell}|\ln x|\right),
& a\in\mathbb Z_{\geq0}.
\end{cases}
\label{0825thetailestimate}
\end{align}
Combining this estimate with the estimate for
$J_R-\mathbf I$ and its derivatives in Proposition
\ref{proponDelta}, we further obtain
\begin{align}
(\ref{0825estimate3})
&\leq
b_{a,k,M}
\begin{cases}
\displaystyle
\max\{x,x^{a+1}\}^{\,M+2-k}
\left(1+x^{a+1-k}\right),
& a\notin\mathbb Z_{\geq0},
\\[1.2ex]
\displaystyle
\max\{x,x^{a+1}|\log x|\}^{\,M+2-k}
\left(1+x^{a+1-k}|\ln x|\right),
& a\in\mathbb Z_{\geq0}.
\end{cases}
\label{thefinaltail}
\end{align}
In fact, the proof of \eqref{thefinaltail} is divided into two
cases: For $0\leq \ell\leq k-1$, the estimate follows directly from the
estimate for the Neumann tail \eqref{0825thetailestimate} and Proposition
\ref{proponDelta}. For the case $\ell=k$, the non-differentiated term
$
\|J_R(\cdot;x,a)-\mathbf I\|_{L^2(\Sigma_R)}
$
provides an additional factor
$
\max\{x,x^{a+1}\}
$
when $a\notin\mathbb Z_{\ge0}$, or
$
\max\{x,x^{a+1}|\ln x|\}
$
when $a\in\mathbb Z_{\ge0}$, respectively. Hence, this contribution is dominated by the right-hand side of
\eqref{thefinaltail}. Moving ahead, by \eqref{seriesforM}, for $a>-1$ with
$a\notin\mathbb Z_{\geq 0}$, we have
\begin{align}\label{seriesexpansionforM2}
\sum_{n=0}^{M}\frac{\i}{2\pi}
\int_{\Sigma_R}
\mathcal C_J^n[\mathbf I](\mu)
\bigl(J_R(\mu)-\mathbf I\bigr)\,\d\mu
&=
\sum_{m=1}^{\infty}
\sum_{n=0}^{\min\{m-1,M\}}
\mathsf b_m^{(n)}(a)x^m
+
\sum_{\ell=1}^{M+1}
\sum_{m=\ell}^{\infty}
\sum_{n=\ell-1}^{\min\{m-1,M\}}
\mathsf b_{m,\ell}^{(n)}(a)x^{m+a\ell},
\end{align}
where
$
\mathsf b_m^{(n)}(a),\,
\mathsf b_{m,\ell}^{(n)}(a)
=
\mathcal O_{n,a}\!\left(m^{2n}\epsilon^{-m}\right)
$
for some $\epsilon>0$.
After differentiating $k$ times with respect to $x$ in \eqref{seriesexpansionforM2}, the infinite
series on the right-hand side of \eqref{seriesexpansionforM2} converges
uniformly for $x\in(0,x_0]$, provided that $x_0>0$ is sufficiently
small. Moreover, its $k$-th derivative is bounded by
\[
\mathcal O_{a,k,M}(1)
+
\mathcal O_{a,k,M}\!\left(x^{a+1-k}\right),\ \ \ a\notin\mathbb{Z}_{\geq 0}.
\]
Similarly, for $a\in\mathbb Z_{\ge0}$, the corresponding bound is
\[
\mathcal O_{a,k,M}(1)
+
\mathcal O_{a,k,M}\!\left(
x^{a+1-k}|\ln x|
\right).
\]
Finally, combining the last estimate with \eqref{thefinaltail} and setting $M=k$, we obtain the conclusion of the current Lemma.
\end{proof}

\section{Derivative Estimates for the Error Matrix at the Origin}\label{esonerrormatrix}

The aim of this section is to establish derivative estimates at the
origin for the error matrix $\widetilde R$ associated with
RHP~\ref{modelproblem2} below. Note that $\widetilde R$ is the final
error matrix arising from a large-$N$ steepest-descent analysis, with jump
contour $\Sigma_{\widetilde R}$ and jump matrix
$J_{\widetilde R}$ given in \eqref{eq:JR_definition}. These estimates are
summarized in Lemma~\ref{lem:derivative-Rtilde-at-zero} and will be
used in the proof of Proposition~\ref{representation}. Throughout, for a $2\times2$ matrix $A$, $\|A\|$ denotes
the operator norm induced by the Euclidean norm on $\mathbb C^2$.

The large-\(\xi\) asymptotic expansion in RHP~\ref{modelproblem}(c)
is uniform only in fixed closed subsectors staying a positive
angular distance away from the jump rays. In the matching analysis
below, however, the variable \(\xi\) ranges over the image of the
matching contour and may approach the jump rays arbitrarily closely.
Thus the sectorial estimate alone is not sufficient to obtain
uniform bounds for the jump perturbation \(\Delta_{\widetilde R}\).
Such uniform control, together with the corresponding estimates for
the \(x\)-derivatives of the model solution, is needed in deriving the
\(L^\infty\)- and \(L^2\)-estimates for
\(\Delta_{\widetilde R}\) and its \(x\)-derivatives; see
Lemma~\ref{lem:derivative-Delta-Rtilde} for details.

We therefore first upgrade the sectorial large-\(\xi\) estimate to
one that is uniform throughout the whole cut plane for each fixed
\(x>0\). The following lemma shows that the
\(\mathcal O(\xi^{-2})\) remainder, initially known only uniformly in
fixed closed subsectors, is in fact uniform up to the jump contours
outside a sufficiently large disk. Its proof uses the jump structure
of the normalized model solution, together with the \(\xi\)-part of
the Lax pair and the Phragm\'en--Lindel\"of principle.

\begin{lem}\label{lem:uniform-R2-fixed-x}
Let $a>-1$, and let $\Psi(\xi;x,a)$ be as defined in
RHP~\ref{modelproblem}. 
For
\[
\xi\in \mathbb{C}\setminus(\Sigma_1\cup\Sigma_2\cup\Sigma_3),
\qquad
-\pi<\arg\xi<\pi,
\]
define
\begin{equation}\label{eq:def-widehat-Psi}
\widehat{\Psi}(\xi;x,a)
:=
\Psi(\xi;x,a)
\e^{-\xi^{\frac{1}{2}}\sigma_3}
\left(
\frac{\textbf{I}+\mathrm{i}\sigma_1}{\sqrt2}
\right)^{-1}
\xi^{\sigma_3/4},
\end{equation}
and
\begin{equation}\label{eq:def-R2}
\mathcal R_2(\xi;x,a)
:=
\widehat{\Psi}(\xi;x,a)
-\textbf{I}-\frac{C_1(x)}{\xi},
\end{equation}
where $C_1(x)$ is the coefficient appearing in the large-$\xi$
expansion in RHP~\ref{modelproblem}(c). 
Let $x_*>0$ be fixed. Suppose that
\[
R_2(\xi;x_*,a)
=
\mathcal O\big(|\xi|^{-2}\big),
\qquad \xi\to\infty,
\]
uniformly in every fixed closed subsector of each connected
component of
\[
\big\{\xi\in\mathbb C\setminus(\Sigma_{1}\cup \Sigma_{2}\cup \Sigma_{3}):-\pi<\arg\xi<\pi\big\}.
\]
Then there exist constants $C>0$ and $\xi_*>0$ such that
\begin{equation}\label{eq:uniform-R2-fixed-x}
\|\mathcal R_2(\xi;x_*,a)\|
\leq
\frac{C}{|\xi|^2},
\qquad \forall\,|\xi|\geq \xi_*,
\end{equation}
uniformly for
\[
\xi\in\mathbb C\setminus(\Sigma_{1}\cup \Sigma_{2}\cup \Sigma_{3}),
\qquad
-\pi<\arg\xi<\pi.
\]
\end{lem}

\begin{proof}
We first consider $\Sigma_1$ and $\Sigma_3$. Since the principal
branches of $\xi^{1/2}$ and $\xi^{1/4}$ have no jump across these
two rays, we have
\[
\widehat\Psi_+(\xi;x,a)
=
\widehat\Psi_-(\xi;x,a)\widehat J_j(\xi;a),\quad j=1,3,
\]
where
\[
\widehat J_j(\xi;a)
:=
\xi^{-\sigma_3/4}
\left(
\frac{\textbf{I}+\mathrm{i}\sigma_1}{\sqrt2}
\right)\e^{\xi^{\frac{1}{2}}\sigma_3}
J_j(a)\,
\e^{-\xi^{\frac{1}{2}}\sigma_3}
\left(
\frac{\textbf{I}+\mathrm{i}\sigma_1}{\sqrt2}
\right)^{-1}\xi^{\sigma_3/4},
\] and for $j=1,3$, write
\[
J_j(a)=\textbf{I}+c_j\begin{pmatrix}
0 & 0\\
1 & 0
\end{pmatrix},
\qquad
c_1=\e^{\pi\mathrm{i}a},
\qquad
c_3=\e^{-\pi\mathrm{i}a}.
\]
A direct calculation gives
\begin{equation}\label{thb3}
\widehat J_j(\xi;a)
=
\textbf{I}+
\frac{c_j}{2}\e^{-2\sqrt{\xi}}
\begin{pmatrix}
\mathrm{i}&\xi^{-1/2}\\
\xi^{1/2}&-\mathrm{i}
\end{pmatrix}.
\end{equation}
Choose $\delta>0$ sufficiently small and let
\[
S_j(\delta)
=
\big\{\xi:|\arg\xi-\theta_j|<\delta\big\},
\qquad
\theta_1=\frac{2\pi}{3},
\quad
\theta_3=-\frac{2\pi}{3},
\]
so that $S_j(\delta)$ contains no jump ray other than $\Sigma_j$
and has opening $2\delta<\pi$. The right-hand side in \eqref{thb3} extends analytically to $S_j(\delta)$;
denote this extension by $\mathcal J_j$. Since
\[
\textnormal{Re}\big(\xi^{\frac{1}{2}}\big)\ge c_\delta\sqrt{|\xi|},
\qquad
\xi\in S_j(\delta),
\]
we have for large $|\xi|$
\[
\mathcal J_j^{\pm1}(\xi)
=
\textbf{I}+
\mathcal O_a\left(
\sqrt{|\xi|}\,
\e^{-2c_\delta\sqrt{|\xi|}}
\right),
\]
uniformly in $S_j(\delta)$. Define
\[
\widetilde\Psi_j(\xi;a)
=
\begin{cases}
\widehat\Psi(\xi;x_*,a),
& \xi \text{ on the $-$ side of }\Sigma_j,\\[1mm]
\widehat\Psi(\xi;x_*,a)\mathcal J_j(\xi;a)^{-1},
& \xi \text{ on the $+$ side of }\Sigma_j.
\end{cases}
\]
Then $\widetilde\Psi_j$ extends analytically across $\Sigma_j$.
Set
\[
F_j(\xi)
=
\xi^2
\left(
\widetilde\Psi_j(\xi;a)
-\textbf{I}-\frac{C_1(x_*)}{\xi}
\right).
\]
On the two boundary rays of $S_j(\delta)$, the hypothesis and the
exponential smallness of $\mathcal J_j^{\pm1}-\textbf{I}$ give
\[
F_j(\xi)=\mathcal O(1).
\]
Moreover, the $\xi$-part of the Lax pair \cite[$(1.14)$]{xu2014critical}, together with a standard
Gronwall estimate along radial segments, yields
\[
\|F_j(\xi)\|
\le C\e^{C|\xi|},
\qquad
\xi\in S_j(\delta),
\]
for sufficiently large $|\xi|$.
Since the opening of $S_j(\delta)$ is strictly smaller than $\pi$,
the Phragm\'en--Lindel\"of principle implies
\[
F_j(\xi)=\mathcal O(1)
\]
uniformly in $S_j(\delta)$. Hence
\[
\widetilde\Psi_j(\xi;a)
=
\textbf{I}+\frac{C_1(x_*)}{\xi}
+\mathcal O(\xi^{-2})
\]
uniformly in $S_j(\delta)$. Since
$\mathcal J_j-\textbf{I}$ is exponentially small, we conclude that
\[
\widehat\Psi(\xi;x_*,a)
=
\textbf{I}+\frac{C_1(x_*)}{\xi}
+\mathcal O(\xi^{-2})
\]
uniformly for
$
\xi\in S_j(\delta)\setminus\Sigma_j, j=1,3.
$ It remains to consider $\Sigma_2$. For $\xi=-r$, $r>0$, one has
\[
(\xi_+)^{\frac{1}{2}}=\mathrm{i}\sqrt r,
\qquad
(\xi_{-})^{\frac{1}{2}}=-\mathrm{i}\sqrt r,
\]
and
\[
\left(
\frac{\textbf{I}+\mathrm{i}\sigma_1}{\sqrt2}
\right)J_2\left(
\frac{\textbf{I}+\mathrm{i}\sigma_1}{\sqrt2}
\right)^{-1}=-\mathrm{i}\sigma_3,
\qquad
\xi_-^{-\sigma_3/4}\xi_+^{\sigma_3/4}
=\mathrm{i}\sigma_3.
\]
Consequently,
 \begin{align*}
\widehat J_2(\xi)
=
 \xi_-^{-\sigma_3/4}
\left(
\frac{\textbf{I}+\mathrm{i}\sigma_1}{\sqrt2}
\right)\e^{\sqrt{\xi}_-\sigma_3}
J_2
\e^{-\sqrt{\xi}_+\sigma_3}
\left(
\frac{\textbf{I}+\mathrm{i}\sigma_1}{\sqrt2}
\right)^{-1}\xi_+^{\sigma_3/4}
=\textbf{I}.
\end{align*}
Thus $\widehat\Psi$ extends analytically through $\Sigma_2$ away
from the origin. Choose $0<\delta<\frac{\pi}{3}$ and define
\[
S_2(\delta)
=
\big\{
\xi:
\pi-\delta<\arg\xi<\pi+\delta
\big\}.
\]
Its two boundary rays lie in fixed closed subsectors of the
original components of $\mathbb C\setminus\Sigma$. Therefore the
hypothesis gives
\[
\widehat\Psi(\xi;x_*,a)
=
\textbf{I}+\frac{C_1(x_*)}{\xi}
+\mathcal O(\xi^{-2})
\]
on both boundary rays. Applying the same Phragm\'en--Lindel\"of
argument to
\[
F_2(\xi)
=
\xi^2
\left(
\widehat\Psi(\xi;x_*,a)
-\textbf{I}-\frac{C_1(x_*)}{\xi}
\right)
\]
gives
\[
F_2(\xi)=\mathcal O(1)
\]
uniformly in $S_2(\delta)$. We have therefore obtained the required estimate in fixed
neighborhoods of all three jump rays. The complement of these
neighborhoods in
\[
-\pi<\arg\xi<\pi
\]
is a finite union of fixed closed subsectors, where the estimate
follows directly from the hypothesis. Hence
\[
\mathcal R_2(\xi;x_*,a)
=
\mathcal O(|\xi|^{-2})
\]
uniformly for
\[
\xi\in\mathbb C\setminus(\Sigma_{1}\cup\Sigma_2\cup\Sigma_3),
\qquad
-\pi<\arg\xi<\pi,
\]
as claimed.
\end{proof}
Lemma~\ref{lem:uniform-R2-fixed-x} shows that, for each fixed
$x>0$, the sectorial large-$\xi$ estimate can be extended uniformly
up to the jump rays. We next strengthen this result by making the
estimate uniform as $x\to0$ and by deriving the corresponding
bounds for the $x$-derivatives of the remainder.
Before stating our next result, we need the following preliminary statement. For ease of presentation, we introduce the
following notation. 
Let \(a>-1\) and $k\in\mathbb{Z}_{\geq 0}$. For $x>0$, define
\begin{equation}\label{definieofpho}
\rho_{a,k}(x)
:=
\begin{cases}
1+x^{a+1-k},
& a\notin\mathbb Z_{\geq0},
\\[1.2ex]
1+x^{a+1-k}|\ln x|,
& a\in\mathbb Z_{\geq0}.
\end{cases}
\end{equation}

\begin{rmk}\label{rem:product-estimate-rho2}
Let $p,q\in\mathbb{Z}_{\geq 1}$. Let $x\in (0,1)$. It is not hard to check that
\begin{equation}\label{eq:product-estimate-rho2}
\rho_{a,p}(x)\rho_{a,q}(x)
=
\mathcal O_{a,p,q}
\left(
\rho_{a,p+q}(x)
\right),
\qquad
a>-1.
\end{equation}
Moreover, if $a>0$, we have
\begin{equation}\label{eq:product-estimate-rho}
\rho_{a,p}(x)\rho_{a,q}(x)
=
\mathcal O_{a,p,q}
\left(
\rho_{a,p+q-1}(x)
\right).
\end{equation}
\end{rmk}
The proof of the following estimates uses the boundary condition for
$v(x)$ as $x\downarrow 0$ given in Lemma~\ref{asyzero}.
\begin{prop}\label{prop:q-t-derivative-estimates}
Let $a>-1$. Let $v(x)$ be defined as in \eqref{36}, and define
\begin{equation}\label{eq:def-q}
q(x)
:=
-xv'(x)
+\frac{1}{2}v(x)
+\frac{1}{2}v(x)^2.
\end{equation}
Moreover, let $\mathfrak t(x)$ be determined by
(cf. \cite[(5.20)]{xu2014critical})
\begin{equation}\label{eq:def-t-prime}
\mathfrak t'(x)
=
\frac{1-q'(x)^2}{v'(x)},
\qquad
\mathfrak t(0)
=
\frac{
(4a^2-1)(4a^2-9)(4a^2-13)
}{1536}.
\end{equation}
Then, for every fixed $m\in\mathbb Z_{\geq0}$, there exists
$x_0=x_0(a,m)\in(0,1)$ such that
\begin{equation}\label{eq:q-derivative-estimate}
q^{(m)}(x)
=
\mathcal O_{a,m}
\left(
\rho_{a,m}(x)
\right),
\end{equation}
and
\begin{equation}\label{eq:tprime-derivative-estimate}
\frac{\d^m}{\d x^m}\mathfrak t'(x)
=
\mathcal O_{a,m}
\left(
\rho_{a,m+1}(x)
\right),
\end{equation}
uniformly for $0<x\leq x_0$.
\end{prop}
\begin{proof}
Differentiating \eqref{eq:def-q}, we obtain
\begin{equation}\label{eq:qprime}
q'(x)
=
-xv''(x)
-\frac{1}{2}v'(x)
+v(x)v'(x).
\end{equation}
The estimate \eqref{eq:q-derivative-estimate} follows from
Lemma~\ref{asympoticsforv(x)forsmallx}, Leibniz's rule, and
Remark~\ref{rem:product-estimate-rho2}. It remains to estimate the derivatives of $\mathfrak t'(x)$. We
distinguish three cases. For $a>0$, Lemma~\ref{asyzero} gives
\[
v(x)
=
\frac{1-4a^2}{8}
+\frac{x}{a}
+\mathcal O(x^2),
\qquad x\downarrow0,
\]
so that
\[
v'(0)=\frac1a\neq0.
\]
Hence, after decreasing $x_0$ if necessary, $v'(x)$ is bounded away
from zero on $(0,x_0]$. Together with
Lemma~\ref{asympoticsforv(x)forsmallx}, repeated differentiation of
$(v')^{-1}$, Leibniz's rule, and
Remark~\ref{rem:product-estimate-rho2}, this yields
\[
\frac{\d^m}{\d x^m}
\mathfrak{t}'(x)
=
\mathcal O_{a,m,x_0}
\left(
\rho_{a,m+1}(x)
\right).
\]
For $a\in(-1,0)$, the asymptotic expansions of $v$ in
Lemma~\ref{asyzero} and Lemma~\ref{asympoticsforv(x)forsmallx}
yield, for $k\in\mathbb{Z}_{\geq0}$,
\[
v'(x)\asymp_a x^a,
\qquad
\left(\frac{1}{v'(x)}\right)^{(k)}
=
\mathcal{O}_{a,k,x_0}\!\left(x^{-a-k}\right),
\qquad
q^{(k+1)}(x)
=
\mathcal{O}_{a,k,x_0}\!\left(x^{a-k}\right)
=
\mathcal{O}_{a,k,x_0}\!\left(\rho_{a,k+1}(x)\right).
\]
Here, the notation
$
f(x)\asymp_a g(x)
$
means that there exist constants \(c_a,C_a>0\), depending only on
\(a\), such that
\[
c_a |g(x)|\le |f(x)|\le C_a |g(x)|
\]
throughout the range under consideration. Consequently,
\[
\frac{\d^j}{\d x^j}\bigl(1-q'(x)^2\bigr)
=
\mathcal{O}_{a,j,x_0}\!\left(x^{2a-j}\right),
\qquad j\in\mathbb{Z}_{\geq0}.
\]
Hence, by Leibniz's rule, for $m\in\mathbb{Z}_{\geq0}$,
\bea\label{0828equation2}
\frac{\d^m}{\d x^m}
\mathfrak t'(x)
=
\mathcal{O}_{a,m,x_0}\!\left(x^{a-m}\right)
=
\mathcal{O}_{a,m,x_0}\!\left(
\rho_{a,m+1}(x)
\right),
\eea
uniformly for $x\in(0,x_0]$, where $x_0>0$ is fixed and sufficiently small.
Lastly, for $a=0$, setting
$
L(x):=1+|\ln x|,
$
Lemma~\ref{asyzero} and Lemma~\ref{asympoticsforv(x)forsmallx} yield,
for $k\in\mathbb{Z}_{\geq0}$,
\[
v'(x)\asymp L(x),
\qquad
\left(\frac{1}{v'(x)}\right)^{(k)}
=
\mathcal{O}_{k,x_0}\!\left(\frac{x^{-k}}{L(x)}\right),
\qquad
q^{(k+1)}(x)
=
\mathcal{O}_{k,x_0}\!\left(x^{-k}L(x)\right)
=
\mathcal{O}_{k,x_0}\!\left(\rho_{0,k+1}(x)\right).
\]
Hence, by Leibniz's rule, for $m\in\mathbb{Z}_{\geq0}$,
\begin{equation}\label{0828eq3}
\frac{\d^m}{\d x^m}
\mathfrak t'(x)
=
\mathcal{O}_{m,x_0}\!\left(x^{-m}L(x)\right)
=
\mathcal{O}_{m,x_0}\!\left(
\rho_{0,m+1}(x)
\right),
\end{equation}
uniformly for $0<x\leq x_0$, with $x_0>0$ sufficiently small. By choosing $x_0=x_0(a,m)>0$ sufficiently small, all the above
estimates hold simultaneously on $(0,x_0]$. Since $x_0$ is fixed
in terms of $a$ and $m$, the implied constants may be regarded as
depending only on $a$ and $m$. This completes the proof.
\end{proof}

\begin{lem}\label{lem:uniform-R2-zero-order}
Let $a>-1$, and let $\Psi(\xi;x,a)$ be as defined in
RHP~\ref{modelproblem}. Let $\widehat{\Psi}(\xi;x,a)$ and $\mathcal R_2(\xi;x,a)$ be defined as in
\eqref{eq:def-widehat-Psi} and \eqref{eq:def-R2}, respectively.
Then, for any integer $k\geq0$, there exists a constant
$x_0=x_0(a,k)\in(0,1)$ such that, as $|\xi|\to\infty$,
the following estimates hold uniformly for
\[
0<x\leq x_0,
\qquad
\xi\in \mathbb{C}\setminus(\Sigma_1\cup\Sigma_2\cup\Sigma_3),
\qquad
-\pi<\arg\xi<\pi.
\]
For $k=0$,
\begin{equation}\label{k=0}
\mathcal R_2(\xi,x)
=
\mathcal O_a
\left(
\frac{\rho_{a,0}(x)}{|\xi|^2}
\right),
\end{equation}
and, for every $k\geq1$,
\begin{equation}\label{generalk}
\frac{\partial^k}{\partial x^k}
\mathcal R_2(\xi,x)
=
\mathcal O_{a,k}
\left(
\frac{\rho_{a,k}(x)}{|\xi|^2}
\right),
\end{equation}
where $\rho_{a,k}(x)$ is defined in
\eqref{definieofpho}.
\end{lem}
\begin{proof}
Fix $k\in\mathbb Z_{\geq0}$. We choose
$x_0=x_0(a,k)>0$ sufficiently small so that all the estimates
required below up to order $k$ hold uniformly on $(0,x_0]$.
Since $x_0$ depends only on $a$ and $k$, any constant depending
on $a$, $k$, and $x_0$ will henceforth be regarded as a constant
depending only on $a$ and $k$. Note that the factors
\[
\e ^{-\xi^{\frac{1}{2}}\sigma_3},
\qquad
\Bigg(\frac{\textbf{I}+\mathrm{i}\sigma_1}{\sqrt2}\Bigg)^{-1},
\qquad
\xi^{\sigma_3/4},\ \ \ \ \ \ \ \ \ \xi\in\mathbb{C}\setminus(-\infty,0],
\]
in \eqref{eq:def-widehat-Psi} are independent of \(x\). By \cite[$(1.14)$]{xu2014critical}, with $C_1(x)$ as in condition (c) of RHP \ref{modelproblem},
\begin{equation}\label{eq:x-Lax-widehat-Psi}
	\frac{\partial}{\partial x}\widehat{\Psi}(\xi;x,a)=\frac{1}{\xi}C_1'(x)\widehat{\Psi}(\xi;x,a),
\end{equation}
and substituting \eqref{eq:def-R2} into
\eqref{eq:x-Lax-widehat-Psi}, we obtain
\[
	\frac{\partial}{\partial x}\mathcal{R}_2(\xi;x,a)+\frac{1}{\xi}C_1'(x)=\frac{1}{\xi}C_1'(x)\bigg(\mathbf{I}+\frac{C_1(x)}{\xi}+\mathcal{R}_2(\xi;x,a)\bigg),
%
%
\]
so that, after simplification,
\begin{equation}\label{eq:R2-exact-x-equation}
	\frac{\partial}{\partial x}\mathcal{R}_2(\xi;x,a)=\frac{1}{\xi}C_1'(x)\mathcal{R}_2(\xi;x,a)+\frac{1}{\xi^2}C_1'(x)C_1(x).
%
\end{equation}
Fix a sufficiently small $x_0>0$. By
Lemma~\ref{lem:uniform-R2-fixed-x}, there exist constants $C>0$ and
$\xi_0>0$ such that
\begin{equation}\label{eq:R2-reference-bound}
\|\mathcal R_2(\xi,x_0)\|
\leq
\frac{C}{|\xi|^2},
\qquad |\xi|\geq\xi_0,
\end{equation}
uniformly for
\[
-\pi<\arg\xi<\pi,
\qquad
\xi\notin\Sigma_1\cup\Sigma_2\cup\Sigma_3.
\]
Integrating \eqref{eq:R2-exact-x-equation} from \(x\in(0,x_0]\) to
\(x_0\), we find
\begin{align}
\mathcal R_2(\xi;x,a)
=
\mathcal R_2(\xi;x_0,a)
-
\frac{1}{\xi}
\int_x^{x_0}
C_1'(u)\mathcal R_2(\xi;u,a)\,\mathrm du
-
\frac{1}{\xi^2}
\int_x^{x_0}
C_1'(u)C_1(u)\,\mathrm du.
\label{eq:R2-integral-equation}
\end{align}
For fixed \(\xi\in\mathbb{C}\setminus(\Sigma_1\cup\Sigma_2\cup\Sigma_3)\), set
\[
F(x,a)
:=
|\xi|^2\|\mathcal R_2(\xi;x,a)\|.
\]
It follows from \eqref{eq:R2-reference-bound} and \eqref{eq:R2-integral-equation} that
\begin{align}
F(x,a)
\leq{}&
C
+
\frac{1}{|\xi|}
\int_x^{x_0}
\|C_1'(u)\|\,F(u,a)\,\mathrm du
+
\int_x^{x_0}
\|C_1'(u)\|\cdot\|C_1(u)\|\,\mathrm du.
\label{eq:F-Gronwall}
\end{align}
Next,
$
\sup_{0<u\leq x_0}\rho_{a,0}(u)<\infty,
$
for \(a+1>0\) and so both
$
u^{a+1}$ and $
u^{a+1}|\log u|$ are bounded on \((0,x_0]\).
Moreover, by \cite[Eq.~(2.5)]{xu2014critical}, with the function
\(v\) in \eqref{36} denoted by \(r\) in the cited reference,
\begin{equation}\label{explicitforC1}
C_1(x)
=
\begin{pmatrix}
q(x)&-\mathrm{i}v(x)\smallskip\\
\mathrm{i}\mathfrak t(x)&-q(x)
\end{pmatrix},
\end{equation}
where $v(x)$, $q(x)$, and $\mathfrak{t}(x)$ are as in
\eqref{36}, \eqref{eq:def-q}, and \eqref{eq:def-t-prime},
respectively.
Hence, by Lemma \ref{asympoticsforv(x)forsmallx} and Proposition \ref{prop:q-t-derivative-estimates}, 
\[
\|C_1(x)\|
=
\mathcal O_{a,x_{0}}\bigl(\rho_{a,0}(x)\bigr),
\qquad
\|C_1'(x)\|
=
\mathcal O_{a,x_{0}}\bigl(\rho_{a,1}(x)\bigr),
\qquad
\forall\,0<x\leq x_0.
\]
Therefore,
\begin{equation}\label{eq:C1-uniform-bound}
\sup_{0<u\leq x_0}\|C_1(u)\|
\leq B_{a,x_0},\ \ \ \ \ \ 
\int_0^{x_0}
\|C_1'(u)\|\,\mathrm du
\leq \widetilde{B}_{a,x_0},
\end{equation}
and consequently we obtain
\[
\int_x^{x_0}
\|C_1'(u)\|\cdot\|C_1(u)\|\,\mathrm du
\leq B_{a,x_0}\widetilde{B}_{a,x_0},
\qquad
\forall\,0<x\leq x_0.
\]
Returning to \eqref{eq:F-Gronwall} we obtain,
%
%
%
\[
F(x,a)
\leq
\widetilde{C}_{a,x_0}
+
\frac{1}{|\xi|}
\int_x^{x_0}
\|C_1'(u)\|\,F(u,a)\,\mathrm du;\ \ \ \ \ \ \widetilde{C}_{a,x_0}:=C+B_{a,x_0}\widetilde{B}_{a,x_0},
\]
and thus with $G(x,a):=\sup_{x\leq u \leq x_0}F(u,a)$,
\begin{equation*}
	F(y,a)\leq\widetilde{C}_{a,x_0}+\frac{1}{|\xi|}\widetilde{B}_{a,x_0}G(x,a)\ \ \ \ \forall\,y\in[x,x_0]
\end{equation*}
and taking the supremum over $y\in[x,x_0]$, therefore
\begin{equation*}
	G(x,a)\leq\widetilde{C}_{a,x_0}+\frac{1}{|\xi|}\widetilde{B}_{a,x_0}G(x,a).
\end{equation*}
Thus, after increasing $\xi_0$, if necessary, so that
$
\widetilde{B}_{a,x_0}\leq 1/(2\xi_0),
$
we have, for every $|\xi|\geq \xi_0$,
$
\widetilde{B}_{a,x_0}
\leq 1/(2|\xi|).
$
Then $G(x,a)\leq 2\widetilde{C}_{a,x_0}$ for all $0<x\leq x_0$ and consequently
\[
\|\mathcal R_2(\xi;x,a)\|
\leq
\frac{2\widetilde{C}_{a,x_{0}}}{|\xi|^2},
\qquad\forall\,
0<x\leq x_{0},
\]
uniformly for
\[
|\xi|\geq\xi_0,
\qquad
\xi\notin\Sigma_1\cup\Sigma_2\cup\Sigma_3,
\qquad
-\pi<\arg\xi<\pi.
\] Finally, since
$
\rho_{a,0}(x)\geq1,
$
we conclude that
\[
\|\mathcal R_2(\xi,x)\|
\leq
C_{a,x_0}
\frac{\rho_{a,0}(x)}{|\xi|^2},
\qquad
\forall\,0<x\leq x_0,
\]
which proves \eqref{k=0}. Since $x_0=x_0(a,0)$ has been fixed depending only on $a$,
the dependence on $x_0$ may be absorbed into the constant.
Hence, (\ref{k=0}) holds uniformly for $0<x\leq x_{0}$.

For fixed $k\in\mathbb{N}$ we now prove the estimates for $0\leq j\leq k$ by induction on $j$, noting that $j=0$ has just been established. Hence, assume that, for every
\(0\leq j\leq k-1\), 
\begin{equation}\label{eq:induction-remainder}
\frac{\partial^j}{\partial x^j}\mathcal R_{2}(\xi;x,a)
=
\mathcal O_{a,j,x_0}
\left(
\frac{\rho_{a,j}(x)}{|\xi|^{2}}
\right).
\end{equation}
Applying \(\partial_x^{k-1}\) to
\eqref{eq:R2-exact-x-equation} and using the Leibniz
rule gives
\begin{align}
\frac{\partial^k}{\partial x^k}\mathcal R_{2}(\xi;x,a)
={}&
\frac{1}{\xi}
\sum_{\ell=0}^{k-1}
\binom{k-1}{\ell}
C_1^{(\ell+1)}(x)
\frac{
\partial^{k-1-\ell}
}{
\partial x^{k-1-\ell}}\mathcal{R}_2(\xi;x,a)
\nonumber\\
&+
\frac{1}{\xi^{2}}
\sum_{\ell=0}^{k-1}
\binom{k-1}{\ell}
C_1^{(\ell+1)}(x)
C_1^{(k-1-\ell)}(x).
\label{eq:kth-remainder-identity}
\end{align}
For the second sum in
\eqref{eq:kth-remainder-identity}, Lemma \ref{asympoticsforv(x)forsmallx} implies that
\[
\begin{aligned}
&
C_1^{(\ell+1)}(x)
C_1^{(k-1-\ell)}(x)
=
\mathcal O_{a,L,k,x_0}
\left(
\rho_{a,\ell+1}(x)
\rho_{a,k-1-\ell}(x)
\right)=
\mathcal O_{a,k,x_0}
\bigl(
\rho_{a,k}(x)
\bigr).
\end{aligned}
\]
For the first sum, the induction hypothesis gives
\[
\begin{aligned}
&
C_1^{(\ell+1)}(x)
\frac{
\partial^{k-1-\ell}
}{
\partial x^{k-1-\ell}
}\mathcal{R}_2(\xi;x,a)=
\mathcal O_{a,k,x_0}
\left(
\frac{
\rho_{a,\ell+1}(x)
\rho_{a,k-1-\ell}(x)
}{
|\xi|^{2}
}\right)=
\mathcal O_{a,k,x_0}
\left(
\frac{\rho_{a,k}(x)}{|\xi|^{2}}
\right).
\end{aligned}
\]
Hence, we obtain \eqref{generalk} for $k$. The proof is complete.
\end{proof}

With the aid of a global parametrix
and two local parametrices
in the neighborhoods $z=0$ and $z=1$,
respectively, we now recall the ratio transformation \cite[$(98)$]{chen2019riemann}: let $S(z)$ be defined as in 
\cite[Equation~(52)]{chen2019riemann}, with its jump contour denoted
\[
\Sigma_S=\bigcup_{k=1}^{3}\Sigma_{S,k}
\]
as depicted in \cite[Figure~1]{chen2019riemann}. Take
\(P^{(\infty)}(z)\), \(P^{(1)}(z)\), and \(P^{(0)}(z)\)
as defined in \cite[Equations~(57), (66), and~(82)]
{chen2019riemann}, respectively. Throughout this section, \(r>0\) is fixed and sufficiently small
so that the disks \(U(0,r)\) and \(U(1,r)\) are disjoint and the
corresponding local parametrices are well defined. We define, see \cite[$(98)$]{chen2019riemann},
\begin{equation}\label{eq:def_R}
\widetilde R(z)
=S(z)
\begin{cases}
(P^{(\infty)}(z))^{-1},
&
z\in
\mathbb{C}
\setminus
\left(
U(1,r)\cup U(0,r)\cup\Sigma_S
\right),
\\[4pt]
(P^{(1)}(z))^{-1},
&
z\in U(1,r)\setminus\Sigma_S,
\\[4pt]
(P^{(0)}(z))^{-1},
&
z\in U(0,r)\setminus\Sigma_S.
\end{cases}
\end{equation}
and collect the relevant analytic and asymptotic properties of $z\mapsto\widetilde{R}(z)$ below.
\begin{problem}\label{modelproblem2}
The matrix-valued function \(\widetilde R(z)\) satisfies
the following Riemann--Hilbert problem.
\begin{enumerate}
\item[$(a)$]
\(\widetilde R(z)\) is analytic in $\mathbb{C}\setminus\Sigma_{\widetilde{R}}\ni z$, with $\Sigma_{\widetilde{R}}$ shown in \cite[Figure $5$]{chen2019riemann}
%
\item[$(b)$]
\(\widetilde R(z)\) satisfies the jump condition
\begin{equation}\label{eq:R_jump}
\widetilde R_+(z)
=
\widetilde R_-(z)J_{\widetilde R}(z),
\qquad
z\in\Sigma_{\widetilde R},
\end{equation}
with jump matrix \(J_{\widetilde R}(z)\) equal to
\begin{equation}\label{eq:JR_definition}
J_{\widetilde R}(z)
=
\begin{cases}
P^{(\infty)}(z)
J_S(z)
(P^{(\infty)}(z))^{-1},
&
z\in\Sigma_{S,1}\cup\Sigma_{S,3}\setminus(U(0,r)\cup U(1,r)),
\\[4pt]
P^{(0)}(z)
(P^{(\infty)}(z))^{-1},
&
z\in\partial U(0,r),
\\[4pt]
P^{(1)}(z)
(P^{(\infty)}(z))^{-1},
&
z\in\partial U(1,r).
\end{cases}
\end{equation}
Here, \(J_S(z)\) denotes the jump matrix for \(S(z)\)
given in \cite[Equations~(53) and~(54)]
{chen2019riemann}.
\item[$(c)$]
As \(z\to\infty\),
\[
\widetilde R(z)
=
\textbf{I}+\mathcal{O}(z^{-1}).
\]
\end{enumerate}
\end{problem}

\begin{rmk}
We remark that, when
\[
z\in\Sigma_{S,1}\cup\Sigma_{S,3}\setminus\big(U(0,r)\cup U(1,r)\big)\ \ \ 
\qquad\text{or}\qquad
z\in\partial U(1,r),
\]
the jump matrix \(J_{\widetilde R}(z)\) is independent of
\(x\). So, by
\cite[$(101)$]{chen2019riemann},
\begin{equation}\label{matcheasy}
J_{\widetilde R}(z)
=
\begin{cases}
\textbf{I}+\mathcal{O}(N^{-1}),
&
z\in\partial U(1,r),
\\[4pt]
\textbf{I}+\mathcal{O}(\e^{-cN}),
&
z\in\Sigma_{S,1}\cup\Sigma_{S,3},
\end{cases}
\end{equation}
where the \(\mathcal{O}\)-terms are independent of \(x\).
\end{rmk}
In order to investigate the $x$-dependency of $J_{\widetilde{R}}(z)$ on $\partial U(0,r)$ we recall the function $\varphi(z)=2z-1+2(z(z-1))^{\frac{1}{2}}$ in \cite[$(47)$]{chen2019riemann} together with, see \cite[$(75)$]{chen2019riemann},
%
%
\begin{equation}\label{defforf0}
f_0(z)
=
(g(z)-g(0))^2
=
\left(
\log\frac{\varphi(z)}{\varphi(0)}
\right)^2,\ \ \ \ g(z)=\ln\varphi(z),\ \ z\in\mathbb{C}\setminus(-\infty,1]
\end{equation}
In turn, 
\begin{equation}\label{esonf0z}
f_0(z)
=
-4z
\left(
1+\frac{1}{3}z+\mathcal{O}(z^2)
\right)
\end{equation}
as \(z\to0\).
Denote also
\begin{equation}\label{def:D0}
\mathcal D_0(z)
:=
(-1)^{N\sigma_3}
\widetilde W(z)^{\sigma_3}
\e^{\frac{\pi\mathrm{i}}{2}\sigma_3},\ \ \ z\in\mathbb{C}\setminus\mathbb{R},
\end{equation}
where
\begin{equation}\label{def:W_tilde}
\widetilde W(z)
=
\begin{cases}
\e^{-\frac{\alpha\pi\mathrm{i}}{2}}
\left(z^a(1-z)^b\right)^{\frac12},
&
\operatorname{Im}z>0,
\\[6pt]
\e^{\frac{\alpha\pi\mathrm{i}}{2}}
\left(z^a(1-z)^b\right)^{\frac12},
&
\operatorname{Im}z<0.
\end{cases}
\end{equation}
is defined in
\cite[Equation~(70)]{chen2019riemann}.
Lastly, denote
\begin{equation}\label{defforEP}
\mathcal E(\xi,2N^2t)
:=
\xi^{\frac14\sigma_3}
\Psi(\xi;2N^2t,a)
\e^{-\xi^{\frac{1}{2}}\sigma_3}
\left(
\frac{\textbf{I}+\mathrm{i}\sigma_1}{\sqrt2}
\right)^{-1},
\end{equation}
where \(\Psi(\xi;2N^2t,a)\) solves the model
Riemann--Hilbert problem \ref{modelproblem} with $x=2N^2t>0$.

\begin{prop}\label{aboutJR}
Let \(\mathcal D_0(z)\) and
\(\mathcal E(\xi,2N^2t)\) be defined as in
\eqref{def:D0} and \eqref{defforEP}, respectively.
For \(z\in\partial U(0,r)\), we have
\begin{align}\label{smallneighbourhood}
J_{\widetilde R}(z,2N^2t)
={}&
P^{(\infty)}(z)
\mathcal D_0(z)
\left(
\frac{\textbf{I}+\mathrm{i}\sigma_1}{\sqrt2}
\right)^{-1}
\nonumber\\
&\times
\mathcal E(N^2f_0(z),2N^2t)
\left(
\frac{\textbf{I}+\mathrm{i}\sigma_1}{\sqrt2}
\right)
\mathcal D_0(z)^{-1}
(P^{(\infty)}(z))^{-1},
\end{align}
where \(f_0(z)\) is defined in \eqref{defforf0}.
Moreover, \(P^{(\infty)}(z)\) in
\cite[Equation~(57)]{chen2019riemann} is independent
of \(N\) and \(t\).
\end{prop}

\begin{proof}
By \cite[Equation~(82)]{chen2019riemann}, the local
parametrix \(P^{(0)}(z)\) is given by
\begin{equation}\label{eq:P0-cited}
P^{(0)}(z)
=
E_0(z)
\Psi(N^2f_0(z);2N^2t,a)
\e^{-\frac{\pi\mathrm{i}}{2}\sigma_3}
\widetilde W(z)^{-\sigma_3}
\varphi(z)^{-N\sigma_3}.
\end{equation}
With $\xi=N^2f_0(z)$,  we choose the branches consistently so that
\[
\xi^{\frac{1}{2}}
\stackrel{\eqref{defforf0}}{=}
N\left(
\log\varphi(z)-\log\varphi(0)
\right).
\]
It follows that
\begin{align}
\e^{\xi^{\frac{1}{2}}\sigma_3}
\e^{-\frac{\pi\mathrm{i}}{2}\sigma_3}
\widetilde W(z)^{-\sigma_3}
\varphi(z)^{-N\sigma_3}
=
\e^{-\frac{\pi\mathrm{i}}{2}\sigma_3}
\widetilde W(z)^{-\sigma_3}
\e^{-N\log\varphi(0)\sigma_3}
=
\mathcal D_0(z)^{-1}.
\end{align}
Furthermore, by the definition of \(\mathcal E\),
\begin{equation}
\xi^{\frac14\sigma_3}
\Psi(\xi,2N^2t)
=
\mathcal E(\xi,2N^2t)
\left(
\frac{\textbf{I}+\mathrm{i}\sigma_1}{\sqrt2}
\right)
\e^{\xi^{\frac{1}{2}}\sigma_3}.
\end{equation}
Combining these identities with the expression for
\(P^{(0)}(z)\) in \eqref{eq:P0-cited} and the definition
of \(J_{\widetilde R}\) on \(\partial U(0,r)\), and with \cite[$(83)$]{chen2019riemann}, we obtain
\eqref{smallneighbourhood}.
\end{proof}

We now return to the jump perturbation
\[
\Delta_{\widetilde R}(z;x)
=
J_{\widetilde R}(z;x,a)-\textbf{I}.
\]
By Proposition~\ref{aboutJR}, the jump on $\partial U(0,r)$ is
expressed in terms of the model solution evaluated at
\[
\xi=N^2f_0(z),
\qquad
x=2N^2t.
\]
The uniform large-$\xi$ estimates established above, together with
their $x$-derivative counterparts, therefore yield corresponding
uniform bounds for $\Delta_{\widetilde R}$ and its $x$-derivatives.
Combining these with the standard matching estimate on
$\partial U(1,r)$ and the exponentially small jumps on
$\Sigma_{\mathrm{exp}}$, we obtain the following lemma.

\begin{lem}\label{lem:derivative-Delta-Rtilde}
Let \(a>-1\) be fixed. Abbreviate
\[
\Delta_{\widetilde R}(z;x)
:=
J_{\widetilde R}(z;x,a)-\textbf{I},
\qquad
z\in\Sigma_{\widetilde R}
\]
with $J_{\widetilde R}(z;x,a)\equiv J_{\widetilde{R}}(z)$ defined in \eqref{eq:JR_definition}. Then, for every fixed $j\in\mathbb{Z}_{\geq 0}$, there exists a constant
$x_0=x_0(a,j)\in(0,1)$ such that
\begin{equation}\label{eq:Delta-Rtilde-L2-Linfty}
\left\|
\frac{\partial^j\Delta_{\widetilde R}(\cdot;x)}
{\partial x^j}
\right\|_{
L^2(\Sigma_{\widetilde R})
\cap
L^\infty(\Sigma_{\widetilde R})
}
=
\mathcal O_{a,j}
\left(
\frac{\rho_{a,j}(x)}{N}
\left(
1+\frac1N
\right)
\right)
\end{equation}
holds uniformly for \(0<x\leq x_0\) and $N\in\mathbb{N}$. Here,  $\rho_{a,k}(x)$ is defined in \eqref{definieofpho}.
\end{lem}

\begin{proof}
Recall that the jump contour for \(\widetilde R\) can be
decomposed as
\[
\Sigma_{\widetilde R}
=
\partial U(0,r)
\cup
\partial U(1,r)
\cup
\Sigma_{\mathrm{exp}};\ \ \ \ \ \Sigma_{\mathrm{exp}}
:=
(\Sigma_{S,1}\cup\Sigma_{S,3})
\setminus
\left(
U(0,r)\cup U(1,r)
\right).
\]
The contour \(\Sigma_{\widetilde R}\) has finite length and is
independent of \(N\) and \(x\). We first consider the jump on \(\partial U(0,r)\).
Let
\[
\xi=N^2f_0(z),
\qquad
z\in\partial U(0,r),
\]
where \(f_0\) is the locally analytic mapping defined in \eqref{defforf0}.
Since \(\partial U(0,r)\) is fixed and \(f_0\) does not
vanish on this circle by (\ref{esonf0z}), there exist constants
\(c_0,C_0>0\) such that
\begin{equation}\label{eq:xi-comparable-N2}
c_0N^2
\leq
|\xi|
\leq
C_0N^2,
\qquad
z\in\partial U(0,r).
\end{equation}
By the definition of the normalized model function in \eqref{eq:def-R2}, we have
\[
\widehat{\Psi}(\xi;x,a)
=
\textbf{I}+\frac{C_1(x)}{\xi}
+\mathcal R_2(\xi;x,a),
\]
and substituting that into \eqref{smallneighbourhood}, while using also \cite[$(83)$]{chen2019riemann}, we obtain
\begin{align}
	\Delta_{\widetilde{R}}(z;x)=&\,\,E_0(z)\Bigg\{\textbf{I}+\frac{C_1(x)}{\xi}+\mathcal{R}_2(\xi;x,a)\Bigg\}E_0(z)^{-1}-\textbf{I}\nonumber\\
	=&\,\,E_0(z)\Bigg\{\frac{\sigma_3}{\xi}q(x)-\frac{\i E_{12}}{\xi}v(x)+\frac{\i E_{21}}{\xi}\mathfrak{t}(x)+\mathcal{R}_2(\xi;x,a)\Bigg\}E_0(z)^{-1},
%
%
%
\label{eq:Delta-Rtilde-expansion}
\end{align}
where $\xi=N^2f_0(z)$ and $x=2N^2t$, as well as
\[
E_{12}
=
\begin{pmatrix}
0&1\\
0&0
\end{pmatrix},
\qquad
E_{21}
=
\begin{pmatrix}
0&0\\
1&0
\end{pmatrix},
\]
and \(E_0(z)\) is analytic and invertible in a
neighborhood of \(\partial U(0,r)\) (see \cite[Equation~(83)]{chen2019riemann}). In particular,
\begin{equation}\label{eq:E0-bounded}
\sup_{z\in\partial U(0,r)}
\left(
\|E_0(z)\|+\|E_0(z)^{-1}\|
\right)
<\infty.
\end{equation}
Since \(E_0\), \(f_0\), and \(\xi\) are independent of
\(x\), differentiating
\eqref{eq:Delta-Rtilde-expansion} with
respect to \(x\) gives
\begin{equation*}
\frac{\partial^j}
{\partial x^j}\Delta_{\widetilde R}(z;x)
=
E_0(z)\bigg\{\frac{\sigma_3}{\xi}q^{(j)}(x)-\frac{\i E_{12}}{\xi}v^{(j)}(x)+\frac{\i E_{21}}{\xi}\mathfrak{t}^{(j)}(x)+\frac{\partial^j}{\partial x^j}\mathcal{R}_2(\xi;x,a)\bigg\}E_0(z)^{-1}.
%
%
%
%
%
%
\end{equation*}
where $v(x)$, $q(x)$, and $\mathfrak{t}(x)$ are as in
\eqref{36}, \eqref{eq:def-q}, and \eqref{eq:def-t-prime},
respectively.
Hence, by Lemma \ref{asympoticsforv(x)forsmallx} and Proposition \ref{prop:q-t-derivative-estimates}, we have for $j\in \mathbb{Z}_{\geq 0}$,
\begin{equation}\label{eq:C1-entry-derivatives}
|q^{(j)}(x)|
+
|v^{(j)}(x)|
+
|\mathfrak t^{(j)}(x)|
=
\mathcal O_{a,j,x_0}
\bigl(
\rho_{a,j}(x)
\bigr),\ \ \ \forall\,0<x\leq x_0.
\end{equation}
Moreover, by \eqref{generalk},
\begin{equation}\label{eq:R2-derivative-used}
\frac{\partial^j}{\partial x^j}\mathcal R_2(\xi;x,a)
=
\mathcal O_{a,j}
\left(
\frac{\rho_{a,j}(x)}{|\xi|^2}
\right),
\qquad
|\xi|\geq \xi_0,\quad 0<x\leq x_0.
\end{equation}
Combining this with \eqref{eq:xi-comparable-N2},
we see that \eqref{eq:R2-derivative-used} applies directly for all
sufficiently large $N$. For the finitely many remaining values of
$N$, the corresponding values of $\xi=N^2f_0(z)$,
$z\in\partial U(0,r)$, range over a fixed compact set bounded away
from the origin. The corresponding bounds can therefore be absorbed
into the constant in \eqref{eq:R2-derivative-used}. Consequently,
\eqref{eq:R2-derivative-used} remains valid for
\[
\xi=N^2f_0(z),
\qquad
z\in\partial U(0,r),
\qquad
N\in\mathbb N,
\]
uniformly for $0<x\leq x_0$. Combining
\eqref{eq:xi-comparable-N2},
\eqref{eq:E0-bounded},
\eqref{eq:C1-entry-derivatives}, and
\eqref{eq:R2-derivative-used}, we obtain
\begin{align}
\left\|
\frac{\partial^j\Delta_{\widetilde R}(z;x)}
{\partial x^j}
\right\|_{L^\infty(\partial U(0,r))}
&=\mathcal O_{a,j,x_0}
\left(
\frac{\rho_{a,j}(x)}{N^2}
+
\frac{\rho_{a,j}(x)}{N^4}
\right)=
\mathcal O_{a,j,x_0}
\left(
\frac{\rho_{a,j}(x)}{N^2}
\left(
1+\frac{1}{N^2}
\right)
\right).
\label{eq:Delta-bound-origin-circle}
\end{align}
We next consider \(\partial U(1,r)\). The local
parametrix at \(z=1\) is independent of \(x\). Therefore,
for \(j\geq1\),
\[
\frac{\partial^j\Delta_{\widetilde R}(z;x)}
{\partial x^j}
=0,
\qquad
z\in\partial U(1,r).
\]
For \(j=0\), the standard matching condition in \eqref{matcheasy} gives
\[
\Delta_{\widetilde R}(z;x)
=
\mathcal O\left(\frac1N\right),
\qquad
z\in\partial U(1,r).
\]
Since \(\rho_{a,0}(x)\geq1\), this estimate is also
consistent with the right-hand side of
\eqref{eq:Delta-Rtilde-L2-Linfty}.
Finally, on \(\Sigma_{\mathrm{exp}}\), the jump matrix is
independent of \(x\), and
\[
\Delta_{\widetilde R}(z;x)
=
\mathcal O\big(\e^{-cN}\big)
\]
for some \(c>0\). Consequently,
\[
\frac{\partial^j\Delta_{\widetilde R}(z;x)}
{\partial x^j}
=0,
\qquad
j\geq1,
\quad
z\in\Sigma_{\mathrm{exp}},
\]
whereas, for \(j=0\),
\[
\|\Delta_{\widetilde R}(\cdot;x)\|_
{L^\infty(\Sigma_{\mathrm{exp}})}
=
\mathcal O\big(\e^{-cN}\big)
=
\mathcal O\big(N^{-1}\big).
\]
Combining the estimates on the three components of
\(\Sigma_{\widetilde R}\), we conclude that
\begin{equation}\label{eq:Delta-Rtilde-Linfty}
\left\|
\frac{\partial^j\Delta_{\widetilde R}(\cdot;x)}
{\partial x^j}
\right\|_{L^\infty(\Sigma_{\widetilde R})}
=
\mathcal O_{a,j,x_0}
\left(
\frac{\rho_{a,j}(x)}{N}
\left(
1+\frac1N
\right)
\right).
\end{equation}
Since \(\Sigma_{\widetilde R}\) has finite length, and is 
independent of \(N\) and \(x\), estimate \eqref{eq:Delta-Rtilde-Linfty} implies the analogous estimate in $L^2(\Sigma_{\widetilde{R}};\mathbb{C}^{2\times 2})$ and thus \eqref{eq:Delta-Rtilde-L2-Linfty} follows. Now, choose $x_0=x_0(a,j)\in(0,1)$ sufficiently small so that
all the small-$x$ estimates used above hold uniformly for
$0<x\leq x_0$. Since $x_0$ depends only on $a$ and $j$, all
constants depending on $a$, $j$, and $x_0$ may be absorbed into
constants depending only on $a$ and $j$. Thus, we obtain the conclusion in this lemma.
\end{proof}
Lemma~\ref{lem:derivative-Delta-Rtilde} provides the required $L^2$- and $L^\infty$-bounds
for the jump perturbation $\Delta_{\widetilde R}$ and its
$x$-derivatives. We now use the standard singular-integral
formulation of the small-norm Riemann--Hilbert problem to transfer
these estimates to the boundary value $\widetilde R_-$. In
particular, the zero-order estimate in Lemma~\ref{lem:derivative-Delta-Rtilde} implies that,
for $N$ sufficiently large, the associated singular-integral
operator is uniformly invertible. Differentiating the resulting
integral equation with respect to $x$ and arguing inductively then
gives the corresponding estimates for the $x$-derivatives of
$\widetilde R_-$.

\begin{lem}\label{lem:derivative-Rtilde-minus}
Let \(a>-1\). Let \(\widetilde R(z;x,a)\equiv\widetilde R(z;x)\) be the
solution of RHP~\ref{modelproblem2}, and let
\(\widetilde R_\pm(\cdot;x)\) denote its non-tangential boundary values
on the oriented contour \(\Sigma_{\widetilde R}\).
Then there exists \(N_0=N_0(a)\in\mathbb N\) such that, for any
\(j\in\mathbb Z_{\geq0}\), there exists
\(x_0=x_0(a,j)\in(0,1)\) for which the following estimates hold
uniformly for \(0<x\leq x_0\) and \(N\geq N_0\).
For $j=0$,
\begin{equation}\label{eq:Rtilde-minus-zero}
\left\|
\widetilde R_-(\cdot;x,a)-\textbf{I}
\right\|_{L^2(\Sigma_{\widetilde R})}
=
\mathcal O_{a}
\left(
\frac{\rho_{a,0}(x)}{N}
\left(
1+\frac1N
\right)
\right),
\end{equation}
and for every $j\in\mathbb{N}$,
\begin{equation}\label{eq:Rtilde-minus-derivative}
\left\|
\frac{\partial^j
\widetilde R_-(\cdot;x,a)}
{\partial x^j}
\right\|_{L^2(\Sigma_{\widetilde R})}
=
\mathcal O_{a,j}
\left(
\frac{\rho_{a,j}(x)}{N}
\left(
1+\frac1N
\right)
\right).
\end{equation}
Here,  $\rho_{a,k}(x)$ is defined in (\ref{definieofpho}).
\end{lem}

\begin{proof}
For convenience, abbreviate
\begin{equation}\label{eq:def-epsilon-N}
\varepsilon_N
:=
\frac1N\left(1+\frac1N\right).
\end{equation}
For \(j=0\), the quantity that is small in
\(L^2(\Sigma_{\widetilde R};\mathbb{C}^{2\times 2})\) is
\(\widetilde R_- - \textbf{I}\), rather than
\(\widetilde R_-\) itself. For \(j\geq1\), the identity
matrix disappears after differentiation, and therefore
one may estimate
\begin{equation*}
	\frac{\partial^j}{\partial x^j}\widetilde R_- (\xi;x,a),\ \ \ \ \xi\in\Sigma_{\widetilde{R}}
\end{equation*}
directly.  Let \(C_-\) denote the minus-side Cauchy projection on
\(\Sigma_{\widetilde R}\), defined analogously as in
\eqref{eq:Cauchy-projections}. Define the singular integral
operator
\begin{equation}\label{eq:def-C-Delta-Rtilde}
\mathcal C_{\Delta_{\widetilde R}}f
:=
C_-\!\left(f\Delta_{\widetilde R}\right),
\qquad
f\in L^2\!\left(
\Sigma_{\widetilde R};\mathbb C^{2\times2}
\right).
\end{equation}
Since \(C_-\) is bounded on
\(L^2(\Sigma_{\widetilde R};\mathbb C^{2\times2})\), it follows from
Lemma~\ref{lem:derivative-Delta-Rtilde} that
\begin{equation}\label{eq:C-Delta-small}
\big\|\mathcal C_{\Delta_{\widetilde R}}\big\|_{L^2\to L^2}
\leq
\|C_-\|_{L^2\to L^2}
\|\Delta_{\widetilde R}\|_{L^\infty(\Sigma_{\widetilde R})}
=
\mathcal O_{a}\big(\rho_{a,0}(x)\epsilon_N\big),
\end{equation}
uniformly for \(0<x\leq x_0\), for some
\(x_0=x_0(a)\in(0,1)\).
Since $\sup_{0<x\leq x_0}\rho_{a,0}(x)<\infty$, there exists $N_0=N_{0}(a)>0$ such that
\begin{equation*}
	\sup_{0<x\leq x_0}\big\|\mathcal{C}_{\Delta_{\widetilde{R}}}\big\|_{L^2\rightarrow L^2}\leq\frac{1}{2}\ \ \ \ \forall\,N\geq N_0.
\end{equation*}
Consequently, $I-\mathcal{C}_{\Delta_{\widetilde{R}}}$ is invertible on $L^2(\Sigma_{\widetilde{R}};\mathbb{C}^{2\times 2})$ for large $N\geq N_0$, with the operator norm bound
\begin{equation}\label{eq:inverse-C-Delta}
\left\|
\left(
\textbf{I}-\mathcal C_{\Delta_{\widetilde R}}
\right)^{-1}
\right\|
\leq2,
\qquad
\forall\,N\geq N_0,
\end{equation}
uniformly for \(0<x\leq x_0\). By general theory, compare also \eqref{z1}, the solution of the RHP for \(\widetilde R\)
satisfies the singular integral equation
\begin{equation}\label{eq:Rtilde-minus-SIE}
\widetilde R_-
=
\textbf{I}+
C_-\left(
\widetilde R_-\Delta_{\widetilde R}
\right).
\end{equation}
Abbreviating $\eta(\xi;x):=\widetilde{R}_-(\xi;x)-\textbf{I}$ and substituting \(\widetilde R_-=\textbf{I}+\eta\) into
\eqref{eq:Rtilde-minus-SIE}, we obtain
\[
\eta
=
C_-\left(\Delta_{\widetilde R}\right)
+
C_-\left(
\eta\Delta_{\widetilde R}
\right)=C_-\left(\Delta_{\widetilde R}\right)+\mathcal{C}_{\Delta_{\widetilde{R}}}\eta
\]
or equivalently,
\begin{equation}\label{eq:eta-SIE}
\left(
\textbf{I}-\mathcal C_{\Delta_{\widetilde R}}
\right)\eta
=
C_-\left(\Delta_{\widetilde R}\right)
\end{equation}
Using \eqref{eq:inverse-C-Delta} and
Lemma \ref{lem:derivative-Delta-Rtilde}, we find
\begin{align}
\|\eta(\cdot;x)\|_{L^2(\Sigma_{\widetilde R})}
\leq
2\|C_-\|\,
\|\Delta_{\widetilde R}\|_
{L^2(\Sigma_{\widetilde R})}
=
\mathcal O_{a}
\left(
\rho_{a,0}(x)\varepsilon_N
\right)
\label{eq:eta-L2-estimate}
\end{align}
uniformly for $0<x\leq x_{0}$.
This proves \eqref{eq:Rtilde-minus-zero}. We next prove \eqref{eq:Rtilde-minus-derivative} by
induction on \(j\in\mathbb{N}\). Since
\(\Delta_{\widetilde R}(\cdot,x)\) is smooth with respect
to \(x\) and
\(\textbf{I}-\mathcal C_{\Delta_{\widetilde R}}\) is uniformly
invertible, the singular integral equation
\eqref{eq:Rtilde-minus-SIE} may be differentiated with
respect to \(x\). Differentiating
\eqref{eq:Rtilde-minus-SIE} while using the Leibniz rule
gives
\begin{equation}\label{eq:differentiate-Rtilde-SIE}
\frac{\partial^j\widetilde R_-}{\partial x^j}
=
\sum_{\ell=0}^{j}
\binom{j}{\ell}\,
C_-\left(
\frac{\partial^\ell\widetilde R_-}{\partial x^\ell}
\frac{\partial^{j-\ell}\Delta_{\widetilde R}}
{\partial x^{j-\ell}}
\right).
\end{equation}
The summand in \eqref{eq:differentiate-Rtilde-SIE} corresponding to \(\ell=j\) equals
\[
\mathcal C_{\Delta_{\widetilde R}}
\left(
\frac{\partial^j\widetilde R_-}{\partial x^j}
\right),
\]
and therefore, with the understanding $\sum_{\ell=1}^0\ldots=0$,
\begin{align}
&
\left(
\textbf{I}-\mathcal C_{\Delta_{\widetilde R}}
\right)
\frac{\partial^j\widetilde R_-}{\partial x^j}=
C_-\left(
\widetilde R_-
\frac{\partial^j\Delta_{\widetilde R}}
{\partial x^j}
\right)
+
\sum_{\ell=1}^{j-1}
\binom{j}{\ell}
C_-\left(
\frac{\partial^\ell\widetilde R_-}{\partial x^\ell}
\frac{\partial^{j-\ell}\Delta_{\widetilde R}}
{\partial x^{j-\ell}}
\right).
\label{eq:Rtilde-derivative-SIE}
\end{align}
We begin by estimating the first term on the right-hand side
of \eqref{eq:Rtilde-derivative-SIE}. Since
\(\widetilde R_-=\textbf{I}+\eta\),
\begin{align}
&
\left\|
C_-\left(
\widetilde R_-
\frac{\partial^j\Delta_{\widetilde R}}
{\partial x^j}
\right)
\right\|_{L^2}
\leq
\|C_-\|\,
\left\|
\frac{\partial^j\Delta_{\widetilde R}}
{\partial x^j}
\right\|_{L^2}+
\|C_-\|\,
\|\eta\|_{L^2}
\left\|
\frac{\partial^j\Delta_{\widetilde R}}
{\partial x^j}
\right\|_{L^\infty}.
\label{eq:first-Rtilde-term}
\end{align}
By Lemma \ref{lem:derivative-Delta-Rtilde} and
\eqref{eq:eta-L2-estimate}, together with the product
estimate
\begin{equation}\label{eq:rho-product-in-Rtilde}
\rho_{a,p}(x)\rho_{a,q}(x)
=
\mathcal O_{a,p,q,x_0}
\bigl(
\rho_{a,p+q}(x)
\bigr),
\qquad\forall\,
0<x\leq x_0,
\end{equation}
we obtain
\begin{align}
&
\left\|
C_-\left(
\widetilde R_-
\frac{\partial^j\Delta_{\widetilde R}}
{\partial x^j}
\right)
\right\|_{L^2}=
\mathcal O_{a,j,x_0}
\left(
\rho_{a,j}(x)\varepsilon_N
\right)
+
\mathcal O_{a,j,x_0}
\left(
\rho_{a,0}(x)\rho_{a,j}(x)
\varepsilon_N^2
\right)=
\mathcal O_{a,j,x_0}
\left(
\rho_{a,j}(x)\varepsilon_N
\right).
\label{eq:first-Rtilde-term-final}
\end{align}
Here we have also used
\(\varepsilon_N^2=\mathcal O(\varepsilon_N)\). For \(j=1\), equations
\eqref{eq:Rtilde-derivative-SIE},
\eqref{eq:first-Rtilde-term-final}, and
\eqref{eq:inverse-C-Delta} immediately imply that
\[
\left\|
\frac{\partial\widetilde R_-}{\partial x}
\right\|_{L^2}
=
\mathcal O_{a,x_0}
\left(
\rho_{a,1}(x)\varepsilon_N
\right).
\]
This establishes the initial step of the induction. Now let \(j\geq2\), and suppose that, for every
\(1\leq\ell\leq j-1\),
\begin{equation}\label{eq:Rtilde-induction-hypothesis}
\left\|
\frac{\partial^\ell\widetilde R_-}
{\partial x^\ell}
\right\|_{L^2(\Sigma_{\widetilde R})}
=
\mathcal O_{a,\ell,x_0}
\left(
\rho_{a,\ell}(x)\varepsilon_N
\right).
\end{equation}
For \(1\leq\ell\leq j-1\), we have by \eqref{eq:rho-product-in-Rtilde} and \eqref{eq:Rtilde-induction-hypothesis},
\begin{align}
&
\left\|
C_-\left(
\frac{\partial^\ell\widetilde R_-}{\partial x^\ell}
\frac{\partial^{j-\ell}\Delta_{\widetilde R}}
{\partial x^{j-\ell}}
\right)
\right\|_{L^2}
\leq
\|C_-\|\,
\left\|
\frac{\partial^\ell\widetilde R_-}{\partial x^\ell}
\right\|_{L^2}
\left\|
\frac{\partial^{j-\ell}\Delta_{\widetilde R}}
{\partial x^{j-\ell}}
\right\|_{L^\infty}
=
\mathcal O_{a,j,x_0}
\left(
\rho_{a,j}(x)\varepsilon_N
\right).
\label{eq:intermediate-Rtilde-terms}
\end{align}
It follows from
\eqref{eq:first-Rtilde-term-final} and
\eqref{eq:intermediate-Rtilde-terms} that the right-hand
side of \eqref{eq:Rtilde-derivative-SIE} is of order
\[
\mathcal O_{a,j,x_0}
\left(
\rho_{a,j}(x)\varepsilon_N
\right).
\]
Applying the uniformly
bounded inverse in \eqref{eq:inverse-C-Delta}, we conclude
that
\[
\left\|
\frac{\partial^j\widetilde R_-(\cdot;x,a)}
{\partial x^j}
\right\|_{L^2(\Sigma_{\widetilde R})}
=
\mathcal O_{a,j,x_0}
\left(
\rho_{a,j}(x)\varepsilon_N
\right),
\]
which is precisely \eqref{eq:Rtilde-minus-derivative}, i.e. the induction is complete.
Now, choose $x_0=x_0(a,j)\in(0,1)$ sufficiently small so that
all the small-$x$ estimates used above, including those for
derivative orders up to $j$, hold uniformly for $0<x\leq x_0$.
Since $x_0$ depends only on $a$ and $j$, all constants depending
on $a$, $j$, and $x_0$ may be absorbed into constants depending
only on $a$ and $j$. Hence the conclusion of the lemma follows.
\end{proof}

We next transfer the $L^2$-estimates for the boundary values
$\widetilde R_-$ obtained in
Lemma~\ref{lem:derivative-Rtilde-minus} to pointwise estimates
for $\widetilde R$ at the origin. Since $0$ stays a positive
distance away from $\Sigma_{\widetilde R}$, the Cauchy
representation of the solution, together with
Lemma~\ref{lem:derivative-Delta-Rtilde} and
Lemma~\ref{lem:derivative-Rtilde-minus}, yields the corresponding
estimates for $\widetilde R(0;x,a)$ and its $x$-derivatives.

\begin{lem}\label{lem:derivative-Rtilde-at-zero}
Let \(a>-1\), and let
\(\widetilde R(z;x,a)\equiv\widetilde R(z;x)\)
be the solution of RHP~\ref{modelproblem2}.
Then there exists \(N_0=N_0(a)\in\mathbb N\) such that, for any
\(j\in\mathbb Z_{\geq0}\), there exists a constant
\(x_0=x_0(a,j)\in(0,1)\) for which the following estimates hold
uniformly for \(0<x\leq x_0\) and \(N\geq N_0\). For \(j=0\),
\begin{equation}\label{eq:Rtilde-at-zero-zero-order}
\widetilde R(0;x)-\mathbf I
=
\mathcal O_a
\left(
\frac{\rho_{a,0}(x)}{N}
\left(1+\frac{1}{N}\right)
\right),
\end{equation}
and, for every fixed \(j\in\mathbb N\),
\begin{equation}\label{eq:Rtilde-at-zero-derivative}
\frac{\partial^j}{\partial x^j}\widetilde R(0;x)
=
\mathcal O_{a,j}
\left(
\frac{\rho_{a,j}(x)}{N}
\left(1+\frac{1}{N}\right)
\right).
\end{equation}
Here, \(\rho_{a,k}(x)\) is defined in
\eqref{definieofpho}.
\end{lem}
\begin{proof}
Since \(0\notin\Sigma_{\widetilde R}\), the Cauchy
representation for \(\widetilde R(z;x)\) can be evaluated at
\(z=0\), giving
\begin{equation}\label{eq:Rtilde-Cauchy-at-zero}
\widetilde R(0;x)
=
\textbf{I}+
\frac{1}{2\pi\mathrm i}
\int_{\Sigma_{\widetilde R}}
\widetilde R_-(\tau;x)
\Delta_{\widetilde R}(\tau;x)
\,\frac{\mathrm{d}\tau}{\tau}.
\end{equation}
As before, set $\eta(\tau;x):=\widetilde{R}_-(\tau;x)-\textbf{I}$ so that
\[
\widetilde R_-\Delta_{\widetilde R}
=
\Delta_{\widetilde R}
+
\eta\Delta_{\widetilde R},
\]
and hence
\begin{align}
\widetilde R(0;x)-\textbf{I}
={}&
\frac{1}{2\pi\mathrm i}
\int_{\Sigma_{\widetilde R}}
\Delta_{\widetilde R}(\tau;x)\frac{\mathrm{d}\tau}{\tau}
+
\frac{1}{2\pi\mathrm i}
\int_{\Sigma_{\widetilde R}}
\eta(\tau;x)
\Delta_{\widetilde R}(\tau;x)
\frac{\mathrm{d}\tau}{\tau}.
\label{eq:Rtilde-zero-split}
\end{align}
Since
$
\inf_{\tau\in\Sigma_{\widetilde R}}|\tau|>0
$
and \(\Sigma_{\widetilde R}\) has finite arclength, we have
\begin{equation}\label{eq:kernel-at-zero-L2}
\left\|
\frac{1}{\,\cdot\,}
\right\|_{L^2(\Sigma_{\widetilde R})}
<\infty.
\end{equation}
For $j\in \mathbb{Z}_{\geq 0}$,
let
\[
x_0=x_{0}(a,j)
:=
\min\left\{
x_0^{(1)}(a,j),
x_0^{(2)}(a,j)
\right\},
\]
where $x_0^{(1)}(a,j)$ and $x_0^{(2)}(a,j)$ are the constants
appearing in Lemmas~\ref{lem:derivative-Delta-Rtilde} and
\ref{lem:derivative-Rtilde-minus}, respectively.
Then both lemmas hold uniformly for $0<x\leq x_0$. With $\epsilon_N$ as in \eqref{eq:def-epsilon-N},  Cauchy-Schwarz inequality and  Lemma \ref{lem:derivative-Delta-Rtilde} yield
\begin{align}
\left\|
\int_{\Sigma_{\widetilde R}}
\Delta_{\widetilde R}(\tau;x)
\frac{\mathrm{d}\tau}{\tau}
\right\|
&\leq
\left\|
\frac{1}{\,\cdot\,}
\right\|_{L^2(\Sigma_{\widetilde R})}
\left\|
\Delta_{\widetilde R}(\cdot,x)
\right\|_{L^2(\Sigma_{\widetilde R})}=
\mathcal O_{a}
\left(
\rho_{a,0}(x)\varepsilon_N
\right).
\label{eq:linear-zero-estimate}
\end{align}
Similarly, using also \eqref{eq:eta-L2-estimate},
\begin{align}
&
\left\|
\int_{\Sigma_{\widetilde R}}
\eta(\tau;x)
\Delta_{\widetilde R}(\tau;x)
\frac{\mathrm{d}\tau}{\tau}
\right\|
\leq
\left\|
\frac{1}{\,\cdot\,}
\right\|_{L^2(\Sigma_{\widetilde R})}
\|\eta(\cdot;x)\|_{L^2(\Sigma_{\widetilde R})}
\|\Delta_{\widetilde R}(\cdot;x)\|_
{L^\infty(\Sigma_{\widetilde R})}
=
\mathcal O_{a}
\left(
\rho_{a,0}(x)^2\varepsilon_N^2
\right).
\label{eq:nonlinear-zero-estimate}
\end{align}
Since $\rho_{a,0}(x)^2=\mathcal{O}_{a}(\rho_{a,0}(x))$ by (\ref{eq:rho-product-in-Rtilde}) and since $\epsilon_N^2=\mathcal{O}(\epsilon_N)$ for all $N\in\mathbb{N}$,
equations \eqref{eq:Rtilde-zero-split}--%
\eqref{eq:nonlinear-zero-estimate} imply
\[
\widetilde R(0;x)-\textbf{I}
=
\mathcal O_{a}
\left(
\rho_{a,0}(x)\varepsilon_N
\right),
\]
and this proves \eqref{eq:Rtilde-at-zero-zero-order}. We now consider derivatives with respect to \(x\).
Differentiating \eqref{eq:Rtilde-Cauchy-at-zero} \(j\)
times under the integral sign and using the Leibniz rule,
we obtain
\begin{equation}\label{eq:Rtilde-j-raw-at-zero}
\frac{\partial^j}{\partial x^j}\widetilde R(0;x)
=
\frac{1}{2\pi\mathrm i}
\sum_{\ell=0}^{j}
\binom{j}{\ell}
\int_{\Sigma_{\widetilde R}}
\left[\frac{\partial^{\ell}}{\partial x^{\ell}}\widetilde R_-(\tau;x)\right]
\left[
\frac{\partial^{j-\ell}}{\partial x^{j-\ell}}
\Delta_{\widetilde R}(\tau;x)\right]
\frac{\mathrm{d}\tau}{\tau}
.
\end{equation}
We note that 
\begin{align*}
\left\|
\int_{\Sigma_{\widetilde R}}
\left[\frac{\partial^{\ell}}{\partial x^{\ell}}\widetilde R_-(\tau;x)\right]
\left[
\frac{\partial^{j-\ell}}{\partial x^{j-\ell}}
\Delta_{\widetilde R}(\tau;x)\right]
\frac{\mathrm{d}\tau}{\tau}
\right\|
\leq
\left\|
\frac{1}{\,\cdot\,}
\right\|_{L^2(\Sigma_{\widetilde R})}
\left\|\frac{\partial^{\ell}}{\partial x^{\ell}}
\widetilde R_-(\cdot;x)
\right\|_{L^2(\Sigma_{\widetilde{R}})}
\left\|\frac{\partial^{j-\ell}}{\partial x^{j-\ell}}
\Delta_{\widetilde R}(\cdot;x)
\right\|_{L^\infty(\Sigma_{\widetilde{R}})}.
\label{eq:mixed-j-at-zero}
\end{align*}
Using Lemmas~\ref{lem:derivative-Delta-Rtilde} and
\ref{lem:derivative-Rtilde-minus}, together with
\eqref{eq:kernel-at-zero-L2} and
\eqref{eq:rho-product-in-Rtilde}, we immediately obtain
\eqref{eq:Rtilde-at-zero-derivative}. This completes the proof.
\end{proof}

\section{Distributional Convergence in the Full Range \(a>-1\)}\label{distconv}

Before turning to the final section of the paper,
Section~\ref{proofs}, where we prove the main results,
Theorems~\ref{limitcase2} and~\ref{maintheorem2}, we first complement
the preceding analysis with a proposition on the convergence in
distribution of \(\mathsf Y_N^{a,b}\), defined in
\eqref{defofYN}. This result will be used in the proof of the
probabilistic representation of the limiting value in
Theorem~\ref{limitcase2}.

We first record the kernel estimates that will be used to control
the tails of the inverse power sums. Let \(K_N(x,y)\) denote the
correlation kernel, with respect to Lebesgue measure on \((0,1)\),
of the Jacobi unitary ensemble with parameters \(a,b>-1\). Thus,
\[
K_N(x,y)
=
\sum_{n=0}^{N-1}\psi_n(x)\psi_n(y),
\]
where
\begin{equation}\label{defofpsi}
\psi_n(x)
=
\frac{p_n^{(a,b)}(x)}{\sqrt{h_n}}\,
x^{a/2}(1-x)^{b/2},
\qquad
p_n^{(a,b)}(x)
=
P_n^{(a,b)}(1-2x),
\end{equation}
and
\[
h_n
=
\frac{
\Gamma(n+a+1)\Gamma(n+b+1)
}{
(2n+a+b+1)\Gamma(n+1)\Gamma(n+a+b+1)
}.
\]
In particular, \(K_N\) is the kernel of a rank-\(N\) orthogonal
projection and therefore
\begin{equation}\label{eq:trace-KN}
\int_0^1K_N(x,x)\,\d x=N.
\end{equation}

\begin{lem}\label{lem:kernel-tail-estimates}
Let \(a,b>-1\). There exist constants
\(c_0=c_0(a,b)\in(0,1]\), \(N_0=N_0(a,b)\in\mathbb N\), and
\(C=C(a,b)>0\) such that, for every \(N\geq N_0\),
\begin{equation}\label{eq:kernel-hard-edge-bound}
K_N(x,x)
\leq
C N^2(N^2x)^a,
\qquad
0<x\leq\frac{c_0}{N^2},
\end{equation}
and
\begin{equation}\label{eq:kernel-intermediate-bound}
K_N(x,x)
\leq
C\frac{N^2}{\sqrt{N^2x}}
=
C\frac{N}{\sqrt{x}},
\qquad
\frac{c_0}{N^2}\leq x\leq\frac14.
\end{equation}
Moreover, there exists \(C_a>0\) such that
\begin{equation}\label{eq:Bessel-kernel-bounds}
\mathsf K_a^{\mathrm{Bes}}(u,u)
\leq
\begin{cases}
C_a u^a, & 0<u\leq1,\\[1mm]
C_a u^{-1/2}, & u\geq1.
\end{cases}
\end{equation}
Here, $\mathsf K_a^{\mathrm{Bes}}(\cdot,\cdot)$ is defined in
Definition~\ref{def:Bessel}.
\end{lem}
\begin{proof}
By Stirling's formula,
\begin{equation}\label{eq:hn-bound}
h_n^{-1}
=
\mathcal O_{a,b}(n+1),
\qquad
n\in\mathbb Z_{\geq0}.
\end{equation}
For \(0<x\leq1/4\), set
\[
\theta=\arccos(1-2x).
\]
Then
\begin{equation}\label{eq:theta-x}
\theta\asymp\sqrt{x},
\qquad
0<x\leq\frac14.
\end{equation}
By Theorem~7.32.2 and the estimates following it in
\cite{Szego1975}, there exist constants
\(c_1>0\), \(n_0\in\mathbb N\), and \(C=C(a,b)>0\) such that,
for every \(n\geq n_0\),
\begin{equation}\label{eq:Szego-two-regimes}
\left|p_n^{(a,b)}(x)\right|
\leq
C
\begin{cases}
(n+1)^a,
& (n+1)\theta\leq c_1,\\[1mm]
(n+1)^{-1/2}\theta^{-a-\frac12},
& (n+1)\theta>c_1,
\end{cases}
\qquad
0<x\leq\frac14.
\end{equation}
By increasing \(C\), if necessary, the same estimates hold for
all \(n\in\mathbb Z_{\geq0}\). We first consider
\[
0<x\leq\frac{c_0}{N^2},
\]
where \(c_0>0\) is chosen sufficiently small. By
\eqref{eq:theta-x}, after decreasing \(c_0\) if necessary,
\[
(n+1)\theta\leq c_1,
\qquad
0\leq n\leq N-1.
\]
Hence, by \eqref{eq:hn-bound} and the first estimate in
\eqref{eq:Szego-two-regimes},
\[
|\psi_n(x)|^2
\leq
C_{a,b}(n+1)^{2a+1}x^a,
\]
where we have also used \(1-x\asymp1\) for \(0<x\leq1/4\).
Since \(a>-1\), we have \(2a+1>-1\), and therefore
\begin{align*}
K_N(x,x)=
\sum_{n=0}^{N-1}|\psi_n(x)|^2\leq
C_{a,b}x^a
\sum_{n=0}^{N-1}(n+1)^{2a+1}\leq
C_{a,b}N^{2a+2}x^a.
\end{align*}
This proves \eqref{eq:kernel-hard-edge-bound}.
We next consider
\[
\frac{c_0}{N^2}\leq x\leq\frac14.
\]
We split the sum defining \(K_N(x,x)\) according to whether
\((n+1)\theta\leq c_1\) or \((n+1)\theta>c_1\).
For the first part, \eqref{eq:hn-bound} and the first estimate in
\eqref{eq:Szego-two-regimes} give
\[
|\psi_n(x)|^2
\leq
C_{a,b}(n+1)^{2a+1}x^a.
\]
Consequently,
\begin{align*}
\sum_{\substack{0\leq n\leq N-1\\
(n+1)\theta\leq c_1}}
|\psi_n(x)|^2
&\leq
C_{a,b}x^a
\sum_{n+1\leq c_1/\theta}
(n+1)^{2a+1}\leq
C_{a,b}x^a\theta^{-2a-2}\leq
C_{a,b}x^{-1},
\end{align*}
where \eqref{eq:theta-x} was used in the last step.

For the second part, \eqref{eq:hn-bound} and the second estimate
in \eqref{eq:Szego-two-regimes} imply
\[
|\psi_n(x)|^2
\leq
C_{a,b}x^a\theta^{-2a-1}.
\]
Using \eqref{eq:theta-x}, we obtain
\[
x^a\theta^{-2a-1}
\leq
C_{a,b}x^{-1/2}.
\]
Thus,
\[
\sum_{\substack{0\leq n\leq N-1\\
(n+1)\theta>c_1}}
|\psi_n(x)|^2
\leq
C_{a,b}Nx^{-1/2}.
\]
Combining the two contributions gives
\[
K_N(x,x)
\leq
C_{a,b}\left(
x^{-1}
+
Nx^{-1/2}
\right).
\]
Since \(x\geq c_0N^{-2}\),
\[
x^{-1}
\leq
c_0^{-1/2}Nx^{-1/2}.
\]
Therefore
\[
K_N(x,x)
\leq
C_{a,b}Nx^{-1/2}
=
C_{a,b}
\frac{N^2}{\sqrt{N^2x}},
\]
which proves \eqref{eq:kernel-intermediate-bound}.

Finally, by taking the diagonal limit in the definition of the
Bessel kernel, we obtain
\[
\mathsf K_a^{\mathrm{Bes}}(u,u)
=
\frac14
\left(
J_a(\sqrt u)^2
-
J_{a-1}(\sqrt u)J_{a+1}(\sqrt u)
\right).
\]
The standard small- and large-argument asymptotics of the Bessel
functions therefore give
\[
\mathsf K_a^{\mathrm{Bes}}(u,u)
=
\mathcal O_a(u^a),
\qquad u\downarrow0,
\]
and
\[
\mathsf K_a^{\mathrm{Bes}}(u,u)
=
\mathcal O_a(u^{-1/2}),
\qquad u\to\infty.
\]
Since \(u\mapsto\mathsf K_a^{\mathrm{Bes}}(u,u)\) is continuous
on compact subsets of \((0,\infty)\), these asymptotics imply
\eqref{eq:Bessel-kernel-bounds}.
\end{proof}
\begin{proof}[Proof of Proposition \ref{lem:YN-convergence-all-a}]
Fix \(c\geq1\).
We decompose
\[
\mathsf Y_N^{a,b}
=
\widetilde{\mathsf Y}_{N,c}^{a,b}
+
E_{N,c}^{a,b},
\]
where
\begin{equation}\label{eq:finiteN-truncation}
\widetilde{\mathsf Y}_{N,c}^{a,b}
=
\frac{1}{N^2}\sum_{j=1}^{N}
\frac{1}{x_j^{(N)}}
\mathbf 1_{\{x_j^{(N)}<c/N^2\}},
\qquad
E_{N,c}^{a,b}
=
\frac{1}{N^2}\sum_{j=1}^{N}
\frac{1}{x_j^{(N)}}
\mathbf 1_{\{x_j^{(N)}\geq c/N^2\}}.
\end{equation}
Similarly, we write
\[
\mathfrak e(a)
=
\mathfrak e_c(a)
+
E_c^{(\mathrm{Bes})},
\]
where
\begin{equation}\label{eq:Bessel-truncation}
\mathfrak e_c(a)
:=
\sum_{\mathfrak B_a^i<c}
\frac{1}{\mathfrak B_a^i},
\qquad
E_c^{(\mathrm{Bes})}
:=
\sum_{\mathfrak B_a^i\geq c}
\frac{1}{\mathfrak B_a^i}.
\end{equation}
For any \(z\geq0\), the triangle inequality gives
\begin{align}
&
\left|
\mathbb E_N^{(a,b)}
\left[
\e^{-z\mathsf Y_N^{a,b}}
\right]
-
\mathbb E
\left[
\e^{-z\mathfrak e(a)}
\right]
\right|
\nonumber\\
\leq&
\left|
\mathbb E_N^{(a,b)}
\left[
\e^{-z\widetilde{\mathsf Y}_{N,c}^{a,b}}
\right]
-
\mathbb E
\left[
\e^{-z\mathfrak e_c(a)}
\right]
\right|
+
\left|
\mathbb E
\left[
\e^{-z\mathfrak e_c(a)}
\right]
-
\mathbb E
\left[
\e^{-z\mathfrak e(a)}
\right]
\right|
+
\left|
\mathbb E_N^{(a,b)}
\left[
\e^{-z\mathsf Y_N^{a,b}}
\right]
-
\mathbb E_N^{(a,b)}
\left[
\e^{-z\widetilde{\mathsf Y}_{N,c}^{a,b}}
\right]
\right|.
\label{tri}
\end{align}
For the last term on the right-hand side of
\eqref{tri}, using
\[
\mathsf Y_N^{a,b}
=
\widetilde{\mathsf Y}_{N,c}^{a,b}
+
E_{N,c}^{a,b},
\]
together with non-negativity and
\[
|\e^{-u}-1|
\leq u,
\qquad
u\geq0,
\]
we obtain
\begin{align*}
\left|
\mathbb E_N^{(a,b)}
\left[
\e^{-z\mathsf Y_N^{a,b}}
-
\e^{-z\widetilde{\mathsf Y}_{N,c}^{a,b}}
\right]
\right|
=
\left|
\mathbb E_N^{(a,b)}
\left[
\e^{-z\widetilde{\mathsf Y}_{N,c}^{a,b}}
\left(
\e^{-zE_{N,c}^{a,b}}-1
\right)
\right]
\right|
\leq
z\,
\mathbb E_N^{(a,b)}
\left[
E_{N,c}^{a,b}
\right].
\end{align*}
The same argument for the Bessel point process gives
\[
\left|
\mathbb E
\left[
\e^{-z\mathfrak e_c(a)}
\right]
-
\mathbb E
\left[
\e^{-z\mathfrak e(a)}
\right]
\right|
\leq
z\,
\mathbb E
\left[
E_c^{(\mathrm{Bes})}
\right].
\]
Hence \eqref{tri} yields
\begin{align}
\left|
\mathbb E_N^{(a,b)}
\left[
\e^{-z\mathsf Y_N^{a,b}}
\right]
-
\mathbb E
\left[
\e^{-z\mathfrak e(a)}
\right]
\right|
\leq
\left|
\mathbb E_N^{(a,b)}
\left[
\e^{-z\widetilde{\mathsf Y}_{N,c}^{a,b}}
\right]
-
\mathbb E
\left[
\e^{-z\mathfrak e_c(a)}
\right]
\right|+
z\,
\mathbb E_N^{(a,b)}
\left[
E_{N,c}^{a,b}
\right]
+
z\,
\mathbb E
\left[
E_c^{(\mathrm{Bes})}
\right].
\label{1stmom}
\end{align}

We first consider the first term on the right-hand side of
\eqref{1stmom}. Recall that, for every fixed
\(\ell\in\mathbb N\), the standard hard-edge convergence for the
Jacobi ensemble gives
\begin{equation}\label{eq:finite-hard-edge-convergence}
\left(
N^2x_1^{(N)},\ldots,N^2x_\ell^{(N)}
\right)
\overset{\mathrm d}{\longrightarrow}
\left(
\mathfrak B_a^1,\ldots,\mathfrak B_a^\ell
\right),
\qquad N\to\infty.
\end{equation}
Fix \(0<\varepsilon<c\), and define
\[
\widetilde{\mathsf Y}_{N,c,\varepsilon}^{a,b}
:=
\frac{1}{N^2}
\sum_{j=1}^{N}
\frac{1}{x_j^{(N)}}
\mathbf 1_{\{
\varepsilon\leq N^2x_j^{(N)}<c
\}}
\]
and
\[
\mathfrak e_{c,\varepsilon}(a)
:=
\sum_{\varepsilon\leq\mathfrak B_a^i<c}
\frac{1}{\mathfrak B_a^i}.
\]
For \(L\in\mathbb N\) and \(N\geq L\), let
\[
\widetilde{\mathsf Y}_{N,c,\varepsilon}^{a,b;(L)}
:=
\sum_{j=1}^{L}
\frac{1}{N^2x_j^{(N)}}
\mathbf 1_{\{
\varepsilon\leq N^2x_j^{(N)}<c
\}}
\]
and
\[
\mathfrak e_{c,\varepsilon}^{(L)}(a)
:=
\sum_{j=1}^{L}
\frac{1}{\mathfrak B_a^j}
\mathbf 1_{\{
\varepsilon\leq\mathfrak B_a^j<c
\}}.
\]
Since the intensity measure of the Bessel point process is
absolutely continuous with respect to Lebesgue measure, the
limiting process almost surely has no point at either of the
fixed locations \(\varepsilon\) and \(c\). Hence
\eqref{eq:finite-hard-edge-convergence} and the continuous mapping
theorem give, for every fixed \(L\),
\[
\widetilde{\mathsf Y}_{N,c,\varepsilon}^{a,b;(L)}
\overset{\mathrm d}{\longrightarrow}
\mathfrak e_{c,\varepsilon}^{(L)}(a),
\qquad N\to\infty.
\]
Since \(y\mapsto\e^{-zy}\) is bounded and continuous on
\([0,\infty)\), it follows that
\begin{equation}\label{eq:finite-L-convergence}
\lim_{N\to\infty}
\mathbb E_N^{(a,b)}
\left[
\e^{-z\widetilde{\mathsf Y}_{N,c,\varepsilon}^{a,b;(L)}}
\right]
=
\mathbb E
\left[
\e^{-z\mathfrak e_{c,\varepsilon}^{(L)}(a)}
\right].
\end{equation}
We now remove the truncation in the number of particles. Let
\[
\mathcal N_{N,c}
:=
\sum_{j=1}^{N}
\mathbf 1_{\{N^2x_j^{(N)}<c\}}.
\]
For fixed \(c\geq1\), and all sufficiently large \(N\), the
kernel bounds in Lemma~\ref{lem:kernel-tail-estimates} give
\begin{align*}
\mathbb E_N^{(a,b)}
\left[
\mathcal N_{N,c}
\right]
&=
\int_0^{c/N^2}K_N(x,x)\,\d x\\
&=
\int_0^{c_0/N^2}K_N(x,x)\,\d x
+
\int_{c_0/N^2}^{c/N^2}K_N(x,x)\,\d x\\
&\leq
C
\int_0^{c_0}u^a\,\d u
+
C
\int_{c_0}^{c}u^{-1/2}\,\d u\leq
C_{a,b,c}.
\end{align*}
Here, in both integrals, we have used the change of variables
\(u=N^2x\). Consequently, Markov's inequality gives
\begin{equation}\label{eq:number-tail-finiteN}
\limsup_{N\to\infty}
\mathbb E_N^{(a,b)}
\left[
\mathbf 1_{\{\mathcal N_{N,c}>L\}}
\right]
\leq
\frac{C_{a,b,c}}{L}.
\end{equation}
Whenever \(\mathcal N_{N,c}\leq L\), we have
\[
\widetilde{\mathsf Y}_{N,c,\varepsilon}^{a,b}
=
\widetilde{\mathsf Y}_{N,c,\varepsilon}^{a,b;(L)}.
\]
Since the corresponding exponentials take values in \([0,1]\),
\begin{equation}\label{eq:finite-L-error}
\limsup_{N\to\infty}
\left|
\mathbb E_N^{(a,b)}
\left[
\e^{-z\widetilde{\mathsf Y}_{N,c,\varepsilon}^{a,b}}
\right]
-
\mathbb E_N^{(a,b)}
\left[
\e^{-z\widetilde{\mathsf Y}_{N,c,\varepsilon}^{a,b;(L)}}
\right]
\right|
\leq
\frac{C_{a,b,c}}{L}.
\end{equation}
Similarly, let
\[
\mathcal N_c^{(\mathrm{Bes})}
:=
\sum_{i\geq1}
\mathbf 1_{\{\mathfrak B_a^i<c\}}.
\]
By \eqref{eq:Bessel-kernel-bounds},
\begin{align*}
\mathbb E
\left[
\mathcal N_c^{(\mathrm{Bes})}
\right]
&=
\int_0^c
\mathsf K_a^{\mathrm{Bes}}(u,u)\,\d u\\
&\leq
C_a\int_0^1u^a\,\d u
+
C_a\int_1^c u^{-1/2}\,\d u
<\infty.
\end{align*}
Whenever
\[
\mathcal N_c^{(\mathrm{Bes})}\leq L,
\]
the ordered nature of the Bessel points implies that
\[
\mathfrak e_{c,\varepsilon}(a)
=
\mathfrak e_{c,\varepsilon}^{(L)}(a).
\]
Since the corresponding exponentials take values in \([0,1]\),
Markov's inequality gives
\begin{align}
&
\left|
\mathbb E
\left[
\e^{-z\mathfrak e_{c,\varepsilon}(a)}
\right]
-
\mathbb E
\left[
\e^{-z\mathfrak e_{c,\varepsilon}^{(L)}(a)}
\right]
\right|
\nonumber\\
&\leq
\mathbb E
\left[
\mathbf 1_{\{
\mathcal N_c^{(\mathrm{Bes})}>L
\}}
\right]
\leq
\frac{
\mathbb E[
\mathcal N_c^{(\mathrm{Bes})}
]
}{L}
\leq
\frac{C_{a,c}}{L}.
\label{eq:Bessel-finite-L-error}
\end{align}
Combining
\eqref{eq:finite-L-convergence},
\eqref{eq:finite-L-error}, and
\eqref{eq:Bessel-finite-L-error}, and then letting
\(L\to\infty\), we obtain
\begin{equation}\label{eq:truncated-eps-convergence}
\lim_{N\to\infty}
\mathbb E_N^{(a,b)}
\left[
\e^{-z\widetilde{\mathsf Y}_{N,c,\varepsilon}^{a,b}}
\right]
=
\mathbb E
\left[
\e^{-z\mathfrak e_{c,\varepsilon}(a)}
\right].
\end{equation}
It remains to remove the lower cutoff \(\varepsilon\). Take
\[
0<\varepsilon\leq c_0.
\]
Since
\(\widetilde{\mathsf Y}_{N,c}^{a,b}\) and
\(\widetilde{\mathsf Y}_{N,c,\varepsilon}^{a,b}\)
coincide whenever there is no rescaled Jacobi eigenvalue in
\((0,\varepsilon)\), we have
\begin{align*}
&\left|
\mathbb E_N^{(a,b)}
\left[
\e^{-z\widetilde{\mathsf Y}_{N,c}^{a,b}}
\right]
-
\mathbb E_N^{(a,b)}
\left[
\e^{-z\widetilde{\mathsf Y}_{N,c,\varepsilon}^{a,b}}
\right]
\right|\\
&\leq
\mathbb E_N^{(a,b)}
\Bigg[
\mathbf 1_{
\Big\{
\sum_{j=1}^{N}
\mathbf 1_{\{N^2x_j^{(N)}<\varepsilon\}}
\geq1
\Big\}}
\Bigg]\\
&\leq
\int_0^{\varepsilon/N^2}
K_N(x,x)\,\d x\leq
C N^2
\int_0^{\varepsilon/N^2}
(N^2x)^a\,\d x\\
&
=
\mathcal O_{a,b}\!\left(\varepsilon^{a+1}\right),
\end{align*}
uniformly for all sufficiently large \(N\). Similarly,
\begin{align*}
&
\left|
\mathbb E
\left[
\e^{-z\mathfrak e_c(a)}
\right]
-
\mathbb E
\left[
\e^{-z\mathfrak e_{c,\varepsilon}(a)}
\right]
\right|\leq
\int_0^\varepsilon
\mathsf K_a^{\mathrm{Bes}}(u,u)\,\d u
=
\mathcal O_a\!\left(\varepsilon^{a+1}\right).
\end{align*}
Combining these estimates with
\eqref{eq:truncated-eps-convergence}, and then letting
\(\varepsilon\downarrow0\), gives, for every fixed \(c\geq1\),
\begin{equation}\label{eq:truncated-convergence}
\lim_{N\to\infty}
\left|
\mathbb E_N^{(a,b)}
\left[
\e^{-z\widetilde{\mathsf Y}_{N,c}^{a,b}}
\right]
-
\mathbb E
\left[
\e^{-z\mathfrak e_c(a)}
\right]
\right|
=
0.
\end{equation}
It remains to control the two tail terms in \eqref{1stmom}.
By the determinantal structure,
\[
\mathbb E_N^{(a,b)}
\left[
E_{N,c}^{a,b}
\right]
=
\int_{c/N^2}^{1}
\frac{1}{N^2x}
K_N(x,x)\,\d x.
\]
For fixed \(c\geq1\) and all sufficiently large \(N\), we have
\(c/N^2\leq1/4\). We therefore split the integral at
\(x=1/4\). Since \(c\geq1\geq c_0\), the intermediate kernel
bound in Lemma~\ref{lem:kernel-tail-estimates} gives
\begin{align*}
\int_{c/N^2}^{1/4}
\frac{1}{N^2x}
K_N(x,x)\,\d x
&\leq
C
\int_c^{N^2/4}
u^{-3/2}\,\d u\leq
C
\int_c^\infty
u^{-3/2}\,\d u,
\end{align*}
where we used \(u=N^2x\). For the remaining interval, using \(x^{-1}\leq4\) on
\([1/4,1]\) and
\[
\int_0^1K_N(x,x)\,\d x=N,
\]
we obtain
\[
\int_{1/4}^1
\frac{1}{N^2x}
K_N(x,x)\,\d x
\leq
\frac4{N^2}
\int_0^1K_N(x,x)\,\d x
=
\frac4N.
\]
Consequently,
\begin{equation}\label{tailbnd}
\mathbb E_N^{(a,b)}
\left[
E_{N,c}^{a,b}
\right]
\leq
C
\int_c^\infty
u^{-3/2}\,\d u
+
\frac4N.
\end{equation}
For the Bessel point process, Tonelli's theorem and
\eqref{eq:Bessel-kernel-bounds} give
\begin{align}
\mathbb E
\left[
E_c^{(\mathrm{Bes})}
\right]
&=
\int_c^\infty
\frac{1}{u}
\mathsf K_a^{\mathrm{Bes}}(u,u)\,\d u
\leq
C_a
\int_c^\infty
u^{-3/2}\,\d u.
\label{eq:Bessel-tail-bound}
\end{align}
Returning to \eqref{1stmom}, for fixed \(c\geq1\), we let
\(N\to\infty\). Using
\eqref{eq:truncated-convergence},
\eqref{tailbnd}, and
\eqref{eq:Bessel-tail-bound}, we obtain
\[
\limsup_{N\to\infty}
\left|
\mathbb E_N^{(a,b)}
\left[
\e^{-z\mathsf Y_N^{a,b}}
\right]
-
\mathbb E
\left[
\e^{-z\mathfrak e(a)}
\right]
\right|
\leq
C_{a,b}z
\int_c^\infty
u^{-3/2}\,\d u.
\]
Letting \(c\to\infty\), the right-hand side tends to zero. Hence,
for every \(z\geq0\),
\[
\lim_{N\to\infty}
\mathbb E_N^{(a,b)}
\left[
\e^{-z\mathsf Y_N^{a,b}}
\right]
=
\mathbb E
\left[
\e^{-z\mathfrak e(a)}
\right].
\]
Since both \(\mathsf Y_N^{a,b}\) and \(\mathfrak e(a)\) are
non-negative random variables, convergence of their Laplace
transforms implies
\[
\mathsf Y_N^{a,b}
\overset{\mathrm d}{\longrightarrow}
\mathfrak e(a).
\]
This completes the proof.
\end{proof}

\section{Proofs of Proposition \ref{representation}, Theorem \ref{limitcase2} and Theorem \ref{maintheorem2}}\label{proofs}

Recall that $\sigma_{N}(z)$ is defined as in (\ref{pathforrealz}), that is 
\begin{equation}\label{solutionandY}
\sigma_N(z)
=
z\frac{\mathrm{d}}{\mathrm{d}z}
\log \mathbb{E}_{N}^{(a,b)}\left[\mathrm{e}^{-z\sum_{j=1}^{N}\frac{1}{x_{j}^{(N)}}}\right],\ \ \ a, b>-1.
\end{equation}
Then by \cite[equation (124)]{chen2019riemann}, for $z>0$,
\begin{equation}\label{August1}
\frac{\d \sigma_{N}(z)}{\d z}
=
-(2N+a+b)\,Y_{12}(0;z)Y_{21}(0;z)-N,
\end{equation}
where \(Y_{ij}(w;z)\) denotes the \((i,j)\)-entry of the unique
\(2\times2\) matrix-valued solution \(Y(w;z)\) of the
Fokas--Its--Kitaev Riemann--Hilbert problem associated with the
orthogonal polynomials for the corresponding Pollaczek-Jacobi-type
weight $w^{(a,b)}(x;z)$ as in (\ref{defofw}); see \cite[Section 2.1]{chen2019riemann}.  Although
\cite{chen2019riemann} works under the stronger assumptions
\(a>0\) and \(b>0\), the derivation of the identity (\ref{August1}) remains
valid for \(a>-1\) and \(b>-1\) when \(z>0\), since the factor
\(\mathrm{e}^{-z/x}\) in  $w^{(a,b)}(x;z)$ ensures integrability at the origin.

We now combine the results of the preceding Riemann--Hilbert
analysis to obtain a representation for the Laplace transform
of the inverse power sum, as stated in
Proposition~\ref{representation}. 

\begin{proof}[Proof of Proposition \ref{representation}]
By \eqref{August1}, it suffices to study $Y(0;t)$. By \cite[$(127)$]{chen2019riemann}, for sufficiently large, but fixed $N\geq N_1$,
\begin{equation}\label{Ypositive}
\begin{aligned}
Y_{+}(w;t)
={}&
4^{-N\sigma_{3}}
\widetilde{R}(w,2N^2t)
E_{0}(w)
\Psi_{-}\!\left(
N^{2}f_{0}(w),2N^{2}t
\right)
\e^{\frac{\mathrm{i}\pi}{2}(a-1)\sigma_{3}}
\\
&\times
\begin{pmatrix}
1 & 0 \\
1 & 1
\end{pmatrix}
\left(
w^{a}(1-w)^{b}\e^{-\frac{t}{w}}
\right)^{-\frac{1}{2}\sigma_{3}},\ \ \ w\in\left(0,\frac{1}{4}\right)\subset\mathbb{R},
\end{aligned}
\end{equation}
and 
\begin{equation}\label{y0}
Y(0;t)=\lim_{w\downarrow 0}Y_{+}(w;t).
\end{equation}
Here, \(\widetilde{R}(z)=\widetilde{R}(z,x)\) denotes the unique solution of the
model Riemann--Hilbert problem \ref{modelproblem2}  and \(\Psi(\xi,x)\) denotes the unique solution of the
model Riemann--Hilbert problem \ref{modelproblem}, with
\[
\xi=N^{2}f_{0}(w),
\qquad
x=2N^{2}t,
\]
and the locally analytic $w\mapsto f_0(w)$ as in (\ref{defforf0}). We emphasize that $\xi<0$ for $w\in(0,\frac{1}{4})\subset\mathbb{R}$, see right below \cite[$(75)$]{chen2019riemann}. Moreover, \(\Psi_{-}(\xi,x)\) denotes the non-tangential
boundary value of \(\Psi(\cdot,x)\) on $\Sigma_2$, taken from the negative
side of its oriented jump contour. Note that
\begin{equation}\label{E0}
E_0(0)
=
2^{-(a+b)\sigma_3}
\frac{\mathbf{I}-\mathrm{i}\sigma_1}{\sqrt{2}}
(-1)^{N\sigma_3}
\begin{pmatrix}
0 & 1 \\
-1 & a+b
\end{pmatrix}
(2N)^{\frac{1}{2}\sigma_3}
\end{equation}
by \cite[(128)]{chen2019riemann}.
By RHP \ref{modelproblem} condition (d) and \eqref{mp2}, with $\xi=N^2f_0(w),x=2N^2t$,
\begin{align}
	\Psi_-(\xi,x)&\,\e^{\frac{\i\pi}{2}(a-1)\sigma_3}\begin{pmatrix}1 & 0\\ 1 & 1\end{pmatrix}\Big(w^a(1-w)^b\e^{-\frac{t}{w}}\Big)^{-\frac{1}{2}\sigma_3}\nonumber\\
	=&\,\widehat{\Psi}(\xi;x,a)\begin{pmatrix}1 & f_-(\xi;x,a)\\ 0 & 1\end{pmatrix}\e^{\frac{x}{\xi}\sigma_3}\xi_-^{\frac{a}{2}\sigma_3}\begin{pmatrix}1&0\\ \e^{\i\pi a} & 1\end{pmatrix}\e^{\frac{\i\pi}{2}(a-1)\sigma_3}\begin{pmatrix}1 & 0\\ 1 & 1\end{pmatrix}\Big(w^a(1-w)^b\e^{-\frac{t}{w}}\Big)^{-\frac{1}{2}\sigma_3}\nonumber\\
	=&\,\widehat{\Psi}(\xi;x,a)\begin{pmatrix}1 & f_-(\xi;x,a)\\ 0 & 1\end{pmatrix}\e^{\frac{x}{\xi}\sigma_3}\xi_-^{\frac{a}{2}\sigma_3}\e^{\frac{\i\pi}{2}(a-1)\sigma_3}\Big(w^a(1-w)^b\e^{-\frac{t}{w}}\Big)^{-\frac{1}{2}\sigma_3}\nonumber\\
	=&\,\Big(Q(x)+\mathcal{O}_-(\xi)\Big)\e^{\frac{x}{\xi}\sigma_3}\xi_-^{\frac{a}{2}\sigma_3}\e^{\frac{\i\pi}{2}(a-1)\sigma_3}\Big(w^a(1-w)^b\e^{-\frac{t}{w}}\Big)^{-\frac{1}{2}\sigma_3},\ \ \ \ \xi\uparrow 0,\label{Ksinegative}
\end{align}
%
%
%
%
%
%
%
where, see \cite[$(4.4)$]{xu2015painleve},
\begin{equation}\label{def:Q}
Q(x)
=
\sqrt{\frac{1+q'(x)}{2}}
\begin{pmatrix}
1
&
\dfrac{\mathrm{i}v'(x)}{1+q'(x)}
\\[8pt]
\dfrac{\mathrm{i}(1-q'(x))}{v'(x)}
&
1
\end{pmatrix}
\chi^{\sigma_3}
\end{equation}
with $\chi$ being an arbitrary nonzero constant, and $v(x)$, $q(x)$ as in
\eqref{36}, \eqref{eq:def-q}, respectively.
Substituting \eqref{Ksinegative} into \eqref{Ypositive} and taking the limit $w\downarrow 0$ (so $\xi\uparrow 0$), while using
\[
f_{0}(w)
=
-4w\left(1+\frac{w}{3}+\mathcal{O}(w^2)\right),
\qquad w\to 0,
\]
we obtain
\begin{equation}\label{finalY0}
Y(0;t)
=
4^{-N\sigma_3}
\widetilde{R}(0,x)E_0(0)Q(x)d_N(t)^{\sigma_3};\ \ \ \ d_N(t)
=
-\i(2N)^a\e^{\frac{t}{6}},\ \ \ x=2N^2t.
\end{equation}
Next, we parametrize, compare also \eqref{eq:Rtilde-Cauchy-at-zero},
\begin{equation}\label{theassumptionofR}
\widetilde{R}(0,x)-\mathbf{I}
=
\begin{pmatrix}
R_{11}(x)
&
R_{12}(x)
\\[8pt]
R_{21}(x)
&
R_{22}(x)
\end{pmatrix}.
\end{equation}
Then by \eqref{finalY0} and \eqref{E0},
\begin{equation}\label{cool}
Y_{12}(0;t)Y_{21}(0;t)=\frac{N}{2} v'(x)
+\widetilde{f}_{N,a,b}(x),\ \ \ \ x=2N^2t,
\end{equation}
where 
\begin{align}
\widetilde f_{N,a,b}(x)
={}&
\mathcal O_{a,b}\left(
1+|q'(x)|+\frac{|\mathfrak t'(x)|}{N}
\right)
\nonumber\\
&+
\mathcal O_{a,b}\left(
\bigl(\|\widetilde{R}(0,x)-\mathbf{I}\|+\|\widetilde{R}(0,x)-\mathbf{I}\|^2\bigr)
\left(
N|v'(x)|
+|q'(x)|
+\frac{|\mathfrak t'(x)|}{N}
+1
\right)
\right),
\label{eq:f-rough-estimate}
\end{align}
and $\mathfrak{t}(x)$ as in \eqref{eq:def-t-prime}. Moreover, the functions
$v'(x)$, $q'(x)$, and $\mathfrak t'(x)$ are smooth on
$(0,\infty)$. By Lemma~\ref{lem:derivative-Rtilde-at-zero}, there exists
\(N_2=N_2(a)\in\mathbb N\) such that, for every \(N\geq N_2\),
the \(N\)-dependent map
$
x\longmapsto \widetilde R(0;x)
$
is \(C^\infty\) on \((0,\infty)\). Therefore, from the above
representation of $\widetilde f_{N,a,b}(x)$, we conclude that
$
\widetilde f_{N,a,b}(x)\in C^\infty((0,\infty))$ whenever $N\geq N_{2}$. Consequently, by Lemma \ref{asympoticsforv(x)forsmallx} and Proposition \ref{prop:q-t-derivative-estimates},
together with another application of Leibniz's rule and
Remark~\ref{rem:product-estimate-rho2}, we obtain, for every
fixed $j\in\mathbb{Z}_{\geq 0}$,
\begin{equation}\label{eq:Rtilde-at-zero-derivative2}
\frac{\d^j}{\d x^j}\widetilde {f}_{N,a,b}(x)
=
\mathcal O_{a,b,j}
\left(
\frac{\rho_{a,j+1}(x)}{N}
\left(
1+\frac{1}{N}
\right)
\right)+\mathcal O_{a,b,j}
\left(\rho_{a,j+1}(x)\right),
\end{equation}
uniformly for \(0<x\leq x_0=x_{0}(a,j)\) and \(N\geq N_0\), with
\(N_0=\max\{N_1,N_2\}\).
By \eqref{August1},
\begin{equation*}
\sigma_{N}(z)=-\int_{0}^{z}(2N+a+b)Y_{12}(0;t)Y_{21}(0;t)\d t-Nz,
\end{equation*}
So
\begin{align}\label{thelimitofsigmaNz}
\sigma_{N}\bigg(\frac{z}{2N^2}\bigg)=&-\int_{0}^{\frac{z}{2N^2}}(2N+a+b)Y_{12}(0;t)Y_{21}(0;t)\d t-\frac{z}{2N}\nonumber\\
\stackrel{\eqref{cool}}{=}&\,-\frac{1}{2}\left(v(z)-\frac{1-4a^2}{8}\right)\left(1+\mathcal{O}_{a,b}\big(N^{-1}\big)\right)+\int_{0}^{z}f_{N,a,b}(x)\d x,
\end{align}
where
\[
f_{N,a,b}(x):=-\frac{1}{N}\left(1+\frac{a+b}{2N}\right)\widetilde {f}_{N,a,b}(x)-\frac{1}{2N}-\frac{a+b}{4N}v'(x),\ \ \ x>0.
\]
By (\ref{eq:Rtilde-at-zero-derivative2}) and 
\begin{equation}
\mathbb{E}_{N}^{(a,b)}\left[\mathrm{e}^{-\frac{z}{2N^2}\sum_{j=1}^{N}\frac{1}{x_{j}^{(N)}}}\right]
=\exp\left(\int_{0}^{z}\sigma_{N}\bigg(\frac{t}{2N^2}\bigg)\frac{\d t}{t}\right),
\end{equation} 
we now obtain the claim in the current proposition.
\end{proof}

\begin{rmk}\label{rem:average-integral-estimate32}
Let \(f_{N,a,b}(x)\) be as defined in Proposition~\ref{representation}. Let
\[
g_{N,a,b}(x)
:=
\frac{1}{x}\int_0^x f_{N,a,b}(t)\,\d t=\int_0^1f_{N,a,b}(xy)\,\d y,\ \ \ x>0.
\]
Hence, for every $k\in\mathbb{N}$,
\[
g_{N,a,b}^{(k)}(x)
=
\int_0^1 y^k f_{N,a,b}^{(k)}(xy)\,\d y.
\]
Since by definition of \(\rho_{a,k}\), 
\[
\int_0^1
y^k\rho_{a,k}(xy)\,\d y
=
\mathcal O_{a,k}
\left(
\rho_{a,k}(x)
\right),\ \ \ \ a>-1,\ k\in\mathbb{Z}_{\geq 0},
\]
uniformly for $0<x\leq 1$, estimate \eqref{eq:Rtilde-at-zero-derivative4} yields
\begin{equation}\label{deriforgNab}
\frac{\d^k}{\d x^k}
\left(
\frac{1}{x}\int_0^x f_{N,a,b}(t)\,\d t
\right)
=g^{(k)}_{N,a,b}(x)=
\mathcal O_{a,b,k}
\left(
\frac{\rho_{a,k+1}(x)}{N}
\left(
1+\frac{1}{N}
\right)
\right),
\end{equation}
uniformly for $0<x\leq x_0$ and $N\geq N_0$, where $x_0$ and $N_0$
are chosen as in Proposition~\ref{representation}. Additionally,
observe that
\begin{equation*}
\frac{1-4a^2-8v(x)}{16x}=-\frac{1}{2x}\int_{0}^{x}v'(t)\d t
\end{equation*}
Let $v(x)$ be as in (\ref{36}).
By Lemma \ref{asympoticsforv(x)forsmallx} and  a similar argument to that of \eqref{deriforgNab}, 
\begin{equation}\label{deriforvdividedbyx}
\frac{\d^k}{\d x^k}\frac{1-4a^2-8v(x)}{16x}=\mathcal{O}_{a,k}\left(\rho_{a,k+1}(x)\right).
\end{equation}
uniformly for \(0<x\leq x_0\), where \(x_0\) is chosen as in
Lemma~\ref{asympoticsforv(x)forsmallx}.
\end{rmk}
The following result relates the moments of a non-negative random variable to its Laplace transform.
\begin{lem}\label{expressionformoments}
Let $h\in\mathbb{C}$, and let $Y$ be a positive random variable. Let $n\in \mathbb{Z}_{\geq 0}$ be an integer such that $n>\textnormal{Re}(h)$, and set $\alpha=n-h$. Then
\begin{equation}\label{integralrep}
\mathbb{E}\!\left[Y^h\right]
=
\frac{(-1)^n}{\Gamma(\alpha)}
\int_{0}^{\infty}
z^{\alpha-1}
\frac{\mathrm{d}^{n}}{\mathrm{d}z^{n}}
\mathbb{E}\!\left[\mathrm{e}^{-zY}\right]
\,\mathrm{d}z<\infty,
\end{equation}
provided that
\begin{equation*}
\int_{0}^{\infty}
z^{\textnormal{Re}(\alpha)-1}
\left|
\frac{\mathrm{d}^{n}}{\mathrm{d}z^{n}}
\mathbb{E}\!\left[\mathrm{e}^{-zY}\right]
\right|
\,\mathrm{d}z
<\infty.
\end{equation*}
\end{lem}
\begin{proof}
For any $y>0,\beta\in\mathbb{C}:\textnormal{Re}(\beta)>0$ and any $m\in\mathbb{Z}_{\geq 0}$, we have the identity
\begin{equation}
y^{m-\beta}
=
\frac{(-1)^{m}}{\Gamma(\beta)}
\int_{0}^{\infty}
z^{\beta-1}
\frac{\mathrm{d}^{m}}{\mathrm{d}z^{m}}
\mathrm{e}^{-zy}
\,\mathrm{d}z,
\label{exact2}
\end{equation}
which follows by directly computing the derivatives and then evaluating the integral on the right-hand side of \eqref{exact2}. Let $\mathbb{P}$ denote the probability distribution of $Y$ on $[0,\infty)$. Then
\begin{equation}
\mathbb{E}\big[Y^{h}\big]
=
\int_{0}^{\infty}
y^{h}\,\mathrm{d}\mathbb{P}(y).
\label{leb2}
\end{equation}
Substituting \eqref{exact2} with $\beta=n-h$ and $m=n$ therein into \eqref{leb2}, so that $\textnormal{Re}(\beta)>0$ by choice of $n$, we obtain
\begin{equation}
\mathbb{E}\big[Y^{h}\big]
=
\int_{0}^{\infty}
\left[\int_{0}^{\infty}
f(y,z)\,
\mathrm{d}z\right]\,\mathrm{d}\mathbb{P}(y),
\label{expression}
\end{equation}
where
\begin{equation}
	(y,z)\mapsto f(y,z)
=
\frac{(-1)^{n}}{\Gamma(n-h)}
z^{n-h-1}
\frac{\mathrm{d}^{n}}{\mathrm{d}z^{n}}
\mathrm{e}^{-zy}
=
\frac{1}{\Gamma(n-h)}
z^{n-h-1}
y^{n}
\mathrm{e}^{-zy}
\end{equation}
is measurable on $(0,\infty)\times(0,\infty)$. Moreover, for $y,z>0$,
\begin{equation*}
|f(y,z)|
\le
\frac{1}{|\Gamma(n-h)|}
z^{n-\mathrm{Re}(h)-1}
y^{n}
\mathrm{e}^{-zy}.
\end{equation*}
Applying Tonelli's theorem, we obtain
\begin{align}
\label{Tonellitheorem}
&\int_{0}^{\infty}
\left[\int_{0}^{\infty}
z^{n-\mathrm{Re}(h)-1}
y^{n}
\mathrm{e}^{-zy}
\,\mathrm{d}z\right]\,\mathrm{d}\mathbb{P}(y)
=
\int_{0}^{\infty}\left[
\int_{0}^{\infty}
z^{n-\mathrm{Re}(h)-1}
y^{n}
\mathrm{e}^{-zy}
\,\mathrm{d}\mathbb{P}(y)\right]\,\mathrm{d}z
\nonumber\\
=&\,\,
(-1)^{n}
\int_{0}^{\infty}
z^{n-\mathrm{Re}(h)-1}
\mathbb{E}
\!\left[
\frac{\mathrm{d}^{n}}{\mathrm{d}z^{n}}
\mathrm{e}^{-zY}
\right]
\mathrm{d}z
=
(-1)^{n}
\int_{0}^{\infty}
z^{n-\mathrm{Re}(h)-1}
\frac{\mathrm{d}^{n}}{\mathrm{d}z^{n}}
\mathbb{E}
\!\left[
\mathrm{e}^{-zY}
\right]
\mathrm{d}z.
\end{align}
By the assumption of the Lemma and \eqref{Tonellitheorem},
\begin{equation*}
\int_{0}^{\infty}
\int_{0}^{\infty}
|f(y,z)|
\,\mathrm{d}z\,\mathrm{d}\mathbb{P}(y)
<\infty,
\end{equation*}
and hence Fubini's theorem implies that the order of integration in \eqref{expression} can be interchanged, i.e. \eqref{integralrep} follows. This completes the proof.
\end{proof}
Now we are ready to prove
Theorem \ref{limitcase2}.

\begin{proof}[Proof of Theorem \ref{limitcase2}]
By Lemma \ref{explicitformulaforfiniteN}, for
$s>-\frac{3}{2}$ and $\textnormal{Re}(h)<s+\frac{3}{2}$,
\begin{align*}
\frac{F_{N}^{\mathrm{USp}}(s,h)}
     {N^{\frac{s(s+1)}{2}+2h}}
&=
\frac{F_N^{\mathrm{USp}}(s,0)}
{N^{\frac{s(s+1)}{2}}}
\mathbb{E}_{N}^{(a,b)}
\left[\left(\mathsf{X}_N^{a,b}\right)^h\right],\ \ \ \ \ a=s+\frac{1}{2},\ b=\frac{1}{2},
\end{align*}
whereas, for $s>-\frac{1}{2}$ and $\operatorname{Re}(h)<s+\frac{1}{2}$,
\begin{align*}
\frac{F_{N}^{\mathrm{SO}}(s,h)}
     {N^{\frac{s(s-1)}{2}+2h}}
&=
\frac{F_N^{\mathrm{SO}}(s,0)}
{N^{\frac{s(s-1)}{2}}}
\mathbb{E}_{N}^{(a,b)}
\left[\left(\mathsf{X}_N^{a,b}\right)^h\right],\ \ \ \ \ a=s-\frac{1}{2},\ b=-\frac{1}{2},
\end{align*}
with $\mathsf{X}_N^{a,b}$ defined in \eqref{defofXN}. By the explicit asymptotic formul\ae\,in \cite{keating2000random}, as $N\rightarrow\infty$,
\begin{align*}
F_N^{\mathrm{USp}}(s,0)=\int_{\mathrm{USp}(2N)}
|\varphi_{\mathbf A}(0)|^{s}
\,\mathrm{d}\mu_{\mathrm{Haar}}(\mathbf A)
&\sim
N^{\frac{s(s+1)}{2}}
\frac{2^{\frac{s^2}{2}}G(1+s)\sqrt{\Gamma(1+s)}}
{\sqrt{G(1+2s)\Gamma(1+2s)}},\\
F_N^{\mathrm{SO}}(s,0)=\int_{\mathrm{SO}(2N)}
|\varphi_{\mathbf A}(0)|^{s}
\,\mathrm{d}\mu_{\mathrm{Haar}}(\mathbf A)
&\sim
N^{\frac{s(s-1)}{2}}
\frac{2^{\frac{s^2}{2}}G(1+s)\sqrt{\Gamma(1+2s)}}
{\sqrt{G(1+2s)\Gamma(1+s)}}.
\end{align*}
Thus, it remains to prove that, for $a>-1$, $b>-1$, and
$\textnormal{Re}(h)<a+1$,
\begin{equation}\label{goal1}
\mathbb{E}_{N}^{(a,b)}
\left[\left(\mathsf{X}_N^{a,b}\right)^h\right]
\longrightarrow
\mathbb{E}
\left[\left(1+\frac{\mathfrak{e}(a)}{2}\right)^h\right].
\end{equation}
From Proposition~\ref{lem:YN-convergence-all-a} and \eqref{defofXN}, we have
\begin{equation}\label{08041}
\mathsf X_N^{a,b}
\overset{\mathrm d}{\longrightarrow}
1+\frac{\mathfrak e(a)}{2}.
\end{equation}
Hence, by \(\mathsf X_N^{a,b}>0\) a.s. and the continuous mapping theorem,
\[
\left(\mathsf X_N^{a,b}\right)^h
\overset{\mathrm d}{\longrightarrow}
\left(1+\frac{\mathfrak e(a)}{2}\right)^h.
\]
Therefore, to prove \eqref{goal1}, it remains to establish the
uniform integrability of
$
\left\{
\left(\mathsf X_N^{a,b}\right)^h
\right\}_{N\geq1}.
$
For this purpose, it is enough to show that, for every real
\(r<a+1\),
\begin{equation}\label{08043}
\sup_{N\geq1}
\mathbb E_N^{(a,b)}
\left[
\left(\mathsf X_N^{a,b}\right)^r
\right]
<\infty.
\end{equation}
Indeed, for any fixed \(h\in\mathbb C\) with
\(\operatorname{Re}(h)<a+1\), one can choose \(p>1\) such that
\[
p\textnormal{Re}(h)<a+1.
\]
Then \eqref{08043} gives
\[
\sup_{N\geq1}
\mathbb E_N^{(a,b)}
\left[
\left|
\left(\mathsf X_N^{a,b}\right)^h
\right|^p
\right]
=
\sup_{N\geq1}
\mathbb E_N^{(a,b)}
\left[
\left(\mathsf X_N^{a,b}\right)^{p\operatorname{Re}(h)}
\right]
<\infty,
\]
which implies uniform integrability. It remains to prove \eqref{08043}. Note that \(\mathsf X_N^{a,b}\) is non-negative almost surely. So by Lemma \ref{expressionformoments}, it suffices to prove that
\begin{equation}\label{0927formula1}
\sup_{N\geq 1}\int_{0}^{\infty}
z^{n-r-1}
\left|\frac{\mathrm{d}^{n}}{\mathrm{d}z^{n}}
\mathbb{E}_N^{(a,b)}\!\left[\mathrm{e}^{-z\mathsf{X}_N^{a,b}}\right]
\right|\,\mathrm{d}z<\infty,
\end{equation}
with 
\[
n=
\begin{cases}
0, & r<0,\\[4pt]
\min\{m\in \mathbb{N}:m>r\}, & r\geq0.
\end{cases}
\]
Note that the differentiation order $n$ in \eqref{0927formula1} is bounded by $a+2$.
Hence, we may choose
\begin{equation*}
x_0=x_0(a)
:=
\min\bigl\{
x_0(a,j):j=0,\ldots,\lfloor a+2\rfloor
\bigr\}>0,
\end{equation*}
where $x_0(a,j)$ is given in Proposition \ref{representation}, right above \eqref{eq:Rtilde-at-zero-derivative4}. When $s\in \mathbb{Z}_{\geq 0}-1/2$, we set $a=s+1/2$ (unitary symplectic case) or $a=s-1/2$ (orthogonal case). Then $a\in \mathbb{Z}_{\geq 0}$. By Remark \ref{rem:average-integral-estimate32} and Proposition \ref{representation}, we obtain
\begin{align}\label{0927}
\frac{\mathrm{d}^{n}}{\mathrm{d}z^{n}}
\mathbb{E}_N^{(a,b)}\!\left[\mathrm{e}^{-z\mathsf{X}_N^{a,b}}\right]
=
\mathcal{O}_{a}(1)
+
\mathcal{O}_{a}\Big(z^{a+1-n}|\ln z|\Big).
\end{align}
Thus,
\begin{align}\label{0927formula2}
z^{n-r-1}
\left|
\frac{\mathrm{d}^{n}}{\mathrm{d}z^{n}}
\mathbb{E}_N^{(a,b)}
\left[
\mathrm{e}^{-z\mathsf{X}_N^{a,b}}
\right]
\right|
=
\mathcal{O}_{a}\!\left(z^{n-r-1}\right)
+
\mathcal{O}_{a}\!\left(z^{a-r}|\ln z|\right),
\end{align}
uniformly for \(0<z\leq x_0\) and \(N\geq N_0\), where
\(N_0=N_0(a,b)\).
When $s\notin \mathbb{Z}_{\geq 0}-1/2$, we set $a=s+1/2$ (unitary symplectic case) or $a=s-1/2$ (orthogonal case). Then $a\notin \mathbb{Z}_{\geq 0}$. Using Remark \ref{rem:average-integral-estimate32} and Proposition \ref{representation} again, we obtain
\begin{align}\label{0927formula21}
\frac{\mathrm{d}^{n}}{\mathrm{d}z^{n}}
\mathbb{E}_N^{(a,b)}\!\left[\mathrm{e}^{-z\mathsf{X}_N^{a,b}}\right]
=
\mathcal{O}_{a}(1)
+
\mathcal{O}_{a}\big(z^{a+1-n}\big),
\end{align}
and thus
\begin{align}\label{0927formula3}
z^{n-r-1}
\left|\frac{\mathrm{d}^{n}}{\mathrm{d}z^{n}}
\mathbb{E}_N^{(a,b)}\!\left[\mathrm{e}^{-z\mathsf{X}_N^{a,b}}\right]
\right|
=
\mathcal{O}_{a}(z^{n-r-1})
+
\mathcal{O}_{a}\Big(z^{a-r}\Big),
\end{align}
holds uniformly for $N\geq N_{0}$ and $0<z\leq x_{0}$. Note that \(n>r\), \(r<a+1\) and $z\mapsto z^{n-r-1}
\left|\frac{\mathrm{d}^{n}}{\mathrm{d}z^{n}}
\mathbb{E}_N^{(a,b)}\!\left[
\mathrm{e}^{-z\mathsf{X}_N^{a,b}}
\right]
\right|$ is continuous on $(0,\infty)$. Hence, by \eqref{0927formula2} and \eqref{0927formula3}, 

\begin{equation}\label{0927formula4}
\sup_{N\geq N_{0}}
\int_{0}^{x_{0}}
z^{n-r-1}
\left|\frac{\mathrm{d}^{n}}{\mathrm{d}z^{n}}
\mathbb{E}_N^{(a,b)}\!\left[
\mathrm{e}^{-z\mathsf{X}_N^{a,b}}
\right]
\right|\,\mathrm{d}z
<\infty,
\end{equation}
Since \(\mathsf X_N^{a,b}>1\) a.s., it is not hard to see that
there exists a constant \(C>0\), depending at most on \(x_0\) and
\(n\) (and hence on \(a\) and \(r\)), such that, for \(z\geq x_0\),
\begin{equation}\label{0927formula5}
z^{n-r-1}
\left|
\frac{\mathrm d^n}{\mathrm dz^n}
\mathbb E_N^{(a,b)}
\left[
\mathrm e^{-z\mathsf X_N^{a,b}}
\right]
\right|
\leq
C z^{n-r-1}\mathrm e^{-z/2}.
\end{equation}
So
\begin{equation}\label{0927formula6}
\sup_{N\geq N_{0}}
\int_{x_{0}}^{\infty}
z^{n-r-1}
\left|\frac{\mathrm{d}^{n}}{\mathrm{d}z^{n}}
\mathbb{E}_N^{(a,b)}\!\left[
\mathrm{e}^{-z\mathsf{X}_N^{a,b}}
\right]
\right|\,\mathrm{d}z
<\infty.
\end{equation}
Then, by (\ref{0927formula4}) and (\ref{0927formula6}), 
\begin{equation}\label{0927formula7}
\sup_{N\geq N_{0}}
\int_{0}^{\infty}
z^{n-r-1}
\left|\frac{\mathrm{d}^{n}}{\mathrm{d}z^{n}}
\mathbb{E}_N^{(a,b)}\!\left[
\mathrm{e}^{-z\mathsf{X}_N^{a,b}}
\right]
\right|\,\mathrm{d}z
<\infty.
\end{equation}
For each fixed \(1\leq N<N_0\), the integral in
\eqref{0927formula1} is finite. Hence, since there are only finitely
many such values of \(N\), the maximum over \(1\leq N<N_0\) is finite.
Together with \eqref{0927formula7}, this proves \eqref{0927formula1}.
Hence, the proof of the theorem is complete.
\end{proof}

\begin{rmk}\label{summarize}
From the above proof, a key ingredient is the uniform small-\(z\)
estimate for the Laplace transform \eqref{1005laplace}, obtained from
Remark~\ref{rem:average-integral-estimate32} and
Proposition~\ref{representation}. We summarize this estimate below.
Let \(a,b>-1\). For every fixed \(n\in\mathbb Z_{\geq0}\), as
\(z\to0^+\),
\begin{equation*}\label{0927combined}
\frac{\mathrm{d}^{n}}{\mathrm{d}z^{n}}
\mathbb{E}_N^{(a,b)}
\left[
\mathrm{e}^{-z\mathsf{X}_N^{a,b}}
\right]
=
\begin{cases}
\mathcal{O}_{a,b,n}(1)
+\mathcal{O}_{a,b,n}\!\left(z^{a+1-n}\right),
& a\notin\mathbb Z_{\geq0},\\[4pt]
\mathcal{O}_{a,b,n}(1)
+\mathcal{O}_{a,b,n}\!\left(
z^{a+1-n}|\ln z|
\right),
& a\in\mathbb Z_{\geq0}.
\end{cases}
\end{equation*}
The estimates in \eqref{0927combined} hold uniformly for all
sufficiently large \(N\).
\end{rmk}

Using Theorem~\ref{limitcase2}, we are now ready to prove
Theorem~\ref{maintheorem2}.
 \begin{proof}[Proof of Theorem \ref{maintheorem2} ]
We set
$a=s+\frac12$in the unitary symplectic case,
and $a=s-\frac12$ in the orthogonal case.
For \(a>-1\) and \(b>-1\), by \eqref{goal1},
\begin{equation}\label{goal0928}
\mathbb{E}_{N}^{(a,b)}
\left[
\left(\mathsf{X}_N^{a,b}\right)^h
\right]
\longrightarrow
\mathbb{E}
\left[
\left(1+\frac{\mathfrak e(a)}{2}\right)^h
\right],
\end{equation}
for \(\operatorname{Re}(h)<a+1\). By \eqref{0927formula1} and Lemma~\ref{expressionformoments},
\begin{equation}\label{0928noon}
\mathbb{E}_{N}^{(a,b)}
\left[
(\mathsf{X}_N^{a,b})^h
\right]
=
\frac{(-1)^n}{\Gamma(\alpha)}
\int_0^\infty
z^{\alpha-1}
\frac{\mathrm d^n}{\mathrm d z^n}
\mathbb{E}_{N}^{(a,b)}
\left[
\e^{-z\mathsf{X}_N^{a,b}}
\right]
\,\mathrm dz,
\end{equation}
where
\[
n=
\begin{cases}
0, & \operatorname{Re}(h)<0,\\[4pt]
\min\{m\in\mathbb N:m>\operatorname{Re}(h)\},
& \operatorname{Re}(h)\geq0,
\end{cases}
\]
and
$
\alpha=n-h.
$
Recall that
\[
\mathsf{X}_N^{a,b}
=
1+\frac{1}{2N^2}
\sum_{j=1}^N\frac{1}{x_j^{(N)}}.
\]
Since
\[
\mathsf{X}_N^{a,b}
\overset{\mathrm d}{\longrightarrow}
1+\frac{\mathfrak e(a)}{2},
\]
for every fixed \(z>0\) and \(m\in\mathbb Z_{\geq0}\), the bounded
continuity of the function
\[
y\longmapsto y^m\e^{-zy},
\qquad y\in[0,\infty),
\]
implies that
\begin{equation}\label{smooth1001}
\frac{\mathrm d^m}{\mathrm dz^m}
\mathbb E_N^{(a,b)}
\left[
\e^{-z\mathsf{X}_N^{a,b}}
\right]
\longrightarrow
\frac{\mathrm d^m}{\mathrm dz^m}
\mathbb E
\left[
\e^{-z(1+\frac{\mathfrak e(a)}{2})}
\right],
\qquad N\to\infty.
\end{equation}
In particular,
\begin{equation}\label{1002}
z\frac{\mathrm d}{\mathrm dz}
\mathbb E_N^{(a,b)}
\left[
\e^{-z\mathsf{X}_N^{a,b}}
\right]
\longrightarrow
z\frac{\mathrm d}{\mathrm dz}
\mathbb E
\left[
\e^{-z(1+\frac{\mathfrak e(a)}{2})}
\right],
\qquad N\to\infty.
\end{equation}
Combining this with \eqref{solutionandY}, \eqref{defofsigma2} and
\eqref{thelimitofsigmaNz}, we obtain
\[
z\frac{\mathrm d}{\mathrm dz}
\log
\mathbb E
\left[
\e^{-z(1+\frac{\mathfrak e(a)}{2})}
\right]
=
\sigma(z;a),
\qquad z\in(0,\infty).
\]
Since \(\sigma(z;a)/z\) is integrable at \(z=0\) by Proposition \ref{prop:sigma-small-x}, and since the Laplace transform equals \(1\) at
\(z=0\), it follows that
\begin{equation}\label{exporep}
\mathbb E
\left[
\e^{-z(1+\frac{\mathfrak e(a)}{2})}
\right]
=
\exp\left[
\int_0^z
\frac{\sigma(t;a)}{t}\,\mathrm dt
\right].
\end{equation}
Hence, by \eqref{smooth1001} and \eqref{exporep}, for every \(a>-1\)
and \(n\in\mathbb Z_{\geq0}\), we have the pointwise convergence
\begin{equation}\label{0928formula21}
\lim_{N\to\infty}
\frac{\mathrm d^n}{\mathrm dz^n}
\mathbb E_N^{(a,b)}
\left[
\e^{-z\mathsf{X}_N^{a,b}}
\right]
=
\frac{\mathrm d^n}{\mathrm dz^n}
\exp\left[
\int_0^z
\frac{\sigma(t;a)}{t}\,\mathrm dt
\right],
\qquad z>0.
\end{equation}
By \eqref{0927formula2}, \eqref{0927formula3} and (\ref{0927formula5}), with \(r\) there
replaced by \(\operatorname{Re}(h)\), the absolute value of the integrand in \eqref{0928noon} is bounded by
an integrable function on \((0,\infty)\), uniformly for all
sufficiently large \(N\). Therefore, using
\eqref{0928formula21} and applying the dominated convergence theorem
to \eqref{0928noon}, we obtain
\[
\lim_{N\to\infty}
\mathbb{E}_{N}^{(a,b)}
\left[
(\mathsf{X}_N^{a,b})^h
\right]
=
\frac{(-1)^n}{\Gamma(\alpha)}
\int_0^\infty
z^{\alpha-1}
\frac{\mathrm d^n}{\mathrm d z^n}
\exp\left[
\int_0^z
\frac{\sigma(t;a)}{t}\,\mathrm dt
\right]
\,\mathrm dz.
\]
Comparing this identity with \eqref{goal0928} yields the Painlev\'{e} representation as claimed in  this theorem.

Moreover, using \eqref{1002}, \eqref{pathforrealz}, and
\eqref{defofsigma2}, making the change of variables
\[
z\mapsto \frac{t}{2N^2}
\]
in \eqref{equationforfiniteN}, and then subtracting \(t\) according to
the definition of \(\sigma(t;a)\) as in (\ref{defofsigma2}), we obtain, by letting \(N\to\infty\),
the \(\sigma\)-Painlev\'e \textnormal{III}$'$ equation for
\(\sigma(t;a)\), as stated in \eqref{e42}.
Finally, the claim (\ref{definitionofsigma}) follows from (\ref{defofsigma2}) and (\ref{36}). We have now completed the proof of the theorem.
\end{proof}

\bibliographystyle{plain}
\bibliography{main}

\noindent{\sc School of Mathematics, University of Edinburgh, James Clerk Maxwell Building, Peter Guthrie Tait Rd, Edinburgh EH9 3FD, U.K.}\newline
\href{mailto:theo.assiotis@ed.ac.uk}{\small theo.assiotis@ed.ac.uk}

\bigskip
\noindent
{\sc School of Mathematics, University of Bristol, 
Fry Building, Woodland Road, Bristol, BS8 1UG, United Kingdom.}\newline
\href{mailto:thomas.bothner@bristol.ac.uk}{\small
thomas.bothner@bristol.ac.uk}

\bigskip
\noindent
{\sc Mathematical Institute, University of Oxford, Andrew Wiles Building, Woodstock Road, Oxford, OX2 6GG, UK. }\newline
\href{keating@maths.ox.ac.uk}{\small keating@maths.ox.ac.uk}

\bigskip
\noindent
{\sc Department of Mathematics, University of Sussex, Brighton, BN1 9RH, UK. }\newline
\href{N.J.Simm@sussex.ac.uk}{\small 
N.J.Simm@sussex.ac.uk}

\bigskip
\noindent
{\sc Key Laboratory of Mathematics and Information Networks (Beijing University of Posts and Telecommunications), Ministry of Education,
Science Building, No. 1 Nanfeng Road, Beijing 102206, China. }\newline
\href{mailto:weif0831@gmail.com}{\small weif0831@gmail.com}

\end{document}